\documentclass{amsart}

\usepackage{a4wide}

\newcounter{saveenum}

\newcommand{\bibpaper}[1]{\textsl{#1}}
\newcommand{\bibbook}[1]{\textsl{#1}}
\newcommand{\bibinbook}[1]{\textsl{#1}}

\newcommand{\jkrestatable}[2]{\expandafter\def\csname jkrestat:#1\endcsname{#2}\ignorespaces}
\newcommand{\restateWithLPrefix}[2]{\begingroup\let\jkoriglabel\label\def\label##1{\jkoriglabel{#2:jkR:##1}}\csname jkrestat:#1\endcsname\endgroup}
\newcommand{\restateWithoutLPrefix}[1]{\begingroup\csname jkrestat:#1\endcsname\endgroup}

\newcommand{\HforHadam}{{\mathscr H}}
\newcommand{\diffop}{{\mathscr D}}
\newcommand{\spaces}{{\mathcal X}}
\newcommand{\tests}{{\mathcal G}}
\newcommand\denseY{Y_0}
\newcommand\densePhi{{\Phi_0}}

\newcommand{\tildePhi}{\widetilde\Phi}
\newcommand{\PhiOnSmallest}{\Phi_0}
\newcommand{\Zmedium}{Z}
\newcommand{\Flextf}{F}

\newcommand{\setS}{S}

\newcommand{\cantor}{{\mathfrak{c}}}
\newcommand{\omegaone}{\omega_1}

\newcommand{\vectorv}{e}
\newcommand{\setTexample}{T}
\newcommand{\spaceExample}{Y}
\newcommand{\spanExample}{X}
\newcommand{\Spanofset}[1]{\Span\ \{#1\}}

\newcommand{\indexONE}{1}
\newcommand{\indexTWO}{2}
\newcommand{\ten}{10}

\newcommand{\firstindex}{0}

\newcommand{\ddelta}{\delta_{0}}

\usepackage{hyperref}
\usepackage{bookmark}
\hypersetup{colorlinks=true}
\hypersetup{
    citecolor=[rgb]{0,0,0.9},
    linkcolor=[rgb]{0,0,0.9},
    urlcolor=[rgb]{0,0,0.9},
}

\usepackage{paralist}
\AtBeginDocument{\setdefaultleftmargin{2em}{}{}{}{}{}}

\usepackage{mathrsfs}
\usepackage{amssymb}

        \newcommand\curlyC {%
            \ensuremath{%
               \setbox0=\hbox{$\mathcal C$}%
                \rlap{$\mskip0.3\thinmuskip$%
                      \raise 1.05\ht0\hbox{%
                                  \scriptsize$\rlap{$\mskip0.9\thinmuskip{\circ}$}%
                                  {\sim}$%
                      }}
                \mathcal C
               }}

\newcommand{\jknewtheorem}{\csname @ifstar\endcsname{\jknewtheoremstar}{\jknewtheoremnostar}}
\newcommand{\jknewtheoremxxenv}[1]{%
     \newenvironment{#1}[1][--noopt--]{%
              \def\tmpa{[##1]}\def\tmpb{[--noopt--]}%
              \ifx\tmpa\tmpb\def\tmpa{}\fi
              \def\tmpc{\eodhh\begin{#1xxxx}}%
              \expandafter\tmpc\tmpa
              \let\qedsymbol\eodsymbol\pushQED\qed
              }%
              {\eodatend\end{#1xxxx}}%
  }%
\newcommand{\jknewtheoremnostar}[1]{%
     \jknewtheoremxxenv{#1}%
     \newtheorem{#1xxxx}}%
\newcommand{\jknewtheoremstar}[1]{%
     \jknewtheoremxxenv{#1}%
     \newtheorem*{#1xxxx}}%
\newcommand{\eod}{\qedhere}
\newcommand{\eodsymbol}{\hbox{\rlap{$\ulcorner$}$\lrcorner$}}
\newcommand{\eodhh}{\let\eodatend\eod\def\eodhere{\eod\global\let\eodatend\ignorespaces}}
\newcommand{\lowereodsymbol}[1]{\def\qedsymbol{\lower#1\eodsymbol}}

\newtheorem{theorem}{Theorem}[section]
\newtheorem{proposition}[theorem]{Proposition}
\newtheorem{lemma}[theorem]{Lemma}
\newtheorem{corollary}[theorem]{Corollary}
\newtheorem*{lemma*}{Lemma}
\newtheorem*{theorem*}{Theorem}

\theoremstyle{remark}
\jknewtheorem{example}[theorem]{Example}

\theoremstyle{remark}

\theoremstyle{definition}
\jknewtheorem{definition}[theorem]{Definition}
\jknewtheorem*{definition*}{Definition}
\jknewtheorem{remark}[theorem]{Remark}
\jknewtheorem*{remark*}{Remark}

\numberwithin{equation}{section}

\newcommand\itemref[1]{\eqref{#1}}

\DeclareMathOperator{\proj}{proj}
\DeclareMathOperator{\dist}{dist}
\DeclareMathOperator{\diam}{diam}
\DeclareMathOperator{\dom}{dom}
\DeclareMathOperator{\Lip}{Lip}
\DeclareMathOperator{\domlip}{domLip}
\DeclareMathOperator{\Span}{span}
\DeclareMathOperator{\closedSpan}{\closure{\Span}}

\newcommand\R{\mathbb R}
\newcommand\N{\mathbb N}

\newcommand\Q{\mathbb Q}
\newcommand\Qplus{\mathbb Q^+}
\newcommand{\e}{\varepsilon}
\newcommand{\lone}{\mathcal L_1}
\newcommand\fcolon{\colon}
\newcommand\setcolon{:}
\newcommand\downto{\searrow}

\newcommand\closure{\overline}
\newcommand{\primeplus}{{\prime+}}

  \def\xomegaone{\lambda}
  \def\xomegaoneNOINDEX{\lambda}
  \def\xomegatwo{\lambda}
  \def\xdeltaone{\kappa}
  \def\xdeltatwo{\kappa}
  \def\xdeltathree{\kappa}

\newcommand*{\coloneq}{\mathrel{\mathrel{\vcenter{\baselineskip0.5ex \lineskiplimit0pt \hbox{\scriptsize.}\hbox{\scriptsize.}}}=}}%

\def\jkifnextchar#1#2#3{\let\jknextcharccc#1\def\jknextcharyyy{#2}\def\jknextcharnnn{#3}\futurelet\jknextcharxxx\jknextcharcont}
\def\jknextcharcont{\ifx\jknextcharccc\jknextcharxxx \let\jknextcharnext=\jknextcharyyy\else \let\jknextcharnext=\jknextcharnnn \fi\jknextcharnext}

\newcommand{\EdomfY}{Y}

\newcommand{\Edomfzero}{\Edomf_0}

\newcommand{\spaceY}{Y}
\newcommand{\spaceZ}{Z}
\newcommand{\testspaceX}{X}
\newcommand{\pointx}{x}
\newcommand{\pointy}{y}

\newcommand{\haty}{{\hat y}}
\newcommand{\BBB}{B}

\newcommand{\Edomf}{F\jkifnextchar_{\!}{}}
\newcommand{\Edomfpreimage}{H}
\newcommand{\domf}{\Edomf}
\newcommand{\gpreimage}{\Edomfpreimage}
\newcommand{\extendedf}{\tilde f}

  \newcommand{\testfn}{g}
  \newcommand{\testfunX}{g}
  \newcommand{\testfunY}{g}

  \newcommand{\fngA}{g}
  \newcommand{\fngB}{g}

  \newcommand{\fngD}{g}
  \newcommand{\tildegD}{\tilde g}

  \newcommand{\fngF}{g}
  \newcommand{\tildegF}{\tilde g}
  
  \newcommand{\fngI}{g}
\newcommand{\fun}{f}
\newcommand{\funf}{f}
\newcommand{\fung}{g}
\newcommand{\spaceXX}{X}
\newcommand{\spaceYY}{Y}

\newcommand{\domainHzeroWasZero}{H_0}
\newcommand{\domainHzeroWasOne}{H_0}

\newcommand\abs[1]{\left|#1\right|}

\newcommand\norm[1]{\left\|#1\right\|}
\newcommand\bignorm[1]{\bigl\|#1\bigr\|}

\newcommand\Gateaux{G\^ateaux }

\newcommand\Frechet{Fr\'echet }

\DeclareMathOperator{\HDoperator}{\mathcal {HD}}
\DeclareMathOperator{\Doperator}{\mathcal {D}}
\newcommand\hadamderset[1][]{\HDoperator_{#1}}
\let\hadamdersetboth\hadamderset
\let\hadamdersetplus\hadamderset
\let\hadamdersetsym\hadamderset
\newcommand\derset[1][]{\Doperator_{#1}}
\newcommand{\bd}{\blacktriangledown}
\DeclareMathOperator{\regtan}{regTan}
\DeclareMathOperator{\regtanex}{distTan}
\newcommand\identity{\mathbf{I}}

\newcommand\ltr{_{\text{r}}}

\let\Uindexedbyfnspace\Uindexedby

\newcommand{\Hindexedby}[1]{\mathop{\mathscr{H}\!_{#1}}\nolimits}
\let\Hindexedbyfnspace\Hindexedby
\newcommand{\Phir}{\Phi\ltr}
\newcommand{\UPhi}{\Hindexedbyfnspace{\Phi}}
\newcommand{\UPhirany}{\Hindexedby{\Phir}}
\newcommand{\UPhirplus}{\Hindexedby{\Phir}}

\newcommand{\UPhiplus}{\Hindexedbyfnspace{\Phi}}
\newcommand{\UPhiplusArg}[1]{\Hindexedby{\Phi,#1}}
\newcommand{\UPhiplusR}{\Hindexedby{\Phi,R}}

\newcommand{\UPhiplusY}{\Hindexedby{\Phi,[Y]}}
\newcommand{\UPhirplusY}{\Hindexedby{\Phir,[Y]}}

\newcommand{\UPhix}{\Uindexedbyfnspace{\Phi}}

\newcommand{\UPhiMP}{\UPhix}

\newcommand{\rnpindex}{*}
\newcommand{\withrnpindex}{_{\rnpindex}}
\newcommand{\EAindexedby}[1]{\mathop{\mathscr{H\!\!A}\!_{#1}}\nolimits}

\newcommand{\EAPhi}{\EAindexedby{\Phi}}

\begin{document}

\title[Chain rule for pointwise Lipschitz mappings]{Chain rule for pointwise Lipschitz mappings}

\author{Jan Kol\'a\v{r}}
\address{Institute of Mathematics, Czech Academy of Sciences}
\email{kolar@math.cas.cz}

\author{Olga Maleva}
\address{School of Mathematics, University of Birmingham}
\email{o.maleva@bham.ac.uk}

\thanks{The first named author acknowledges support by the
Institute of Mathematics, Czech Academy of Sciences (RVO 67985840). Both authors acknowledge support by the EPSRC grant EP/N027531/1 and by the LMS Research in Pairs grant 42338.}


\begin{abstract}
The classical Chain Rule formula $(f\circ g)'(x;u)=f'(g(x);g'(x;u))$
gives the (partial, or directional) derivative of the composition of mappings $f$ and $g$.
We show how to get rid of the unnecessarily strong assumption of differentiability at all of the relevant points:
the mappings do not need to be defined on the whole space and
it is enough for them to be pointwise Lipschitz. 
The price to pay is that the Chain Rule holds almost everywhere. We extend this construction 
to infinite-dimensional spaces with good properties (Banach, separable, Radon-Nikod\'ym) with an appropriate notion of almost everywhere. Pointwise Lipschitzness is 
%inherently 
a local condition in contrast to the global Lipschitz property: 
%on the whole space or an open neighbourhood; we also 
we do not need the mappings to be defined on the whole space, or even locally in a neighbourhood, nor to know their behaviour far away from the points we consider. This distinguishes our results from recent research on
the differentiation of the composition of Lipschitz mappings.
The methods we develop for the purpose of proving the Chain Rule also
allow us to strengthen the Rademacher-Stepanov type theorem on almost everywhere
differentiability of a mapping.
\end{abstract}

\maketitle

\thispagestyle{empty}

\section{Introduction}

        The classical chain rule formula $(f\circ g)'(x)=f'(g(x))g'(x)$ for
        differentiating a composition of mappings makes perfect sense if one assumes
        that~$g$ is differentiable at~$x$ and $f$ is differentiable at~$g(x)$.  This
        assumption is, however, too strong for  $(f\circ g)'(x)$ to exist: an extreme
        example is a constant mapping~$g$ where nothing is required of~$f$ other than
        being defined at one point.
It is therefore imperative to find
       a meaningful interpretation
        of the right-hand side of the chain rule, where the assumptions about
        $f$ and $g$ are relaxed.

        Results to date have led to establishing the chain rule formula when~$f$ and~$g$ 
        are Lipschitz mappings between Banach spaces. 
        Differentiability of Lipschitz mappings has been 
        an area
        of active research
        for at least a century; recall for example the celebrated Rademacher theorem
        which guarantees that a real-valued Lipschitz function on~$\R^n$ is
        differentiable almost everywhere. So for $f\colon Y\to \R$  and $g\colon
        \R^n\to Y$ Lipschitz, $(f\circ g)'(x)$ exists for a.e.\ $x\in\R^n$ whereas $f'$ may
        not be defined on~$g(\R^n)$. 
        If $Y$ is finite-dimensional, this situation is covered by the result of Ambrosio and Dal~
        Maso~\cite{AD} who proposed to replace the full derivative of~$f$ by what is now known as
        ``derivative assignment'' (possibly dependent on~$g$).
        More recently, the second named author
        and  Preiss proved in~\cite{MP2016} that 
        a complete derivative assignment for $f$ exists independently of $g$.
        The result of~\cite{MP2016} also holds for mappings between
        infinite-dimensional Banach spaces. ``Almost everywhere''
        in~\cite{MP2016} means a complement of a set from a $\sigma$-ideal
        defined by non-differentiability of Lipschitz mappings to codomains
        with Radon-Nikod\'ym property; see Definition~\ref{def:Lnull} below. In
        finite-dimensional spaces, such sets are Lebesgue null.
        The interested reader may also
	 refer to~\cite{AM} where 
		derivative assignments arise in investigation of 
		a.e.\ differentiability with
        respect to finite measures on~$\R^n$, though not in relation to the chain rule.

        The results of the present paper mean that pointwise Lipschitzness, a natural condition for existence of Hadamard-type 
        derivatives for mappings defined on arbitrary sets, is enough to establish the Chain Rule -- 
	the global Lipschitz condition
        is not necessary.         
        In Theorem~\ref{thm:chain:nintro} we give the 
        more general chain rule for arbitrary mappings, restricting the statement to the domain of pointwise Lipschitzness.
This theorem is one of the main results of the present paper, along with Proposition~\ref{p:chain-orig} which is a chain rule for derivative assignments, as well as Theorem~\ref{thm:nchain:iter} and Proposition~\ref{p:chainIter} which are extensions of the former two to the case of an arbitrary number of mappings. Finally, Corollary~\ref{c:domindom-orig} is a strengthening of Rademacher-Stepanov theorem,
at the same time strong enough to provide the almost everywhere chain rule.

\section{Main Results}
In this paper 
we present the Chain Rule for the important class of pointwise Lipschitz mappings 
between Banach spaces. We 
go even further, and consider general mappings,
restricting the claim
to the set of points at which the
mappings are Lipschitz.
Here, by a \emph{partial\/} function we mean a mapping defined on a subset of the space. 
The natural notion of derivative for
partial functions
turns out
to be the Hadamard derivative, denoted $f'_\HforHadam(x;u)$ in this section;
cf.\ Definition~\ref{def:Hadamard}.

\medbreak

Assume
$\spaceY$ and $\spaceZ$ are Banach spaces and
        $f\fcolon E\subset \spaceY\to \spaceZ$ is a mapping. We say that  $f$ is \emph{Lipschitz at~$\pointy\in E$} if
        \[
                \Lip_\pointy(f,E)
                \coloneq
                \limsup_
                  {\substack
                     {
                            \pointx \to \pointy \\
                            \pointx \in E
                     }}
                \frac { \norm { f(\pointx) - f(\pointy) }}{ \norm { \pointx - \pointy }}
        \]
        is finite.
We say that $f$ is \emph{pointwise Lipschitz}
if $f$ is Lipschitz at each $y\in E$.

\medbreak
        It should be noted that even the case
        of a composition $f\circ g$
        of pointwise Lipschitz mappings between
        the Euclidean spaces
        is by no means
        straightforward: although
        the mappings
        are 
        differentiable almost everywhere
        by the 
        Stepanov theorem of 1923, $f$ may be nowhere differentiable on 
        the image of~$g$.
        Furthermore, in the case of partial functions the derivatives 
        clearly need to be defined carefully
        as, for a fixed point
        $x$ and a direction $u$, the domain $E$ may not contain enough
        points to define directional derivative at $x$ in the direction of $u$
        meaningfully.
        Finally, for infinite-dimensional spaces 
        it is imperative 
        to define a suitable notion of `almost everywhere' or complement of a `negligible' set. 

\medbreak

        We take our starting point at the following generalisation of the
        Stepanov theorem, where the class $\curlyC$ of subsets of $\spaceY$ is
        defined in~\cite{Preiss2014}.  
\begin{theorem}[{\cite[Theorem~4.2]{Preiss2014}}] \label{thm:p}
                Let $f$ be a map of a subset $E$ of a separable Banach space~$\spaceY$
                to a Banach space~$\spaceZ$ with the Radon-Nikod\'ym property. Then $f$ is Hadamard
                differentiable at $\curlyC$-almost every $\pointy\in E$ at which $\Lip_\pointy(f,E) < \infty$.
\end{theorem}

        In the present paper we strengthen
        Theorem~\ref{thm:p}
        in such a way
        that it provides base for an almost everywhere (directional) Chain Rule for pointwise Lipschitz mappings, see Theorem~\ref{thm:chain:nintro},
 by describing the almost everywhere nature of the collection of point-direction pairs at which $f'_\HforHadam$ exists. See Theorem~\ref{t:PointLip-copy} and Corollary~\ref{c:domindom-orig}.      
         
        Recall that the domain of a mapping $f$, denoted by $\dom(f)$, is the set of all points at which $f$ is defined.

\medbreak

\begin{definition}
\label{def:domination-general}
        Assume $Y$ is a Banach space, $\Edomf \subset Y$.
        Let $\spaces$ be a class of metric spaces.
        For every $X\in\spaces$, let us fix the notation $\mathcal E(X)$ and $\tests(X)=\tests(X,Y)$ for the following collections. By
        $\mathcal E(X)$ we denote a $\sigma$-ideal of ``small'' subsets of $X$ (usually related to differentiability), and
        $\tests(X)=\tests(X,Y)$  denotes a collection of mappings $g$ from subsets of $X$ to $Y$.
        The elements of $\tests(X)$ are called \emph{test mappings}.
        For $X\in \spaces$ and $g\in \tests(X)$,
        let $\diffop_X(g)$ be an \emph{abstract
        (directional)
        differentiability operator}, that is, a mapping defined on a subset of $\dom(g)\times X$ with values in $Y$.

        We say that a set~\(
        M
        \subset \Edomf \times Y
        \)
        is
        \emph{$\mathcal E$-dominating in~$\Edomf\times Y$
        with respect to $(\spaces,\tests)$ and $\diffop$}
        if
        for every $X\in\spaces$
        and
        every test mapping $g\in\tests(X)$,
        there is
        a set
        $N\in\mathcal E(X)$
        such that        
        $\bigl(\testfn(x), \diffop_X(g)(x,u)\bigr) \in M$
        whenever $x\in \testfn^{-1}(\Edomf) \setminus N$, $u\in X$ and
        $(x,u)$ belongs to the domain of $\diffop_X(g)$.

        We say that $M$ is \emph{linearly $\mathcal E$-dominating in~$\Edomf\times Y$
        with respect to $(\spaces,\tests)$
        and $\diffop$}
        if,
        in addition, for every $y \in \Edomf$, the section 
        $M_y=\{v\in Y\setcolon (y,v)\in M\}\subset Y$ 
        is
        a closed linear subspace of~$Y$.
\end{definition}
        An important particular case of the above
        notion
        that is studied in depth in the present paper is obtained by
        specifying $\spaces$, $\tests$ and
        $(\diffop_X)_{X\in \spaces}$
        as is done in the definition below.
        The definition uses the notion of
        Hadamard derivative
        which is formally
        defined in Definition~\ref{def:Hadamard}.

\begin{definition}
\label{def:domination}
Let $\spaces$ be the class of all separable Banach spaces;
assume $Y$ is a Banach space and $\Edomf \subset Y$.
        We say that a set~\(
        M
        \subset \Edomf \times Y
        \)
        is
        \emph{$\mathcal E$-dominating in~$\Edomf\times Y$}
        if
        for every 
$X\in\spaces$,
        every set $G\subset X$
        and every Lipschitz mapping $\testfn\fcolon G\to Y$
        there is
        a set
        $N\in\mathcal E(X)$,
        such that
        $(\testfn(x), \testfn'_\HforHadam (x; u)) \in M$
        whenever $x\in \testfn^{-1}(\Edomf) \setminus N$, $u\in X$ and
        Hadamard
        derivative $\testfn'_\HforHadam(x;u)$ exists.

        We say that $M$ is \emph{linearly $\mathcal E$-dominating in~$\Edomf\times Y$}
        if,
        in addition, for every $y \in \Edomf$, the section $M_y\subset Y$ is
        a closed linear subspace of~$Y$.
\end{definition}
	Various versions of this notion can be used.
        For example in the variant which \cite{MP2016} uses implicitly,
        without giving it a particular name,
        all test mappings from $\tests(X)$ are defined on the whole space~$X$ ($G=X$).
\begin{remark}
        We can
        paraphrase some of the main
        results of \cite{MP2016} as follows:
        Assume $Y,Z$ are separable Banach spaces and $Z$ satisfies the Radon-Nikod\'ym property; for a Banach space $X$ choose
        the class $\mathcal E(X)=\mathcal L(X)$ as in Definition~\ref{def:Lnull},
        then:\\
        1.~
        A Lipschitz mapping $f\fcolon Y \to Z$ is (directionally)
        differentiable
        in a
        Borel
        set $M$
        of point-direction pairs $(y,u)$ which is linearly $\mathcal E$-dominating with respect to the set of test mappings Lipschitz on the whole space $G=X$.\\
        2.~
        For a fixed $y$, the derivative $f'(y; \cdot)$ is linear and continuous on the section $M_y$.\\
        3.~
        Given two Lipschitz mappings
        $h\fcolon Y\to Z_1$ and  $f\fcolon Z_1\to Z_2$,
        (where $Z_1$, $Z_2$ are separable Banach spaces with the Radon-Nikod\'ym property)
        the
        (directional)
        derivative of the composition $f\circ h$ is given by the Chain Rule formula
        \[
                (f \circ h)' (y;u) = f'( h(y); h'(y;u) )
        \]
        in a
        Borel
        set of point-direction pairs $(y,u)$ which is
        linearly
        $\mathcal E$-dominating in $Y$.

        We
        generalise
        these statements in Theorems~\ref{t:PointLip-copy}
        and~\ref{thm:chain:nintro} below.
\end{remark}

\let\VimSpellOn(

        For the research presented in this paper,
        the following class of negligible sets is the natural choice.

\begin{definition}\label{def:L1null}
        Let $X$ be a separable Banach space.
        We define $\lone=\lone(X)$ to be the $\sigma$-ideal generated by sets
        $N\subset X$ for which there is a set $G\subset X$ containing $N$, a
        Banach space $Y$ with the Radon-Nikod\'ym property
        and a Lipschitz mapping $\fngI\fcolon G\to Y$ such that for each
        $x\in N$,
        the mapping $\fngI$ is not Hadamard differentiable at $x$, see Definition~\ref{def:HadamardDifferentiable}.
        Members of this $\sigma$-ideal are also referred to as \emph{$\mathcal
        L_1$~null} sets.
\end{definition}
        We note that by Theorem~\ref{thm:p},
        every set in $\mathcal L_1$ is also in $\curlyC$, a non-trivial
        $\sigma$-ideal.
\begin{remark}\label{r:equiv-def-l1}
        We show in Lemma~\ref{l:equiv-def-l1} that $\mathcal L_1$ coincides
        with the $\sigma$-ideal generated by the sets of Hadamard
        non-differentiability of pointwise Lipschitz
        (Banach space valued)
        mappings defined
        on the whole space~$X$.

        Also, if $X$ is a finite-dimensional space, then any set $N\in
        \lone(X)$ is Lebesgue null, by Theorem~\ref{thm:p}.
\end{remark}

The following theorem
is proved in Section~\ref{sec:PointLip}, see the
broader
statement of Theorem~\ref{t:PointLip}.
By Lemma~\ref{l:almost},
Theorem~\ref{t:PointLip-copy}
is indeed a strengthening of Theorem~\ref{thm:p}.

\begin{theorem}\label{t:PointLip-copy}
        Let $\spaceY$ and $\spaceZ$ be Banach spaces.
        Assume
        that
        $\spaceY$ is separable
        and
        $\spaceZ$ has the Radon-Nikod\'ym property.
        Let $\Edomf \subset \spaceY$
        and let $f\fcolon \Edomf\to \spaceZ$
        be a mapping.
        Let $\Edomfzero = \{ \pointy\in \Edomf \setcolon \text{$f$ is Lipschitz at $\pointy$}\}$.
        Then there is a linearly $\mathcal L_1$-dominating set $S$ in~$\Edomfzero \times \spaceY$,
        (relatively) Borel in~$\Edomfzero \times \spaceY$,
        such that
        $f'_\HforHadam(\pointy;u)$ exists for every $(\pointy,u)\in S$, and
        depends continuously and linearly on $u\in S_\pointy$ for every
        $\pointy\in \Edomfzero$.
\end{theorem}

        Note that
        if $\pointy\in \dom(f\circ h)$ and $\dim S_\pointy < \infty$
        \textup(for example, when $\spaceY$ is a Euclidean space\textup)
        then
        the limit defining
        the Hadamard derivative at~$\pointy$
        converges uniformly
        with respect to all directions~
        $u$ from any relatively compact (i.e., bounded) subset of~$S_\pointy$
        and $u \in S_\pointy \mapsto f'_\HforHadam(\pointy;u)$
        is a \Frechet derivative of the restriction of $f$ to the intersection of $\Edomf$ and the affine subspace $\pointy + S_\pointy$, 
        see also \cite[Corollary~2.50]{Penot}.

The following simple lemma explains why the notion of dominating set is 
stronger than the usual `almost everywhere' notion. 
In particular, application of Lemma~\ref{l:almost} to the $\lone$-dominating set from Theorem~\ref{t:PointLip-copy} 
recovers Stepanov and Rademacher Theorems for mappings between Euclidean spaces, see Remark~\ref{r:equiv-def-l1}.
        For a further generalization of the Stepanov theorem see Corollary~\ref{c:domindom-orig},
        which has enough power to ensure almost everywhere validity of the Chain Rule formula.

\begin{lemma}\label{l:almost}
        Assume $Y$ is a separable Banach space and $E\subset Y$.
        If $M\subset E\times Y$ is
        $\mathcal E$-dominating in $E\times Y$
        then $M_y=Y$ for $\mathcal E(Y)$-almost every $y\in E$.
\end{lemma}
\begin{proof}
        Let $\testfn\fcolon Y \to Y$ be the identity mapping. Then $\testfn$ is Hadamard differentiable at
        every point $x\in Y$ in any direction $u\in Y$.
        Using Definition~\ref{def:domination} for $M$, we get a set $N\in \mathcal E(Y)$
        such that $(x,u)\in M$ for every $x \in E \setminus N$, $u \in Y$.
\end{proof}

	An important case of Theorem~\ref{t:PointLip-copy} is when
	$f$ is the distance function to a set~$M$.
        We obtain this way what
        we call \emph{a regularized tangent} to $M$, see Definition~\ref{def:aregTan}; it plays an important
        role
        in the present paper.
     For the next corollary see  Definition~\ref{def:porous} of porosity.   
\begin{corollary}\label{c:setDomNpor}
	Let $Y$ be a separable Banach space, $M\subset Y$.
	Then there is a relatively Borel subset $T$ of $M\times Y$, linearly $\mathcal L_1$-dominating in $M\times Y$
	such that, for every $(y,u) \in T$, $M$ is not porous at $y$ in the direction of $u$.
\end{corollary}
\begin{proof}
        Indeed, if $M$ is porous at $y$ in the direction of $u$, see
        Definition~\ref{def:porous}, then $f=\dist(\cdot,M)$ is not (Hadamard)
        differentiable at $y$ in the direction of $u$, see
        Remark~\ref{r:derandporous}.
        Let $T\subset M\times Y$ be the set
        $S$
        from Theorem~\ref{t:PointLip-copy}.
\end{proof}
        In the next theorem we prove  Chain Rule  for general mappings,
        restricting the statement to the \emph{domain of pointwise Lipschitzness} of the
        mappings, in the spirit of Theorem~\ref{thm:p}. 
        We define the domain of pointwise Lipschitzness of a mapping $f\fcolon E\subset Y\to Z$ as
\begin{align}\label{eq:dom_Lip}
        \domlip(f)=\{y\in E\setcolon \Lip_y(f,E)<\infty\}.
\end{align}

\jkrestatable{T:chain}{%
                Let $X, Y, Z$ be Banach spaces.
                Assume that
                $X$ and $Y$ are separable, and $Y$ and $Z$ have
                the Radon-Nikod\'ym property.
                Assume that
                $h\colon E_X \subset X \to Y$
                and
                $f\colon E_Y \subset Y \to Z$
                are mappings.
                Let $\domainHzeroWasZero=\domlip(h)\cap h^{-1}(\domlip (f))$.
                Then
                \[
                        (f \circ h)'_\HforHadam (x;u)
                        =
                        f'_\HforHadam \bigl( h(x); h'_\HforHadam(x;u) \bigr)
                \]
                for all $(x,u) \in D$
                where $D$ is
                a (relatively Borel) set linearly $\mathcal L_1$-dominating in $\domainHzeroWasZero \times X$.
                \\
                For every
                $x\in  \domainHzeroWasZero$,
                the dependence
                of $ (f \circ h)'_\HforHadam (x;u) $
                on $u\in D_x$ is linear and continuous.
}

\begin{theorem}\label{thm:chain:nintro}
        \restateWithLPrefix{T:chain}{Nintro}
\end{theorem}

        Theorem~\ref{thm:chain:nintro} is proved in
        Section~\ref{sec:chain_rule}, see Theorem~\ref{thm:chain:n}.
        Its proof is supported by Proposition~\ref{p:chain-orig}, proved in Proposition~\ref{p:chain}.
        These two results provide a framework where the Chain Rule for Hadamard derivatives of arbitrary mappings can be understood and proved to be valid `almost everywhere' in the appropriate domain, see $\domainHzeroWasZero$ in  Theorem~\ref{thm:chain:nintro}.    
        
        We then end the paper with 
        Proposition~\ref{p:chainIter} and
        Theorem~\ref{thm:nchain:iter} which provide the Iterated Chain Rule, for an arbitrary number of mappings,         
        generalising the case of just two mappings and explaining that the meaning of `almost everywhere' does not need to be changed when the number of mappings and spaces they act between increases.

       \medbreak

        Composition of an arbitrary number of mappings also motivates
        the following corollary.
        To prove Corollary~\ref{c:domindom-orig}, one has to apply the ideas of
        the proof of Theorem~\ref{thm:chain:nintro} to a constant partial
        function $f$. See Corollary~\ref{c:domindom-copy} for more details.

\jkrestatable{rst:domindom}{%
                Let $X, Y$ be Banach spaces.
                Assume that
                $X$ is separable and $Y$ has
                the Radon-Nikod\'ym property.
                Let
                $h\colon E_X \subset X \to Y$
                be an arbitrary mapping.

                Let $F\subset Y$, and let $\setS \subset F\times Y$ be
                a relatively Borel set linearly dominating in $F\times Y$.
                Let $\domainHzeroWasZero=\domlip(h) \cap h^{-1}(F) $.
                Then
                there is
                $D \subset \domainHzeroWasZero \times X$,
                relatively Borel and linearly $\mathcal L_1$-dominating in $\domainHzeroWasZero \times X$
                such that
                $h'_\HforHadam(x;u)$ exists and
                \[
                       (h(x) ,  h'_\HforHadam(x;u) )
                       \in \setS 
                \]
                for all $(x,u) \in D$.
                For every 
                $x\in  \domainHzeroWasZero$,
                the dependence
                of $ h'_\HforHadam (x;u) $
                on $u\in D_x$ is linear and continuous.
}

\begin{corollary}\label{c:domindom-orig}
        \restateWithLPrefix{rst:domindom}{orig}
\end{corollary}

       It is not difficult to check that
       Theorem~\ref{thm:chain:nintro} can be recovered from
       Corollary~\ref{c:domindom-orig}
       when used together
       with Theorem~\ref{t:PointLip-copy}, Lemma~\ref{l:classchain},
       Corollary~\ref{c:setDomNpor} and Remark~\ref{r:thickandporous}.
       Just take $S\subset \Edomfzero \times \spaceY$ provided by
       Theorem~\ref{t:PointLip-copy} and apply Corollary~\ref{c:domindom-orig}
       to get $D\subset \domainHzeroWasZero \times X$. Then
       Lemma~\ref{l:classchain} provides the chain rule formula on $D$.
       Let us also note that Lemma~\ref{l:classchain} is a
       version
       of~\cite[Theorem~2.28]{Penot}.

       As Theorem~\ref{t:PointLip-copy} itself is a special case of Corollary~\ref{c:domindom-orig}
       (with
       $F=Y$ and
       $\setS =Y\times Y$),
       and
       Corollary~\ref{c:setDomNpor} can be also deduced from Corollary~\ref{c:domindom-orig}
       (with the same proof)
       we see that
       Theorem~\ref{thm:chain:nintro}
       and Corollary~\ref{c:domindom-orig} are essentially equivalent.
       Obviously,
       Corollary~\ref{c:domindom-orig} is a strengthening of Theorem~\ref{t:PointLip-copy}.

\bigbreak
\let\makeTheSillyVimSpellerWork(

        We use the case of spaces of finite dimension
        to demonstrate how the above results bring us closer to an optimal one. We understand the optimality in
        terms of almost everywhere differentiability, as
        for most applications,
        a reasonable required or expected outcome
        would be the validity of the chain rule \emph{almost everywhere}.
        This, as we explain in the beginning of this paper,
        is driven by the conclusion of the Rademacher theorem which guarantees
        the almost everywhere differentiability of $f\circ h$ under
        minimal necessary conditions on $f$ and $h$.

        Let us consider the outer mapping first.
        The necessary requirements for the outer mapping
        need to be rather strong, disregarding whether they are stated in terms of the Lipschitzness 
        or differentiability properties.
        Indeed, if the inner mapping is, for example, a linear mapping with rank smaller than $\dim Y$, then
        no \emph{almost everywhere} assumption on~$f$ can lead to a valid result.
        Thus the pointwise Lipschitz property seems
        to be a reasonable condition on the outer mapping.    
        For the chain rule to make sense almost everywhere
        we should then assume that the inner mapping is (Fr\'echet) differentiable almost everywhere, as
        its derivative is a part of the expression on the right hand side.
        The following theorem shows that with these assumptions on the outer
        and inner mappings the chain rule holds for finite-dimensional spaces.

   Note that in the context of Theorem~\ref{t:total} the partial derivatives of $\partial f/\partial y_j$ do not need to exist
   at~$h(x)$
but the theorem
implicitly asserts that those directional derivatives $f'(h(x);\cdot)$ that are used in~\eqref{eq:total},~\eqref{eq:partial}
exist, for almost every $x\in \R^n$.

\begin{theorem}\label{t:total}
        Let 
        $h \fcolon \R^n \to \R^k$ be differentiable almost everywhere in the ordinary sense, and 
        $f \colon \R^k \to \R^m$ be 
        pointwise Lipschitz. Then,
        for almost every $x\in \R^n$,
        $f\circ h$ is differentiable in the ordinary sense at~$x$ and
  \begin{equation}\label{eq:total}
                        (f \circ h)' (x;u)
                        =
                        f' \bigl( h(x); h'(x;u) \bigr)
                    \qquad
                        \text{for all } u \in \R^n.
  \end{equation}
  In particular, for almost every $x\in \R^n$,
  \begin{equation}\label{eq:partial}
                        \frac{ \partial (f(h(x)))  }
                             { \partial x_i }
                        =
                        f' \Bigl( h(x); \frac {\partial h(x)}{\partial x_i} \Bigr)
                    \qquad
					i=1,\dots,n.
  \end{equation}
\end{theorem}
\begin{proof}
        Let $H_1$ be the set of all $x\in \R^n$ such that~$h$ is differentiable at~$x$;
        then $\R^n \setminus H_1$ is a set of measure zero.
        Obviously, if $x\in H_1$, then~$h$ is Lipschitz at~$x$,
        so
    $\domlip(h) \supset H_1$; 
    also $\domlip(f) = \R^k$.
        By Theorem~\ref{thm:chain:nintro},
  \begin{equation}\label{eq:total2}
                        (f \circ h)'_\HforHadam (x;u)
                        =
                        f'_\HforHadam \bigl( h(x); h'_\HforHadam(x;u) \bigr)
                    \qquad
  \end{equation}
        holds true for every $(x,u) \in D$, where $D$ is a set $\lone$-dominating in $\domlip(h)\times \R^n$.
        By Remark~\ref{r:equiv-def-l1} all sets  $N\in \lone(\R^n)$ are of Lebesgue measure zero.
        By Lemma~\ref{l:almost}, $D_x = \R^n$ for
        $\lone(\R^n)$ 
        almost all $x\in \domlip(h)$, and therefore for almost all $x\in \R^n$.
        That means that, for almost all $x\in \R^n$,
        $(f \circ h)'_\HforHadam (x;u)$ exists for all $u\in \R^n$.
        Referring, as we have
        done 
        after Theorem~\ref{t:PointLip-copy},
        to \cite[Corollary~2.50]{Penot}, we obtain that
        the existence of the Hadamard derivative for all directions implies
        the differentiability
        of $f\circ h$
        at the point.
        Finally,~\eqref{eq:total} follows from~\eqref{eq:total2}. 
\end{proof}
   We note that if
   $h$ and $f$
   are only defined on open subsets
   of $\R^n$ and $\R^k$, respectively,
   and the range of $h$ is contained in
   the domain of~$f$
   then the statement of Theorem~\ref{t:total}
   still holds for almost every $x$
   in the domain of~$h$.    
        In order to obtain~\eqref{eq:partial},
        the domains do not really need
        to be assumed
        to be open sets:
        it is sufficient to assume  
        that
        the domains contain enough points so that
        both the partial derivative of $h$ and the directional derivative
        of $f$ are well defined, for almost every $x$ from the domain of $h$ and every $i=1,\dots,n$.
        If Hadamard directional derivatives
        from Definition~\ref{def:Hadamard} are used instead of
        partial and directional derivatives,
        such an assumption may be removed altogether. This can be seen 
        in the statement of
        Theorem~\ref{thm:chain:nintro}
        where the dominating set is implicitly asserted to contain only
        the appropriate point-direction pairs
        and yet it provides, via Lemma~\ref{l:almost},
        the result for almost every $x\in \dom h\subset \R^n$.

\begin{remark}
    One may ask whether the directional
    derivative of~$f$ in the right-hand side of~\eqref{eq:partial}
    could be replaced with an expression involving only the partial derivatives.
    The following example demonstrates
    that this is impossible.
    Let    $f(x,y)=\abs{x-y}$, $(x,y)\in \R^2$,
    $h(t) = (t,t)$, $t\in \R$.
    Indeed, the partial derivatives of~$f$
    exist at no point of the diagonal $(t,t)$, $t\in \R$.
\end{remark}

\begin{remark}
The statement and proof of Theorem~\ref{t:total} will of course stay true if we 
consider $h \fcolon \R^n \to Y$  and 
$f \fcolon Y\to Z$, where the Banach spaces
$Y$ and $Z$ have the Radon-Nikod\'ym property, and $Y$ is separable. The reason we state Theorem~\ref{t:total} only for finite-dimensional spaces is to present a classical result for mappings between finite-dimensional spaces without the need to go deep into notions of functional analysis.
\end{remark}

\section{Preliminaries}

We first recall the definition of Hadamard differentiability. 

\begin{definition}\label{def:Hadamard}
Let $Y,Z$ be Banach spaces, $y\in \Edomf\subset Y$, $u\in Y$, $z\in Z$ and $f\fcolon \Edomf\to Z$ a function.
We say that $z$ is a \emph{Hadamard derivative} of $f$ at $y$ \emph{in the direction of} $u$
if
\[
   \lim_{n\to \infty} \norm{\frac{f(y + t_n u_n)-f(y)}{t_n} - z } = 0
\]
whenever $t_n\to 0$ and $u_n\to u$ are such that $y + t_n u_n \in \Edomf$ and $t_n\neq 0$ for every $n\in \N$. 
From now on we only use Hadamard derivatives, which in most cases we either indicate by using words ``Hadamard derivative'' or a subscript $\HforHadam$: $z=f'(y;u)=f'_\HforHadam(y;u)$. 

Note that it may happen that there are no such sequences $0\ne t_n\to0$ and $u_n\to u$ that $y+t_n u_n\in \Edomf$ for all $n$; in such case  any vector $z\in Z$ is a Hadamard derivative of $f$ at $y$ in the direction of $u$, and $f'_\HforHadam(y;u)$ is not defined.

Also, $u=0$ is allowed; then  $z=0$ is a Hadamard derivative of $f$ at $y$ in the direction $u=0$ under the additional assumption that $f$ is Lipschitz.
See also \cite[Lemma~2.3]{Z-PAMS}.
\end{definition}

        The following definition agrees with the one given in \cite[p.~504]{Preiss2014}.
\begin{definition}\label{def:HadamardDifferentiable}
Let $Y,Z$ be Banach spaces, $y\in \Edomf\subset Y$ and $f\fcolon \Edomf\to Z$ a function.
We say that $f$ is  \emph{Hadamard differentiable}  at $y$ if
there is a continuous linear map $L\fcolon Y \to Z$
called a \emph{Hadamard derivative} of $f$ at $y$
such that for every $u\in Y$ we have that $L(u)$ 
 is a Hadamard derivative of $f$ at $y\in \Edomf$ in the direction of $u$.
\end{definition}
\begin{remark}\label{r:HadSameGat}
        If $f\fcolon Y\to Z$ is a Lipschitz function defined on the whole space then
        the notions of Hadamard and \Gateaux derivative are equivalent. The same applies to
        Hadamard
        directional derivative
        and
        directional derivative,
        and also to Hadamard and \Gateaux
        differentiability.

  Therefore the results of~\cite{MP2016}
        can be read with
        the
        Hadamard differentiability in mind.
\end{remark}

\bigbreak

Let $Y,Z$ be Banach spaces.
Let us define, for $y_0, u\in Y$, $\delta>0$, $\omega>0$,
the
truncated
cone
\[
        C(u,\delta,\omega) =
                        \{
                                t \hat u
                        \setcolon
                                t\in(-\delta,0)\cup (0,\delta),
                                \hat u \in B(u, \omega)
                        \}
\]
and $ C_{y_0}(u,\delta,\omega) = y_0 + C(u,\delta,\omega)$.
Notice that $C(0,\delta,\omega) =B(0,\delta\omega)$ and $C_{y_0}(0,\delta,\omega) =B(y_0,\delta\omega)$.

\smallbreak
We now introduce an equivalent definition of  directional Hadamard derivatives.

\begin{definition}\label{def:Hadamard2}
Let $\Edomf$ be a subset of $Y$ and $f\fcolon \Edomf \to Z$ a function, $y\in \Edomf$ and $u\in Y$.   
A vector $z\in Z$ is called a \emph{Hadamard derivative} of $f$ at $y\in \Edomf$ \emph{in the direction of} $u$
if for every $\varepsilon >0$ there are $\delta>0$ and $\omega>0$
such that
\begin{equation}\label{eq:Hadamard2}
   \norm{\frac{f(y + t \hat u)-f(y)}{t} - z } < \varepsilon
\end{equation}
for every
$t\in(-\delta,0)\cup (0,\delta)$
and
$\hat u \in B(u, \omega)$
such that $y+t\hat u\in \Edomf$.

This condition degenerates and is trivially satisfied if $\Edomf\cap C_{y}(u,\delta,\omega)$ is empty for some $\delta>0$ and $\omega>0$.
\end{definition}
\begin{remark}\label{rem:Hadamard3}
Sometimes it will be more convenient for us to define the Hadamard derivative by replacing~\eqref{eq:Hadamard2} by an equivalent condition 
        \begin{equation*}\label{eq:Hadamard3}
           \norm{\frac{f(\hat y)-f(y)}{t} - z }
           <
           \varepsilon
        \end{equation*}
        for every
        $t\in(-\delta,0)\cup (0,\delta)$
        and
        $\hat y \in \Edomf \cap B(y + t u,
        \abs t
        \omega)$.
\end{remark}

\smallbreak

We now introduce the notion which would guarantee that 
the condition in the Definition~\ref{def:Hadamard2} of the Hadamard directional
                        derivative does not degenerate.

\begin{definition}\label{def:thick}
Let $Y$ be a Banach space.
A set $\Edomf\subset Y$ is called \emph{thick at~$y_0\in \Edomf$ in the direction of~$u\in Y$}
if $\Edomf\cap C_{y_0}(u,\delta,\omega)$ is non-empty for every $\delta>0$ and $\omega>0$.
This is equivalent to the existence of sequences
$t_n\to 0$ and $u_n\to u$ such that $y_0 + t_n u_n \in \Edomf$ and $t_n\neq 0$ for every $n\in \N$.
\end{definition}
Note that $u=0$ is allowed in the above definition as a borderline case, and that every set $\Edomf$
is thick in the direction of~$0$ at every $x\in \Edomf$.

\begin{remark}\label{r:contingent}
	A set $\Edomf$ is thick at $y_0\in \Edomf$ in a direction of
        $u\in Y$
        if and only if $u$ or $-u$ lies in the contingent cone of~$\Edomf$
        at~$y_0$. For the contingent (or, tangent) cone see e.g.~\cite[Theorem~6.3.6]{B},
        \cite[p.~364]{KK}
or        \cite[p.~233]{Federer}. 
\end{remark}

        \smallbreak
\begin{lemma}\label{l:HadamMeas}
        Assume $\Edomf \subset Y$ is separable and $f\fcolon \Edomf \to Z$ is continuous.
        Then
        \[
             \{
                (y,u) \in \Edomf \times Y
                \setcolon
                \textup{Hadamard derivative of~$f$ at~$y$ in the direction of $u$ exists}
             \}
        \]
        is Borel in $\Edomf \times Y$
        and
        $(y,u) \mapsto f'_\HforHadam (y;u)$
        is Borel measurable on
        \begin{multline}\label{eq:defMBorelon}
        M \coloneq
        \{ (y,u) \in \Edomf \times Y \setcolon \Edomf
             \textup{ is thick at~$y$ in the direction of~$u$ and}
             \\
             \textup{Hadamard derivative $f'_\HforHadam(y;u)$ exists}\}
             .
        \end{multline}
\end{lemma}

\begin{proof}
        Let $Q$ be a dense countable subset of $\Edomf$.
By Remark~\ref{rem:Hadamard3},
        for $y\in \Edomf$ and $u\in Y$,
        Hadamard derivative $f'(y;u)$ exists
        if and only if
        \begin{equation}\label{eq:E78:2}
        (y,u) \in
        \bigcap_n
        \
        \bigcup_k
        \
        \bigcap_
          {\substack
             {
                \abs{ t_1 }, \abs{ t_2 }
                \in \Q \cap (0,1/k) \\
                \hat y_1, \hat y_2 \in Q
             }}
        E_{n,k,(t_1, t_2),(\hat y_1, \hat y_2)}
        ,
        \end{equation}
        where
        \begin{multline*}
        E_{n,k,(t_1, t_2),(\hat y_1, \hat y_2)}
            =
             \Bigl\{
                (y,u) \in \Edomf \times Y
               \setcolon
                \norm{\hat y_1 - ( y + t_1 u ) } < \abs{ t_1 }/ k
                \text{ and }
                \norm{\hat y_2 - ( y + t_2 u ) } < \abs{ t_2 }/ k
                \\
                \text{ implies }
                \norm{
                      \tfrac{ f(\hat y_1) - f(y) }{ t_1 }
                      -
                      \tfrac{ f(\hat y_2) - f(y) }{ t_2 }
                     }
                   \le 1/n
             \Bigr\}
             ,
        \end{multline*}
        which obviously is Borel in $\Edomf\times Y$.

        Now let $G$ be an open subset of $Z$.
        Assume
        $y\in \Edomf$, $u\in Y$,
        $\Edomf$ is thick at~$y$
        in the direction of~$u$
        and
        Hadamard derivative $f'(y;u)$ exists.
        We show that
        $f'(y;u)\in G$ if and only if
\[
(y,u) \in
\bigcup_n
\
\bigcup_k
\
\bigcap_{\substack{
       \abs t
       \in \Q \cap (0,1/k) \\
       \hat y \in Q
       }}
E_{n,k,t,\hat y},
\]
where
        \[
        E_{n,k,t,\hat y}
            =
             \left\{ (y,u) \in \Edomf \times Y
               \setcolon
                \norm { \hat y - ( y + t u) } \ge  \abs t /k
                \text{ or  }
                \dist
                   \left(
                      \tfrac{ f(\hat y) - f(y) }{ t }
                      ,
                      Z\setminus G
                   \right)
                   \ge 1/n
             \right\}
             .
        \]
Indeed,
        $f'(y;u) \in G$ if and only if $\dist(f'(y;u),Z\setminus G)>0$, which is equivalent to the existence of $n>1$ such that $\dist(f'(y;u),Z\setminus G)\ge 1/(n-1)$. By definition of the Hadamard derivative, this implies that for $k$ large enough we have that
\[        \norm{\hat y-(y+tu)}<\abs t/k
\quad\text{ implies } \quad
                \dist\bigl(\tfrac{f(\hat y)-f(y)}{t},Z\setminus G\bigr)\ge1/n\]
         for all $\abs t\in(0,1/k)$
         and $\hat y \in \Edomf$.
        Vice versa, if for some $n$ and $k$, the latter condition is satisfied for all $\hat y\in Q$ and all $\abs t\in(0,1/k)\cap \Q$, then $\dist(f'(y;u),Z\setminus G)\ge1/n$.
        
        Thus
        the condition $f'(y;u) \in G$
        cuts a (relatively) Borel
        set out of $M$.
\end{proof}

\begin{definition}\label{def:porous}
        Let $Y$ be a Banach space.
        A set $M\subset Y$ is called \emph{porous at $y_0 \in M$ in the direction of $u\in Y$}
        if there is $c>0$ and a sequence $t_n \to 0$ of positive numbers such that
        $B(y_0 + t_n u, c t_n) \cap M = \emptyset$ for every $n\in \N$.
        Note that $M$ is not porous at $y_0\in M$ in the direction of $u=0$.

        If $V\subset Y$ then $M$ is called
        \emph{$V$-directionally porous}
        (or \emph{$V$-porous})
        if, for every $y_0 \in M$, there 
        are
        $u\in V$, $c>0$
        and a sequence $t_n \to 0$ of positive numbers
        such that
        $B(y_0 + t_n u, c t_n ) \cap M = \emptyset $ for every $n\in \N$.

        A set $M$ is \emph{$\sigma$-$V$-directionally porous}
        if $M=\bigcup_{i=1}^\infty M_i$
        where each $M_i$ is $V$-directionally porous.

        In the case $V=Y$ we speak about
        \emph{directionally porous sets} or
        \emph{$\sigma$-directionally porous sets}, respectively.
\end{definition}
        Note that the definitions of $\sigma$-directionally porous sets above and in~\cite[p.~4689]{MP2016}
        are slightly different; it is clear that any $\sigma$-directionally porous set in the sense of~\cite[p.~4689]{MP2016}
is
         $\sigma$-directionally porous in the sense of Definition~\ref{def:porous} above.
        If $Y$ is separable, the two notions are equivalent.

\begin{remark}\label{r:thickandporous}
                It is clear that
                        if a set
                $\Edomf\subset Y$ is not porous at $y\in \Edomf$ in the direction of $u\in Y$
                then
                $\Edomf$ is thick at $y$ in the direction of $u$.
                Recall
                that any non-empty set is thick at any of its points in the direction of $0$, therefore the statement
                holds for all $u\in Y$
                including the zero direction.
\end{remark}
\begin{remark}\label{r:derandporous}
        If $y\in \Edomf \subset Y$, $u\in Y$ and $d(x)=\dist(x,\Edomf)$ (for $x\in Y$)
        denotes the distance function,
        then
        $\Edomf$ is not porous at~$y$ in the direction of~$u$
        and
        $\Edomf$ is not porous at~$y$ in the direction of~$-u$
        if and only if
        $d$ is differentiable at~$y$ in the direction of~$u$
        (in the ordinary or Hadamard sense, equivalently).
        Recall that a set is not porous at any of its points in the direction of zero.
\end{remark}

\begin{remark}\label{r:thickandunique}
        If a set
        $\Edomf\subset Y$ is thick at $y\in \Edomf$ in the direction of $u\in Y$
        then there is at most one vector $z\in Z$ which is a 
Hadamard derivative of $f\fcolon \Edomf\subset Y\to Z$ at $y$ in the direction of $u$. Recall such $z$ is denoted $f'_\HforHadam(y;u)$ or simply $f'(y;u)$.
This works for both $u=0$ and $u\ne0$.

         If $\Edomf\subset Y$ is thick at $y\in \Edomf$ in the direction of $u\in Y$, $\Edomf\subset\Edomf_1\subset Y$
         and $f\fcolon \Edomf\subset Y\to Z$ happens to be a restriction of a function $g\fcolon\Edomf_1\to Z$ for which
         $g'_\HforHadam(y;u)$ exists then
         $f'_\HforHadam(y;u)$ also exists and
         $f'_\HforHadam(y;u) = g'_\HforHadam(y;u)$.
\end{remark}
\begin{remark}\label{r:thickandmap}
        If $X$ and $Y$ are Banach spaces, $e\in X$, $g\fcolon H \to Y$,
		where
        $H\subset X$ is thick at $x\in H$ in the direction of $e$,
        and Hadamard derivative $g'(x; e)$ exists,
        then $g(H)$ is thick at $g(x)$ in the direction of $g'(x; e)$.

        Indeed, by Definition~\ref{def:thick},
                there is sequence $x_n=x+t_n e_n \in H$ where $0\neq t_n\to 0$ and $e_n\to e$.
                Let $y_n=g(x_n)\in g(H)$.
                For $w_n := (y_n - g(x))/t_n$ we have
                $
                w_n
                \to g'(x;e)$
                and, of course,
                $g(x)+t_n w_n = y_n \in g(H)$.
                By Definition~\ref{def:thick} again,
                the sequence $y_n$ witnesses
                that $g(H)$ is thick at $g(x)$ in the direction of $g'(x;e)$.
\end{remark}

\bigbreak

In the present paper we consider differentiability of Lipschitz functions defined on subsets of a Banach space, 
and from the two different notions of differentiability, \Gateaux and Hadamard, we choose the Hadamard differentiability.
The results of~\cite{MP2016} then become a particular case of the more general framework developed by the present paper, cf.\ Remark~\ref{r:HadSameGat}.
\smallbreak

        We now introduce
        $(\delta,\omega)$-approximating
        Hadamard derived sets.

\begin{definition}\label{def:Hadamard-derived-set}
Let $Y$, $Z$ be Banach spaces,
$\Edomf\subset Y$, $f\fcolon \Edomf\to Z$.
For
$\delta>0$, $\omega>0$, $y\in \Edomf$, $v\in Y$
we let
\[
      \hadamdersetplus[\delta,\omega] f(y,v)
      =
      \left\{
                \frac{ f( y + t \hat u) - f( y ) }{ t }
        \setcolon
                t \in (0,\delta),
                \
                \hat u \in B(v, \omega)
		\text{ such that }
		y + t\hat u \in \Edomf
      \right\}
      .
      \eodhere
\]
\end{definition}
\begin{remark}\label{r:HDsetandder}
Note that  $z\in Z$ is a Hadamard derivative of $f$ at $y\in \Edomf$ in the direction of $v\in Y$ if and only if for any $\tau>0$ we have
$\hadamdersetplus[\delta,\omega]f(y,v)\subset B(z,\tau)$
and
$\hadamdersetplus[\delta,\omega]f(y,-v)\subset B(-z,\tau)$
for all sufficiently small $\delta,\omega>0$.
\end{remark}
\begin{remark}\label{r:derinHA}
	Note also that
	if $F$ is thick at $y$ in the direction of~$u$ and
	$f'_\HforHadam(x;u)$ exists, then
	$f'_\HforHadam(x;u) \in \closure { \hadamderset [\delta,\omega] f(y, u)}$
	for every $\delta>0$, $\omega>0$.
\end{remark}

\begin{remark}\label{r:HDMonot}
 If $0<\omega_1<\omega_2$
 and
 $0<\delta_1<\delta_2$
 then obviously
 \begin{align}
   \label{eq:HDMonotOmega}
      \hadamderset[\delta,\omega_1] f(y,v)
      & \subset
      \hadamderset[\delta,\omega_2] f(y,v)
   ,
   \\
   \label{eq:HDMonotDelta}
      \hadamderset[\delta_1,\omega] f(y,v)
      & \subset
      \hadamderset[\delta_2,\omega] f(y,v)
   .
 \end{align}
        If
        $\norm{u-\bar u} < \xi$,
        then
        $B(\bar u, \omega_1) \subset B(u, \omega_1 + \xi)$
        and therefore
\begin{align}
\label{eq:coneapprox}
                \hadamderset[\delta_1, \omega_1] f(y, \bar u)
                &\subset
                \hadamderset[\delta_1, \omega_1+\xi] f(y, u).
  \\[-1\baselineskip]
  \notag  &\qquad  \eodhere
\end{align}
\end{remark}
 \begin{remark}\label{r:HDlip}
 Let $f\fcolon\Edomf\subset Y\to Z$ be a $K$-Lipschitz function and $K>0$,
   then   
   $
   \hadamderset[\delta,\omega] f(y,v)
   \subset B(0_Z, (\norm{v}+\omega)K)$
   as the norm of $t\hat u$ in Definition~\ref{def:Hadamard-derived-set}
   is less than $(\norm{v}+\omega)t$.
   If, in addition, $0<\omega\le \norm{v}$, then the latter is a subset of
   $B(0_Z, 2K\norm{v})
   $. 
 \end{remark}

\begin{lemma}\label{l:HDbyHD}
Let $Y$, $Z$ be Banach spaces,
$\Edomf\subset Y$, $f\fcolon \Edomf\to Z$
be a $K$-Lipschitz function with $K>0$, $u , v \in Y$ and $r=\norm{v-u}$.

\begin{inparaenum}[\textup\bgroup(a)\egroup]
\item\label{HDbyHD:one}
        Assume
        that
        $\Edomf$
        is not porous at $y\in\Edomf$
        in the direction of $v$.
        Then, for every $\eta>0$,
        there is $\delta_0>0$ such that
        for every $\delta \in (0, \delta_0)$ and $\omega>0$,
        \begin{align}\label{eq:HDbyHD}
        \hadamdersetplus[\delta, \omega] f(y,u)
        \subset
        \hadamdersetplus[\delta, \eta] f(y,v) + B(0_Z,(r+\omega+\eta)K)
        .
        \end{align}

\item\label{HDbyHD:two}
If $\Edomf=Y$, then~\eqref{eq:HDbyHD} is satisfied for all positive $\delta,\omega,\eta$.
\end{inparaenum}
\end{lemma}
\begin{proof}
\itemref{HDbyHD:one}
        If the conclusion does not hold,
        there exists \(\eta
        >0
        \) such that for an
        arbitrarily small
        $\delta>0$,
        there is
        $\omega>0$,
        $t\in(0,\delta)$ and $\hat u\in B(u,\omega)$ such that
        \begin{equation}\label{eq:eq66}
        \frac{  f(y+t\hat u) - f(y) }{ t }
        \notin
        \hadamderset[\delta, \eta] f(y,v) + B(0_Z,(r+\omega+\eta)K)
        .
        \end{equation}
        As $\Edomf$ is not
        porous at $y$
        in the direction of~$v$,
        there exists a $\rho>0$ such that for all $t\in(0,\rho)$ the ball $y+tB(v,\eta)$ intersects $\Edomf$.
        Choose any $\delta\in(0,\rho)$ and let
        $\omega>0$,
        $t\in(0,\delta)$ and $\hat u\in B(u,\omega)$ be such that
 \eqref{eq:eq66} holds.
        Let $\hat v\in B(v,\eta)$ be such that $y+t\hat v\in \Edomf$.
        Notice that $\frac{f(y+t\hat v)-f(y)}{t}\in\hadamderset[\delta, \eta] f(y,v)$ by definition.
        On the other hand,
        \(
        \norm{ \hat u-\hat v }
        <
        \norm{ u-v }+\eta+\omega
        =r+\eta+\omega
        \), so
\[
\Bigl\|      \frac{  f(y+t\hat u) - f(y) }{ t }
  -     \frac{  f(y+t\hat v) - f(y) }{ t }\Bigr\|
        <
                (r+\eta+\omega)K,
\]
        which contradicts~\eqref{eq:eq66}.

\itemref{HDbyHD:two}
If $\Edomf=Y$, then the ball $y+tB(v,\eta)$ intersects $\Edomf$ for all positive $t$, so $\rho$ can be set to be~$\infty$.
\end{proof}

\begin{lemma}\label{l:HDbyDelta}
Let $Y$, $Z$ be Banach spaces,
$f\fcolon Y\to Z$ a $K$-Lipschitz function with $K>0$, $v\in Y$, $0<\delta_\indexONE<\delta_\indexTWO$ and $\omega>0$.
Then
\begin{equation*}
         \hadamdersetboth[\delta_\indexONE,\omega] f(y,v)
         \subset
         \hadamdersetboth[\delta_\indexTWO,\omega] f(y,v)
         \subset
         \hadamdersetboth[\delta_\indexONE,\omega] f(y,v)
         + B\bigl(0_Z, 2 K (\norm v + \omega) (1 - \delta_\indexONE/\delta_\indexTWO)
            \bigr)
         .
\end{equation*}
\end{lemma}
\begin{proof}
        The first inclusion follows by~\eqref{eq:HDMonotDelta}.
        To show the second inclusion, fix
        $z_\indexTWO \in \hadamdersetboth[\delta_\indexTWO,\omega] f(y,v)$. Then
        there is $\hat v \in B(v, \omega)$ and $t_\indexTWO \in (0,\delta_\indexTWO) $
        such that $z_\indexTWO = ( f(y+ t_\indexTWO \hat v) - f(y) ) / t_\indexTWO $.
        Let $t_\indexONE=t_\indexTWO \delta_\indexONE / \delta_\indexTWO \in (0, \delta_\indexONE)$
        and
        $z_\indexONE := ( f(y+ t_\indexONE \hat v) - f(y) ) / t_\indexONE$.
        Obviously, $z_\indexONE \in \hadamdersetboth[\delta_\indexONE,\omega] f(y,v)$.
        Let also $z_3 := \tfrac { t_\indexONE}{t_\indexTWO}  z_1$.
        We only have to estimate
        $\norm { z_\indexTWO - z_\indexONE } \le \norm { z_\indexTWO - z_3 }+ \norm {z_3 - z_\indexONE}
        \le
        K \norm { \hat v } ( t_\indexTWO - t_\indexONE ) /t_\indexTWO + ( 1 - \tfrac { t_\indexONE}{t_\indexTWO}) K \norm { \hat v }
        =
        2 K ( 1 - \tfrac { \delta_\indexONE}{\delta_\indexTWO}) \norm { \hat v }
        <
        2 K ( 1 - \tfrac { \delta_\indexONE}{\delta_\indexTWO}) (\norm v + \omega )
        $.
\end{proof}
\begin{lemma}\label{l:DerByDer}
Let $Y$, $Z$ be Banach spaces,
$\Edomf\subset Y$, $f\fcolon \Edomf\to Z$
be a Lipschitz function with constant $K>0$.
        Assume that for $i=1,2$, the vectors $u_i\in  Y$ and $y\in \Edomf$ are chosen so that
        the Hadamard derivatives $z_i=f'(y; u_i)$ exist and
        $\Edomf$
        is not porous at $y$
        in the direction of each of $u_i$.
        Then $\norm { z_1 - z_2 } \le K \norm{ u_1 - u_2 }$.
\end{lemma}
\begin{proof}
        Let $r=\norm{ u_1 - u_2 }$.
        Fix $\eta_0 > 0$.
        By Remark~\ref{r:HDsetandder}
        there is $\eta\in (0,\eta_0)$ such that,
        for all $\delta,\omega \in (0, 2\eta)$ and $i=1,2$,
        \begin{equation}\label{eq:55}
        \hadamdersetboth[\delta, \omega] f(y,u_i) \subset  B(z_i, \eta_0)
        .
        \end{equation}
        Let $\delta_0$ be as in Lemma~\ref{l:HDbyHD}.
        Let $\delta=\omega=\min(\eta, \delta_0)<\eta_0$.
        By the non-porosity assumption there is some $\hat y = y + t \hat u \in \Edomf$
        with $t\in (0, \delta)$ and $\hat u \in B(u_1, \omega)$.
        Hence $\hat z_1 \coloneq (f(\hat y) - f(y) )/ t \in
        \hadamderset[\delta, \omega] f(y,u_1) \subset  B(z_1, \eta_0)
        $.
        By~\eqref{eq:HDbyHD},
        $\hat z_1 \in   \hadamderset[\delta, \eta] f(y,u_2) +  B(0_Z, (r+2\eta)K )
        $.
        Using~\eqref{eq:55} for $i=2$
        this set is contained in $B(z_2, \eta_0) +  B(0_Z, (r+2\eta_0)K ))$.
        Thus $\norm{z_2 - z_1}
        \le \bignorm{ z_2 - \hat z_1 } + \bignorm{\hat z_1 - z_1 }
        \le \eta_0 + (r+2\eta_0)K + \eta_0$.
        Since $\eta_0$ was arbitrary,
        we have $\norm{z_2 - z_1} \le K r$.
\end{proof}

        The following
        lemma is motivated by \cite[Lemma~2.3]{MZ}
        which followed \cite[Lemma~4.4]{Duda} and \cite[p.~518--519]{Bongiorno}.

\begin{lemma}\label{lem:MZ}
        Let $\spaceY$
        and $\spaceZ$
        be Banach spaces and
        $p_0\in P\subset \Edomf \subset \spaceY$, $u\in \spaceY$, $\rho>0$, $K>0$. Let
        $f\fcolon \Edomf \to \spaceZ$ and
        $f^*\fcolon P \to \spaceZ$ be functions satisfying $f|_P=f^*$ and
        $\norm{f(\pointy)-f(p)}\le K\norm{\pointy - p}$ whenever $p\in P$ and $\pointy\in \Edomf$
        with $\norm{\pointy - p}<\rho$
        and $\norm{p-p_0} < \rho$.
        Assume that both
        $\varphi(\pointy)=\dist(\pointy,P)$ and $f^*$ are Hadamard differentiable at $p_0$ in the direction of $u$.

        Then $f$ is Hadamard differentiable at $p_0$ in the direction of $u$ and $f'_{\HforHadam}(p_0;u)=(f^*)'_{\HforHadam}(p_0;u)$.
\end{lemma}
\begin{proof}
         Since $\varphi$
         is non-negative and Hadamard differentiable at $p_0$ in the direction of $u$, we have $\varphi'_\HforHadam(p_0; u) =0$.

         Let $z=(f^*)'_\HforHadam(p_0; u)$.
         Let $\varepsilon >0$ be arbitrary, then,
         using the definition of Hadamard derivative as per Remark~\ref{rem:Hadamard3},
         find $\delta,\omega > 0$
         such that
         \(
        \delta \norm{u} +
		\delta\omega<\rho\) and,
         for all $t\in (-\delta, 0)\cup (0,\delta)$ and $p \in P \cap B(p_0 + t u, \abs t \omega)$,
         \[
          \norm { \frac{ f^*(p) - f^*(p_0) }{ t } - z } <  \varepsilon/2
          .
         \]
         Since
         $\varphi'_\HforHadam(p_0; u) =0$ and $\varphi(p_0) = 0$,
         there are
         $\delta_1\in (0,\delta)$, $\omega_1 \in (0,\omega/2)$
         such that,
         for all $t\in (-\delta_1, 0)\cup (0,\delta_1)$ and $\hat y \in B(p_0 + t u, \abs t \omega_1)$,
\begin{equation}\label{eq:MZ:por}
         \abs { \frac { \dist(\hat y, P) }{t} } = \abs { \frac{ \varphi(\hat y)  } { t } } < \min( \omega /4 , \varepsilon / (4K))
         .
\end{equation}

         To show that $f'_\HforHadam(p_0; u) = z$, let
         $t\in (-\delta_1, 0)\cup (0,\delta_1)$,
         $\hat y \in \Edomf \cap B(p_0 + t u, \abs t \omega_1 )$.
         Choose $p\in P$ with
                 \( \norm { \hat y - p }
                 \le
                 2 \dist(\hat y, P)
         \).
         Then, using~\eqref{eq:MZ:por},
         \[
                 \norm{\hat y - p}
                 \le2\dist(\hat y,P)               
                 < \abs t \omega / 2
                 <\rho.
         \] 
         This implies
         $p \in P\cap B(p_0 + t u, \abs t (\omega/2 + \omega _1) )  \subset P\cap B(p_0 + t u, \abs t \omega) $.
         Hence
         $ \norm { \frac{ f^*(p) - f^*(p_0) }{ t } - z } <  \varepsilon/2 $.
         Since
         $\norm{\hat y-p}<\rho$
         and
         $\norm{p - p_0}<\rho$,
         we get
         \(
              \norm{ f(\hat y) - f^*(p) } = \norm{ f(\hat y) - f(p) } \le K \norm { \hat y - p }\le 2K\dist(\hat y,P)
              < \abs t \varepsilon / 2
         \),
         using~\eqref{eq:MZ:por} again. Finally, we use $f(p_0)=f^*(p_0)$ to conclude
         \[
          \norm { \frac{ f(\hat y) - f(p_0) }{ t } - z } <  \varepsilon
         \]
         for all $t\in (-\delta_1, 0)\cup (0,\delta_1)$ and $\hat y \in \Edomf \cap B(p_0 + t u, \abs t \omega_1)$.
\end{proof}
\begin{lemma}\label{lem:MZ:derext}
        Let $\spaceY$
        and $\spaceZ$
        be Banach spaces and
        $p_0\in P\subset \Edomf \subset \spaceY$.
        Assume $f\fcolon \Edomf \to \spaceZ$ is a Lipschitz function.
        If $u \in \spaceY$, $(f|_P)'_\HforHadam(p_0; u)$ exists
        and $P$ is not porous at~$p_0$ in the directions of~$\pm u$,
        then $f'_\HforHadam(p_0; u)$ also exists
        and $f'_\HforHadam(p_0; u)=(f|_P)'_\HforHadam(p_0; u)$.
\end{lemma}
\begin{proof}
        Note that
        the distance function $\dist(\cdot, P)$
        is Hadamard differentiable at~$p_0$ in the direction of~$u$
        since $P$ is not porous at~$p_0$ in the directions of~$\pm u$,
        cf.\ Remark~\ref{r:derandporous}.
        The conclusion then follows by Lemma~\ref{lem:MZ}.
\end{proof}

\section{Sigma-ideals $\lone$ and $\mathcal L$}
Here we look at sigma-ideals of exceptional sets which arise naturally in the study of differentiability of partial functions and of Lipschitz functions on the whole space. Main results of this paper do not depend on this section. However, it allows us to see how the notions of `almost everywhere' compare between the context of the present work, \cite{PZ2001} and~\cite{MP2016}.  

We start with the following notion introduced in
        \cite[p.~23]{PZ2001}, see also \cite[Section~1]{MP2016}.

\begin{definition}\label{def:Lnull}
        Let $X$ be a separable Banach space. 
    By $\mathcal L=\mathcal L(X)$ we denote
the $\sigma$-ideal generated by sets of points of \Gateaux non-differentiability of Lipschitz functions $f$ defined on the whole of $X$ with values in a Banach space with Radon-Nikod\'ym property.
\end{definition}
        Note that obviously $\mathcal L\subset \mathcal L_1$,
        cf.\ Definition~\ref{def:L1null}.
        
\begin{remark}\label{r:porousAreNull}
        Note that every
        directionally porous set
        belongs to $\mathcal L\subset \mathcal L_1$.
        Indeed, let $N \subset X$ be a directionally porous set
        and let $g(\cdot)=
        \dist(\cdot,N)$
        be
        the distance function
		defined on $X$.
        If $x\in N$
        and $N$ is porous at $x$ in the direction of $u\in X$ then by Remark~\ref{r:derandporous},
        $g$
        is not differentiable at $x$ in the direction of~$u$.
        Therefore also every $\sigma$-directionally porous subset of~$X$ belongs to $\mathcal L\subset \mathcal L_1$.
        
        The following consideration when $S\subset X$ is an arbitrary set will be useful. Let
        \[
        N_S=\{x\in S\setcolon \text{ there exists }u\in X,
        \text{ such that } S \text{ is porous at } x 
        \text{ in the direction of }u\}.
        \]
        Then $N_S$ is obviously a directionally porous set, hence $N_S\in\lone(X)$.        
\end{remark}

        The following lemma shows that the definition of $\mathcal L_1$ may be replaced with an equivalent one, similar to but not exactly the same as the one used for $\mathcal L$.
        The question whether $\mathcal L_1(X)=\mathcal L(X)$ for all separable Banach spaces $X$, is open, compare with a similar question in~\cite[p.~518]{Bongiorno}.
\begin{lemma}\label{l:equiv-def-l1}
        Let $X$ be a separable Banach space.
        Then
        the $\sigma$-ideal $\mathcal L_1(X)$ coincides with the $\sigma$-ideal generated by sets $N\subset X$ such that there is a Banach space $Y$ with the Radon-Nikod\'ym property
        and a pointwise Lipschitz function $g\fcolon X\to Y$ such that for each $x\in N$, $g$ is not
        Hadamard differentiable at $x$.

        The same is true if the pointwise Lipschitz functions $g$ are allowed to 
        be defined on an arbitrary set~$G$ satisfying $N\subset G\subset X$. 
\end{lemma}
\begin{proof}
        Let us denote the two new $\sigma$-ideals defined in the statement by $\mathcal L_{1,G=X}$ and $\mathcal L_{1,G\subset X}$, respectively.
        It is clear that
        that larger class of functions defines larger class of exceptional sets, hence
        $\mathcal L_{1,G=X}\subset\mathcal L_{1,G\subset X}$.

We show first that $\mathcal L_{1,G\subset X}\subset\lone$. 
Assume $N$ is one of the sets generating the $\sigma$-ideal $\mathcal L_{1,G\subset X}$.
Let
$g$ and $G=\dom(g)$ be as in the statement of the lemma. 
Let 
\[
L_k=\{x\in
G
\setcolon
\norm{ g(x) } \le k
\text{ and }
\norm{g(x)-g(y)}\le k\norm{x-y}
\text{ for every } y\in
G\cap
B(x,1/k)\}
\]
for each $k\ge1$.
As $g$ is pointwise Lipschitz on $G$, we get that $\bigcup_{k\ge1}  L_k=G$. Moreover, $g|_{L_k}$ is $2k^2$-Lipschitz.
Define
\begin{align*}
L_{k,\text{por}}&=\{x\in L_k\setcolon \text{ there is } v\in X\text{ such that }L_k \text{ is porous at }x \text{ in the direction of }v\},\\
L_{k,\text{np}}&=L_k\setminus L_{k,\text{por}}.
\end{align*}
For each $k\ge1$, we have $
        N\cap L_{k,\text{por}}
\subset
L_{k,\text{por}}
\in\mathcal L_1$ as it is directionally porous, see Remark~\ref{r:porousAreNull}.
We will show
that 
also
$N\cap L_{k,\text{np}}\in \mathcal L_1$
which implies
$N\cap L_k\in\mathcal L_1$ for each $k\ge1$, hence $N\in\mathcal L_1$.

        Fix $k\ge1$.
        To prove that \(
        N
        \cap L_{k,\text{np}}\in \mathcal L_1\),
        we only need to show that, at every point of $N\cap L_{k,\text{np}}$,
        function $g|_{L_k}$ is not Hadamard differentiable, as $g|_{L_k}$ is Lipschitz on its domain $L_k$.
        So assume that there exists $x\in N\cap L_{k,\text{np}}$ where $g|_{L_k}$ is Hadamard differentiable.
        Then for every $v\in X$, $(g|_{L_k})'_\HforHadam (x;v)$ exists
        and depends linearly on~$v$.
		Fix an arbitrary $v\in X$.
        Note that
        $L_k$ is not porous at~$x$ in the direction of~$\pm v$.
        By a standard argument, this implies that $g'_\HforHadam (x;v)$ exists
        and $g'_\HforHadam (x;v) = (g|_{L_k})'_\HforHadam (x;v)$,
        cf.\ Lemma~\ref{lem:MZ} and Remark~\ref{r:derandporous}.
        Since this holds for every $v$
        and the dependence is linear (and continuous), we obtain that $g$ is Hadamard differentiable at $x$, a contradiction to $x \in N$.

        Thus, for every $Y$ with the Radon-Nikod\'ym property,
        every $G\subset X$ containing $N$
        and every pointwise Lipschitz function $g\fcolon G\to Y$, subsets of the set of points of Hadamard non-differentiability
        of~$g$ are in $\mathcal L_1$.

        \smallbreak

        We now show $\lone\subset\mathcal L_{1,G=X}$. 
        Assume $N$ is one of the sets generating the $\sigma$-ideal $\lone$,  as in the Definition~\ref{def:L1null},
        and let $G$ and $g$
        be the corresponding subset of $X$ containing $N$
        and Lipschitz function $g\fcolon G\to Y$.
        Then
        $g$ can be extended to a pointwise Lipschitz function $h$ defined on the whole space,
        by first extending $g$ to a Lipschitz function on~$\closure G$ and
        using
        \cite[Theorem~3.1~(iii),~(vi)]{KK}
        with $L=0$ and with the
        help of 
        \cite[Lemma~2.5(e) and Remark~2.6(a)]{KK}.
        For every $x\in N$, $g$ is not Hadamard differentiable at~$x$, hence the same holds true for its extension~$h$.
        Therefore, $N \in \mathcal L_{1,G=X}$.
\end{proof}

        In the following proposition,
        we consider one-sided Hadamard
        directional derivatives, 
        which are defined similarly to the ordinary
        one-sided derivatives,
        by replacing the 
        restriction on $t$
        in Definition~\ref{def:Hadamard2}
        with $t\in (0,\delta)$,
        so only positive~$t$
        are allowed.

        Note that if~$G$ is an open set, or if~$x$ is an interior point of~$G$,
        and $f$ is a Lipschitz mapping
        defined on~$G$
        then the existence of one-sided Hadamard directional derivative of~$f$
        at~$x$ would follow from the
        existence of one-sided directional derivative of~$f$ at~$x$.
        Hence Proposition~\ref{p:dir-had2} is a generalisation of \cite[Theorem~2,
        Theorem~5]{PZ2001}.

\begin{proposition}\label{p:dir-had2}
        Let $X$ be a separable Banach space and $G\subset X$ any subset.
        Let $f$ be a Lipschitz function from $G$ to a Banach space $Y$.

        Then there is a $\sigma$-directionally porous $A_0\subset G$
        such that
        the following holds at every point $x\in G\setminus A_0$:

\begin{inparaenum}[\textup\bgroup(1)\egroup]
\item\label{p:dir-had2:first}
 The set $W_x$ of those directions $u\in X$ in which the one-sided
 Hadamard derivative $f^\primeplus_\HforHadam(x;u)$ exists
 is a closed linear subspace of~$X$. Moreover, the mapping
 $u\mapsto f^\primeplus_\HforHadam(x;u)$
 is linear (and continuous) on~$W_x$.

\item\label{p:dir-had2:second}
        If the
        one-sided
        Hadamard directional derivative
        $f^\primeplus_\HforHadam(x;u)$
        exists in all directions~$u$ from a set $U_x\subset X$
        whose linear span is dense in~$X$,
        then $f$ is Hadamard differentiable at~$x$.
\end{inparaenum}
\end{proposition}

\begin{remark*}
        The proof below is inspired by the proof of \cite[Theorem~2]{PZ2001}.
        We  notice that 
        to obtain also a generalisation of \cite[Theorem~5]{PZ2001}
        it
        is
        enough to replace their ``$u+v$'' by ``$v-u$''
        and add observation~\eqref{eq:had-der2}.
        Such an argument is much shorter than 
        the one employing~\cite[Lemma~3]{PZ2001}.
\end{remark*}
\begin{proof}[Proof of Proposition~\ref{p:dir-had2}]
        We can assume that $Y$ is separable, because otherwise we can replace~$Y$ by the closed linear span of~$f(G)$.
        We may also suppose that $\Lip f > 0$.

        Let $G_0$ be the set of those $x\in G$ for which there is $u\in X$
        such that $G$ is porous at~$x$ in the direction of~$u$.
        Note that $G_0$ is a directionally porous set.
        Let $G_1=G\setminus G_0$.

        For any $u,v \in X$, $y,z \in Y$, $\varepsilon, \delta > 0$, denote
        by $A(u,v; y,z; \varepsilon,\delta)$ the set of all $x\in G_1$
        such that
\begin{inparaenum}[(i)]
\par
\item\label{item:dir-had2:first}
        for every
        $t\in (0,\delta)$,
        $\hat u \in B(u,\delta)$
        and $\hat v \in B(v,\delta)$
  \begin{align*}
        \norm{ f( x + t \hat u ) - f(x) - t y } &\le \varepsilon t
        \qquad \text{if $x + t \hat u \in G$, and}
        \\
        \norm{ f( x + t \hat v ) - f(x) - t z } &\le \varepsilon t
        \qquad \text{if $x + t \hat v \in G$}
        ,
  \end{align*}
\par
\item\label{item:dir-had2:second}
        for every $\eta>0$, there are
        $t\in (0,\eta)$
        and $\hat w \in B(
        v - u
        , \eta)$,
        such that
  \begin{align}\label{eq:second2}
        x + t\hat w \in G
        \qquad \text{and} \qquad
        \norm{ f( x + t \hat w ) - f(x) - t (
        z - y
        ) } > 4 \varepsilon t
        .
  \end{align}
\end{inparaenum}

        Let $u,v\in X$, $y,z\in Y$ and $\varepsilon,\delta>0$ be fixed,
        let $\delta_0=\min(\delta/2, \varepsilon/ \Lip f)$.
		        Assume that $x \in A(u,v; y,z; \varepsilon,\delta)$.
        Since $G$ is not porous at~$x$ in any direction,
        there is $\delta_1>0$ such that, for every $t\in (0,\delta_1)$,
\begin{equation}\label{eq:had-der2}
        G \cap
        B( x + t v,  t \delta / 2 )\ne\emptyset.
\end{equation}
        We will show that
        $A(u,v; y,z; \varepsilon,\delta)$
        is porous at~$x$ in the direction of~$(v-u)$.
        It is enough to prove that
         there are arbitrarily small
         $t>0$
        for which
\begin{equation}\label{eq:had2:por}
        B( x + t
        ( v - u)
        ,  t \delta_0
        ) \cap A(u,v; y,z; \varepsilon,\delta) = \emptyset.
\end{equation}
Let $\eta\in(0,\min(\delta_0, \delta_1))$.
Then by~\itemref{item:dir-had2:second} there are
$t\in (0,\eta)$ and $\hat{w}\in B(v-u,\eta)\subset B(v-u,\delta_0)$ which satisfy~\eqref{eq:second2}.

Assume that the intersection~\eqref{eq:had2:por} is not empty;
let        $q \in  B( x + t (v-u) , t \delta_0 ) \cap A(u,v; y,z; \varepsilon,\delta)$.
        We may
        use~\eqref{eq:had-der2} to find
        \(
        p \in
        G \cap
        B( x + t v,  t \delta / 2 )
        \),
                then
        $\hat v \coloneq (p-x) / t$ lies in $B(v,\delta)$.
		Let
        \(\hat u \coloneq
        (
        p
        - q
        ) / t\);
        as $\delta_0<\delta/2$, we have $\hat u \in B(u,\delta)$.
        Using $p=q+t\hat u$
		and $p=x+t\hat v$,
        we conclude that
\begin{align*}
        \rlap{$ \norm { f(q + t \hat u) - f(q) - t y } $} \qquad &
        \\
        &\ge
          \norm{
                  f( p )
                  -
                  f( x + t \hat w)
                  - t y
          }
          -
          \norm{
               f(q)
               -
               f( x + t \hat w)
               }
        \\
        &\ge
        \Bigl(
          \norm{
                  f( x + t \hat w )
                  - f(x) - t ( z - y )
          }
          -
          \norm{
                  f(p)
                  - f(x) - t z }
        \Bigr)
          -
          \Lip f\, \norm{  x + t \hat w - q }
        \\
        &>
          \varepsilon t
        .
\end{align*}
        This contradicts~\itemref{item:dir-had2:first}
        as required by
        $ q \in  A(u,v; y,z; \varepsilon,\delta)$.

        Define $A$ as the union of all the sets $A(u,v; y,z; \varepsilon,\delta)$ obtained
        by choosing $u$, $v$ from a dense countable subset $U$ of~$X$,
        $y$, $z$ from a dense countable subset $V$ of~$Y$,
        and rational numbers $\varepsilon, \delta>0$.
        Then $A$ is $\sigma$-directionally porous.

        Let $x\in G_1\setminus A$.
        Now we will show that
        $W_x\subset X$, the set of all directions $u\in X$
        such that
        one-sided
        Hadamard directional derivative
        $f^\primeplus_\HforHadam(x;u)$
        exists, is a linear subspace of $X$.
        For $s>0$ and $u_0 \in W_x$,
        we have
        $f^\primeplus_\HforHadam(x; s u_0) = s f^\primeplus_\HforHadam(x;u_0)$
        and so $s u_0 \in W_x$.

      Assume that
        $u_0, v_0 \in W_x$. 
        For
        any rational $\varepsilon >0$
        find rational $\delta \in (0, \varepsilon / \Lip f)$ such that
        for all $t\in (0,\delta)$,
        $\hat u_0 \in B(u_0, 2 \delta)$,
        $\hat v_0 \in B(v_0, 2 \delta)$
        \begin{align*}
         \norm { f(x + t \hat u_0) -f (x) - t f^\primeplus_\HforHadam(x; u_0) } &\le \varepsilon t / 2
         \qquad\text{if $x + t \hat u_0 \in G$, and}
         \\
         \norm { f(x + t \hat v_0) -f (x) - t f^\primeplus_\HforHadam(x; v_0) } &\le \varepsilon t / 2
         \qquad\text{if $x + t \hat v_0 \in G$.}
        \end{align*}
        Furthermore, pick $y, z \in V$ and $u, v \in U$ such that
        $ \norm { y - f^\primeplus_\HforHadam(x; u_0) } < \varepsilon / 2$,
        $ \norm { z - f^\primeplus_\HforHadam(x; v_0) } < \varepsilon / 2$,
        $ \norm { u - u_0 } < \delta $
        and
        $ \norm { v - v_0 } < \delta $.
        Observe that~\itemref{item:dir-had2:first} holds since,
        for example
        \begin{align*}
                \rlap{$
                        \norm { f(x+ t \hat u) -f(x) - t y }
                $} \qquad &
                \\
                &\le
                \norm { f( x + t \hat u ) - f(x) - t f^\primeplus_\HforHadam(x; u_0) }
                + \norm { t y - t f^\primeplus_\HforHadam(x; u_0) }
                \\
                &\le
                \varepsilon t / 2
                + \varepsilon t / 2
                = \varepsilon t,
        \end{align*}
        whenever $t\in (0,\delta)$,
        $\hat u \in B(u, \delta) \subset B(u_0, 2\delta) $
        and $x+ t \hat u \in G$.
        As~\itemref{item:dir-had2:first} holds true
        and $x\notin A(u,v; y,z; \varepsilon, \delta)$,
        there is $\eta\in(0,\delta)$ such that, for all $t\in (0,\eta)$
        and $\hat w \in B(v-u, \eta)$,
  \begin{equation}\label{eq:NOTsecond2}
        \norm{ f( x + t \hat w ) - f(x) - t ( z - y ) }
        \le
        4 \varepsilon t
        \qquad \text{if }
        x + t\hat w \in G
        .
  \end{equation}
        This means that,
        for all $t\in (0,\eta]$,
  \begin{equation}\label{eq:NOTsecond3}
        \hadamderset[t,\eta] f(x, v-u)
        \subset
        B( z - y , 4 \varepsilon )
        .
  \end{equation}
        By Lemma~\ref{l:HDbyHD}\itemref{HDbyHD:one},
        there is $\delta_2\in (0,\eta)$ such that
        for every $t\in(0,\delta_2)$
        and $\bar w \in B(v_0-u_0, \delta_2)$
        such that $ x + t \bar w \in G$,
  \begin{align}\notag
        \frac{ f( x + t \bar w ) - f(x) }{ t } 
        &\in
        \hadamderset[t,\eta] f(x, v-u)
        + B(0_Y, (\norm { v - u -(v_0-u_0) } + \delta_2 + \eta)\, \Lip f)\\
        \notag &\subset 
        \hadamderset[t,\eta] f(x, v-u)
        + B(0_Y,
        4\delta \Lip f)
		\subset B(z-y,4\varepsilon+4\delta\Lip f)
        ,
  \end{align}
        where the last inclusion follows from~\eqref{eq:NOTsecond3}.
        This implies that,
        for every $t\in(0,\delta_2)$
        and $\bar w \in B(v_0-u_0, \delta_2)$
        such that $ x + t \bar w \in G$,
  \begin{align*}
        \norm{ f( x + t \bar w ) - f(x) - t ( z - y ) }
       &\le
        4 \varepsilon t + 4 t \delta \Lip f
        \le
        8 \varepsilon t
  \end{align*}
  and hence
  \begin{equation*}
        \norm{ f( x + t \bar w ) - f(x) - t ( f^\primeplus_\HforHadam(x;v_0)-f^\primeplus_\HforHadam(x;u_0) ) }
        \le
        9 \varepsilon t
        .
  \end{equation*}
        Since $\varepsilon$ was arbitrary,
        we proved that
        $f^\primeplus_\HforHadam(x;v_0 - u_0)$ exists and is equal
        to $f^\primeplus_\HforHadam(x;v_0)-f^\primeplus_\HforHadam(x;u_0)$
        and therefore also $v_0-u_0 \in W_x$.
        Note that a particular case is when $v_0=0$ and then
        the conclusion is that
        $-u_0 \in W_x$
        and
        $f^\primeplus_\HforHadam(x;-u_0)=-f^\primeplus_\HforHadam(x;u_0)$
        whenever
        $x\in G_1\setminus A$ and
        $u_0$ lies in $W_x$.

        Since we proved that $v-u \in W_x$ and $s u \in W_x$ whenever $u, v\in W_x$ and $s>0$, we conclude that $W_x$ is a linear subspace of $X$.

To see that $W_x$ is closed and $f^\primeplus_\HforHadam(x)$ is continuous on $W_x$, suppose $\bar v\in\closure{W_x}$.
For $n\in \N$, let $v_n \in  W_x$ be such that $v_n\to\bar v$.
Let $y_n$ be the
(one-sided)
Hadamard derivative $f^\primeplus_\HforHadam(x)(x; v_n)$.
The sequence $\{v_n\}\subset W_x$ is Cauchy and hence, by Lemma~\ref{l:DerByDer},
$\{y_n\}$ is also Cauchy. Hence there is $\bar y\in Y$ such that $y_n\to \bar y$.

Fix $\eta_0>0$.
We can fix an $m$ with $\norm{y_m - \bar y} < \eta_0$
and $\norm{v_m - \bar v} < \eta_0$.
Then we find $\eta_1 \in (0,\eta_0)$ such that,
        for all $\delta,\omega \in (0, 2\eta_1)$,
        \begin{equation}\label{eq:45:copy}
        \hadamderset[\delta, \omega] f(x,v_m) \subset  B(y_m, \eta_0) .
        \end{equation}
Applying Lemma~\ref{l:HDbyHD} we find $\delta_3\in (0,2\eta_1)$ such that,
for every $\delta, \omega \in (0,\delta_3)$,
        \[
        \hadamderset[\delta, \omega] f(x,\bar v)
        \subset
        \hadamderset[\delta, \eta_1] f(x,v_m) + B(0_Y,(\norm{v_m - \bar v}+\omega+\eta_1)\Lip f).
        \]
Combining that with~\eqref{eq:45:copy} and $\norm{y_m - \bar y} < \eta_0$
we obtain
        \[
        \hadamderset[\delta, \omega] f(x,\bar v)
        \subset
        B(y_m, \eta_0) + B(0_Y,(\norm{v_m - \bar v}+\omega+\eta_1)\Lip f )
        \subset
        B(\bar y, \eta_0 + 3 \eta_0 \Lip f+\eta_0 )
        \]
for every $\delta, \omega \in (0, \delta_3)$.
Since
$\eta_0>0$ was arbitrary, this proves that the
one-sided
Hadamard derivative of $f$ at $x$ in the direction of $\bar v$ exists and is equal to $\bar y$, and
        $f^\primeplus_\HforHadam(x;\cdot)$ is continuous at $\bar v$.

        The exceptional set $A\cup G_0$ is a $\sigma$-directionally porous set.
\end{proof}

        The following generalisation of Proposition~\ref{p:dir-had2} to the case of pointwise Lipschitz functions can be proved using Lemma~\ref{lem:MZ}.

\begin{proposition}\label{p:dir-had4}
        Let $X$ be a separable Banach space and $G\subset X$ any subset.
        Let $f$ be a pointwise Lipschitz function from $G$ to a Banach space $Y$.
        Then there is a $\sigma$-directionally porous $A\subset G$
        such that
        \itemref{p:dir-had2:first}
        and
        \itemref{p:dir-had2:second}
        of Proposition~\ref{p:dir-had2}
        hold
        at every point $x\in G\setminus A$.
\end{proposition}
\begin{proof}
        For each $k\ge1$, let
        \[
                L_k=\{x\in
                G
                \setcolon
                \norm{ f(x) } \le k
                \text{ and }
                \norm{f(x)-f(y)}\le k\norm{x-y}
                \text{ for every } y\in
                G\cap
                B(x,1/k)\}.
        \]
        As $f$ is pointwise Lipschitz on $G$, we get that $\bigcup_{k\ge1}  L_k=G$. Moreover, $f|_{L_k}$ is $2k^2$-Lipschitz.
        Let
        $A_k\subset L_k$
        be the exceptional set from Proposition~\ref{p:dir-had2} for $f|_{L_k}\fcolon L_k \to Y$.

        Let $L_{k,\text{por}}$ be the set of all points $x\in L_k$ such that $L_k$ is porous at~$x$ in some direction $v\in X$
        and let $A= \bigcup_{k\ge 1} ( A_k \cup L_{k,\text{por}} )$.

       Fix an arbitrary $x\in G \setminus A$. Fix $k\ge 1$ such that $x\in L_k$.
       Let
       $W_x$ denote the set of those directions $u\in X$ in which the one-sided
       Hadamard derivative $f^\primeplus_\HforHadam(x;u)$
       exists.
       Denote by
       $W_x^*$ the set of those directions $u\in X$ for which
       $(f|_{L_k})^\primeplus_\HforHadam(x;u)$
       exists.
       By Proposition~\ref{p:dir-had2}, $W_x^*$
       is a closed linear subspace of~$X$ and the mapping $u\mapsto (f|_{L_k})^\primeplus_\HforHadam(x;u)$
       is linear and continuous on~$W_x^*$.
       Obviously, $W_x \subset W_x^*$.
       If $u\in W_x^*$, then
       $-u\in W_x^*$,
       and by Lemma~\ref{lem:MZ}
       we obtain
       $\pm u\in W_x$
       and $f^\primeplus_\HforHadam(x;\pm  u) = (f|_{L_k})^\primeplus_\HforHadam(x;\pm u)$
       because $L_k$ is not porous at $x\in L_k \setminus  L_{k,\text{por}}$ in the direction of~$\pm u$,
       cf.\ also Remark~\ref{r:derandporous}.
       Therefore, $W_x = W_x^*$ is a closed linear subspace and
       $u\mapsto f^\primeplus_\HforHadam(x;u)$
       is linear (and continuous) on~$W_x$.
\end{proof}

\begin{lemma}\label{l:equiv-def-l1-C}
        Let $X$ be a separable Banach space.
        Then the following $\sigma$-ideals coincide with $\lone$:
\begin{inparaenum}[\textup\bgroup(a)\egroup]

\par\noindent
\item\label{item:equiv:L1*}
        the $\sigma$-ideal
        $\lone^{*}$
        generated by sets $N\subset X$ for which there is a set $G\subset X$ containing $N$, a Banach space $Y$ with
        the Radon-Nikod\'ym property
        and a Lipschitz function $g\fcolon G\to Y$ such that for each
        $x\in N$,
        there exists $u\in X$ such that
        $g$ is not
        Hadamard differentiable at~$x$
        in the direction of~$u$.

\par\noindent
\item\label{item:equiv:L1**}
        the $\sigma$-ideal
        $\lone^{**}$
        generated by sets $N$
        such that there a Banach space $Y$ with the Radon-Nikod\'ym property and a pointwise Lipschitz function $g\fcolon X\to Y$
        such that for each $x\in N$, $g$ is not
        Hadamard differentiable at $x$.

\par\noindent
\item\label{item:equiv:L1***}
\end{inparaenum}
        the $\sigma$-ideal
        $\lone^{***}$
        generated by sets $N$
        such that there a Banach space $Y$ with the Radon-Nikod\'ym property and a pointwise Lipschitz function $g\fcolon X\to Y$
        such that for each $x\in N$,
        there exists $u\in X$ such that
        $g$ is not
        Hadamard differentiable at~$x$
        in the direction of~$u$.
\end{lemma}
\begin{proof}
Note that $\sigma$-directionally porous sets belong to each of the $\sigma$-ideals $\lone$, $\lone^*$, $\lone^{**}$, $\lone^{***}$, cf. Remarks~\ref{r:porousAreNull} and~\ref{r:derandporous}.
  
        The equality $\lone=\lone^{**}$ follows from Lemma~\ref{l:equiv-def-l1}.
        As Hadamard differentiability includes the requirement that the Hadamard derivatives in all directions exist,
        the inclusions $\lone^{*}\subset \lone$ and $\lone^{***}\subset\lone^{**}$ are obvious.
        The reverse inclusions $\lone\subset \lone^{*}$
        and $\lone^{**}\subset\lone^{***}$
        follow 
        from~Proposition~\ref{p:dir-had4}.
\end{proof}

\section{Hadamard derivative assignments. Chain Rule}
        In this section we define the notion of Hadamard derivative assignment which, for a
        Lipschitz function $f\fcolon\Edomf\subset Y\to Z$ defined on a set,
        plays a role of its (Hadamard) derivative and the linear space in the
        direction of which the derivative is taken, see
        Definition~\ref{def:assignment}. We then continue to explain when an
        assignment is complete and Borel measurable, see
        Definitions~\ref{def:completeAssignment} and~\ref{def:assignment-meas}.
        Later we show in Sections~\ref{sec:Uphi},~\ref{s:step3}
        and~\ref{s:Borel} that for any 
		Lipschitz
         $f$ there exists a Hadamard
        derivative assignment which is complete and Borel measurable at the
        same time, under some additional assumptions on the Banach spaces $Y$
        and $Z$, 
        and, in Section~\ref{sec:PointLip}, construct a complete Borel
        measurable Hadamard derivative assignment for a pointwise Lipschitz
        function.

		In the present section, we also define the composition of
        derivative assignments in Definition~\ref{def:circ} and state the Chain
        Rule formula in Proposition~\ref{p:chain-orig}.

\begin{definition}[Hadamard derivative assignment]\label{def:assignment}
        Let $Y$, $Z$ be Banach spaces, $\Edomf\subset Y$ and
        $f\fcolon \Edomf \to Z$ a function.
        A~\emph{domain assignment for~$f$}
        is a map~$U$
        defined on~$\Edomf$,
        such that
        for every $y\in \Edomf$, we have that
        $U(y)$ is a closed linear subspace of~$Y$
        (including the possibility of the trivial subspace $\{0_Y\}$)
        satisfying
        the following three properties:
\begin{enumerate}
\item
\label{def:assignment:item:notporous}
        for every $y\in \Edomf$ and $u\in U(y) \setminus \{0_Y\} $,
        $\Edomf$ is thick at $y$ in the direction of $u$,
\item
\label{def:assignment:item:hadamDer}
        for every $y\in \Edomf$ and $u\in U(y) \setminus \{0_Y\} $,
        Hadamard derivative
        $L_y(u) \coloneq f'(y;u)$
        exists,
\item
\label{def:assignment:item:contlin}
        $L_y:U(y)\to Z$ is a continuous linear operator if we additionally define $L_y( 0_Y )=0$.
\end{enumerate}
        An~\emph{assignment pair for~$f$}
        is a pair $(U, f^\bd)$
        where $U$ is as above and,
        for every $y\in \domf$,
        \[f^\bd(y) = L_y \fcolon U(y) \to Z\] is
        the continuous linear map
        defined in~\itemref{def:assignment:item:hadamDer}.
        A~\emph{derivative assignment for~$f$}
        is $f^\bd$
        from the above assignment pair.

        We will call any of the three objects defined above
        a \emph{Hadamard derivative assignment for~$f$}.
Also, we will sometimes refer to the domain assignment as a \emph{Hadamard} (\emph{derivative}) domain assignment.
        Note that while an assignment pair clearly defines both domain and derivative assignments,
        the reverse is also true, cf.~Remark~\ref{r:reconstruct}.
\end{definition}
Note that this definition generalises 
\cite[Definition~4.1]{MP2016}.
Indeed, in the special case when $f$ is a Lipschitz function defined on the whole space $\Edomf=Y$,
the two definitions agree.
To see that, recall that in this case the notions of Hadamard and (ordinary) directional derivatives agree.
The following remarks discuss various aspects relating to  Definition~\ref{def:assignment} and its conditions \itemref{def:assignment:item:notporous}--\itemref{def:assignment:item:contlin}.

\begin{remark}\label{r:reconstruct}
        If a derivative assignment $f^\bd$
        for a function $f\fcolon \Edomf\subset Y \to Z$
        is given,
        then $(U,f^\bd)$, where $U(y) = \dom f^\bd(y)$ for every $y\in \Edomf$,
        is an assignment pair for $f$.
        If $U$ is a domain assignment
        for a function $f\fcolon \Edomf\subset Y \to Z$
        then
        the corresponding derivative assignment $f^\bd$
        is uniquely defined by
        \itemref{def:assignment:item:hadamDer}
        using
        Remark~\ref{r:thickandunique}
        and
        \itemref{def:assignment:item:notporous}.

        Therefore, a \emph{Hadamard derivative assignment} for~$f$
        may indeed be given
        interchangeably
        in the form of the assignment pair $(U, f^\bd)$, the domain assignment~$U$ or the derivative assignment~$f^\bd$.
\end{remark}
\begin{remark}\label{r:item-notporous2}
        The specific assignments $\UPhi f$ and $\EAPhi f$ that we define later in Definitions~\ref{def:assignmentPhi} and~\ref{def:assignmentEA} 
        satisfy the following condition stronger than \itemref{def:assignment:item:notporous}, see
        Lemma~\ref{l:nporousHA}: 
        for every $y\in \Edomf\coloneq\dom f$ and $u\in \UPhi f(y) \setminus \{0\}$ or $u\in \EAPhi f(y) \setminus \{0\}$, $\Edomf$ is not porous
        at~$y$ in the direction of~$u$.
        The same condition is satisfied by $\regtanex \Edomf$ from~\eqref{eq:regtanex:def} of
        Definition~\ref{def:assignmentPhi}
        and by the composition of derivative assignments from Definition~\ref{def:circ}. See also Lemma~\ref{l:nporous} and Corollary~\ref{c:restriction}.
\end{remark}
\begin{definition}[Complete assignment]\label{def:completeAssignment}
Let $Y$ be a Banach space.  We say that
        an assignment (that is, a map)
        $U\fcolon \Edomf\subset Y \to 2^Y$
        is \emph{complete}
        if
        for every separable Banach space $X$,
        every set $G\subset X$
        and every Lipschitz function $\testfn\fcolon G\to Y$
        there is an
        $N\in\mathcal L_1(X)$
        such that
\begin{equation}\label{eq:propofN}
\text{
        $\testfn'(x; u) \in U(\testfn(x))$
        whenever $x\in \testfn^{-1}(\Edomf) \setminus N$, $u\in X$ and
        $\testfn'(x;u)$ exists,
}
\end{equation}
        where $\testfn'(x; u) = \testfn'_\HforHadam(x; u)$ denotes the Hadamard derivative.
        
Let $Z$ be a Banach space.
        We say that a Hadamard derivative assignment $f^\bd$
        for $f\fcolon\Edomf\subset Y\to Z$ is \emph{complete} 
        if its Hadamard domain assignment $y\mapsto U(y)=\dom f^\bd(y)$ is complete.
\end{definition}
\begin{remark}\label{r:complete-general}
        Similarly to the Definition~\ref{def:domination-general}, we may define
        a more general notion of an assignment, \emph{complete with respect to}
        the classes $(\spaces,\tests)$ of spaces and test mappings, and an
        abstract differentiability operator $\diffop$. In particular, we may
        replace the class of Lipschitz test mappings as used in
        Definition~\ref{def:completeAssignment} by the wider class of all
        pointwise Lipschitz functions, as we will do in Remark~\ref{r:testPointwise}. 
\end{remark}

\begin{remark}\label{r:completeDomin}
Note that
        an assignment $U\fcolon \Edomf\subset Y \to 2^Y$
        is complete
        if and only if
	\[
           D \coloneq
           \{ (y,u) \setcolon y\in \Edomf,\ u \in U(y) \}
	\]
	is $\mathcal L_1$-dominating in $\Edomf \times Y$, 
       cf.\ Definition~\ref{def:domination}.
        If $U\fcolon \Edomf\subset Y \to 2^Y$
        is a complete derivative assignment for some function $f$
        then
	$D$ is linearly $\mathcal L_1$-dominating in $\Edomf \times Y$.
\end{remark}
For the rest of the paper, we mostly use the complete derivative assignments rather than referring to the notion of $\lone$-domination.
\begin{remark}\label{r:cap_complete}
If assignments $U_1,U_2\fcolon \Edomf\subset Y \to 2^Y$
are complete, then $U_1\cap U_2$ defined by
\[
(U_1\cap U_2)(y)\coloneq U_1(y)\cap U_2(y),
\qquad 
y\in\Edomf 
\]
is complete, too.
	
Indeed, let $N_i\in\lone(X)$ be such that~\eqref{eq:propofN} is satisfied with $U=U_i$, $i=1,2$. Then $N=N_1\cup N_2\in\lone(X)$ and~\eqref{eq:propofN} is satisfied for $U=U_1\cap U_2$. 	
\end{remark}
\begin{remark}\label{r:restriction:complete}
Let $U\fcolon \Edomf\subset Y\to 2^Y$ be a complete assignment and $H\subset \Edomf$.
Let $\tilde U(y)\coloneq U(y)$ for $y\in H$. Then $\tilde U\fcolon H\to 2^Y$ is a complete assignment too.
\end{remark}
\begin{remark}\label{r:completeAssignmentcorr}
The notion of completeness of an assignment does not change if in Definition~\ref{def:completeAssignment},
        we additionally require
        that the set
        $N$ contains all points of directional porosity of $G	$
        and that,
        for any $x\in G\setminus N$ and any direction $u$,
        there is at most one possible value for the Hadamard derivative $\testfn'(x;u)$.

        Indeed,
    let $N'\subset G$ be the set of points $x\in G$ such that
    there is $v\in X$ for which
    the Hadamard derivative of $\testfn$ at $x$ in the direction of~$v$ has more than one
    possible
    value, and $N''\subset G$ be the set of points $x\in G$ at which $G$ is directionally porous. 
    By Remarks~\ref{r:thickandporous} and~\ref{r:thickandunique} we conclude that $N'\subset N''$. Since $N''\subset G$, we conclude $N''$ is directionally porous.
    By Remark~\ref{r:porousAreNull},
    $N'' \in \mathcal L\subset \lone$,
    hence we are free to assume in the
Definition~\ref{def:completeAssignment} that $N$ is large enough to contain 
$N''\supset N'$. 
\end{remark}
\begin{remark}\label{r:grestriction}
        The notion of completeness in Definition~\ref{def:completeAssignment} also stays the same if
        we require in addition that 
        all functions $\testfn$ considered
there satisfy $\testfn(G) \subset \domf$.

        Indeed, every $\testfn$ satisfying this additional requirement is of course
        considered in the original definition.
        Conversely, let now $\testfn\fcolon G \subset X\to Y$ be any function considered in 
        Definition~\ref{def:completeAssignment} with range in $Y$.
        Let $\testfn_1$ be the restriction of $\testfn$ to $G_1\coloneq\testfn^{-1}(\Edomf) \subset G$.
        Then the values of $\testfn_1$ lie in $\Edomf$, hence $\testfn_1$ is among the functions
        considered for the modified definition.
        Let
        $N\in \lone$
        be the
        null
        set corresponding to $g_1$,
        which moreover contains all points at which $G_1$ is directionally porous (see Remark~\ref{r:completeAssignmentcorr}).
	If now $x\in
        \testfn^{-1}(\Edomf) \setminus N
        =
        G_1 \setminus N$, $u\in X$ and the Hadamard derivative $\testfn'(x;u)$ exists,
     then the Hadamard derivative  $\testfn_1'(x;u)$ exists and is equal to $\testfn'(x;u)$ (here we refer to Remarks~\ref{r:thickandporous} and~\ref{r:thickandunique}). We then have
     $\testfn'(x;u)=\testfn_1'(x;u)\in U(\testfn_1(x))=U(\testfn(x))$ as Definition~\ref{def:completeAssignment}
is satisfied for $\testfn_1$ and $N$.     
\end{remark}
\begin{remark}\label{r:testPointwise}
        In Definition~\ref{def:completeAssignment},
        the requirement
        that the test mapping $g$ is Lipschitz can be replaced
        by requiring $g$ to be pointwise Lipschitz, i.e.\ the class $\tests_{L}$ of Lipschitz functions may be replaced by the class of pointwise Lipschitz functions $\tests_{pL}$, see Remark~\ref{r:complete-general}.
		We show now that the complete assignments we get for this wider class of test mappings are the same as those defined by Definition~\ref{def:completeAssignment}. 
        First, if $U$ is complete  with respect to $\tests_{pL}$,
        then it is complete
		with respect to smaller class $\tests_L$.

        Assume now that $U$ is complete with respect to $\tests_L$.
        Let $X$ be a separable Banach space
        and let \(g\in \tests_{pL}(X,Y)\), i.e., $g\fcolon G\subset X \to Y$ is a pointwise
        Lipschitz function.
        For $n\in \N$, define
        \(  G_n = \{ x\in X \setcolon
                     \norm{ g(y) - g(x) } \le n \norm { y - x }
                     \text{ for every } y \in G\cap B(x, 1/n) \}
        \). 
        Using 
        the
        separability of $X$,
        find sets $H_{n,k}\subset X$ such that $\bigcup_k H_{n,k} = X$ and
        $\diam H_{n,k} < 1/n $ for every $k\in \N$.
        For $n, k\in \N$, let $G_{n,k} = G_n \cap H_{n,k}$ and
        let $g_{n,k}$ denote the restriction of $g$ to $G_{n,k}$.
        Then $g_{n,k}$ is a Lipschitz function
        and if $g'_\HforHadam(x;u)$ exists, $x\in G_{n,k}$ and $G_{n,k}$ is not directionally porous at $x$,
        then also,
        by Remarks~\ref{r:thickandporous} and~\ref{r:thickandunique},
        $(g_{n,k})'_\HforHadam(x;u)$ exists and $(g_{n,k})'_\HforHadam(x;u)=g'_\HforHadam(x;u)$.
        Since $g_{n,k} \in \tests_L$, there exists a set
        $N_{n,k} \in \mathcal L_1(X)$ such that,
        whenever $x\in (g_{n,k})^{-1}(\Edomf) \setminus N_{n,k}$, $u \in X$ and $(g_{n,k})'_\HforHadam(x;u)$
        exists, we have $(g_{n,k})'_\HforHadam(x;u) \in U(g_{n,k}(x)) = U(g(x))$.
        By Remark~\ref{r:completeAssignmentcorr}, we can assume that $N_{n,k}$ contains
        all points of directional porosity of $G_{n,k}$.
        Letting $N=\bigcup_{n,k} N_{n,k} \in \mathcal L_1(X)$, we conclude that
        \eqref{eq:propofN} of Definition~\ref{def:completeAssignment} is satisfied
        for the pointwise Lipschitz function $g$ that we started with.
\end{remark}

\begin{remark}\label{r:testGeneral1}
        In Definition~\ref{def:completeAssignment},
        the requirement
        that the test mapping $g$ is Lipschitz can be removed if additional 
        condition $x\in \domlip g$
        is added to the right hand side of~\eqref{eq:propofN}.
        The complete assignments we get for this definition are the same,
        and the proof remains the same as in Remark~\ref{r:testPointwise},
        because $\domlip g \subset \bigcup_n G_n$.
\end{remark}

\begin{remark}\label{r:lip-ptwslip}
        The statements of Remarks~\ref{r:completeAssignmentcorr},
        \ref{r:grestriction} and~\ref{r:testPointwise} remain valid for the
        notion of $\mathcal E$-dominating set of
        Definition~\ref{def:domination} if the $\sigma$-ideal $\mathcal E$
        contains all directionally porous sets. The proofs are identical to
        those given in the respective Remarks.
\end{remark}	
\begin{remark}\label{r:Yindependent}
	Strictly speaking, in Definition~\ref{def:completeAssignment}
        an
        assignment
	should be called
	\emph{complete with respect to $Y$}.
	However, the notion does not depend on $Y$.
	That is: the Hadamard derivative assignment $U$ is complete with respect to $Y$ if and only if
	$U$ is complete with respect to $Y_1$, for any Banach spaces $Y$, $Y_1$ that contain $\domf$
        and the sets $U(y)$, $y\in F$
	and whose norms are equivalent
	on the linear span of $\domf$.
	Indeed, the modified (though  equivalent) definition suggested in Remark~\ref{r:grestriction}
	does not depend on the particular choice of $Y$.
	For the same reason,
	if
	$\mathcal E$ contains all directionally porous sets,
	if $\domf \subset Y\cap Y_1$ and the norms of $Y$, $Y_1$ are equivalent on the linear span of $\domf$,
	then a set $D\subset \domf \times (  Y\cap Y_1 )$ is $\mathcal E$-dominating in~$\domf \times Y$
	if and only if it is  $\mathcal E$-dominating in~$\domf \times Y_1$.
\end{remark}

\begin{remark}\label{r:completeMPxKM}
        Our definition of complete
        derivative
        assignment differs from the one of \cite[Definition~4.1]{MP2016} mainly
        in that it applies to functions $f$ defined not on the whole space $Y$.

        Besides that,
        our definition of complete assignments is
        also
        possibly more restrictive than the one of \cite[Definition~4.1]{MP2016},
        because we allow $\testfn$ to be defined on a subset of $X$, even if $\Edomf=Y$.
        On the other hand, \cite{MP2016} requires $\mathcal L$-null sets as exceptional sets while we allow them to be $\mathcal L_1$-null,
        which possibly does not ensure
        they are $\mathcal L$-null.
        See also Lemma~\ref{l:equiv-def-l1}.
\end{remark}

The following examples illustrate some problems one faces when differentiating Lipschitz functions defined on subsets of Banach spaces,
and justify the requirement~\itemref{def:assignment:item:notporous} in Definition~\ref{def:assignment}.
Here we consider the space $\R^3$ with the standard basis $\{e_1,e_2,e_3\}$ and let $\R^2=\Span\{  e_1,e_2\}$ and $\R=\Span\{e_1\}$.
\begin{example}\label{e:composition1}
        Consider the following functions:
        $h\fcolon \R \to \R^2$, $h(x) =(x,x^2)$
        and
        $f\fcolon M \subset \R^2 \to \R$,
        where $M= \{ (x,0) \setcolon x\in \R \}$ and $f(x,0)=3x$.
        Then we can calculate the following Hadamard derivatives:
        $h'(0; e_1) = e_1 \in \R^2$,
        $f'((0,0); e_1) = 3e_1 \in \R$.
        So we can expect $(f\circ h)'(0;e_1) = 3e_1$
        at least in the sense that that $3e_1$ is a ``perfect'' candidate for this derivative.
        However $f \circ h$ is a function defined only at a single point $0\in \R$,
        so
        Definition~\ref{def:assignment}\itemref{def:assignment:item:notporous}
        forbids ``differentiating'' of $f\circ h$ at $0$ in the direction of $e_1$.
\end{example}
\begin{example}\label{e:composition2}
        Adding a new coordinate to the previous example, we let
        $h\fcolon \R^2 \to \R^3$, $h(x,z) =(x,x^2,z)$
        and
        $f\fcolon M \subset \R^3 \to \R$,
        where $M= \{ (x,0,z) \setcolon x,z\in \R \}$ and $f(x,0,z)=3x+5z$.
        Again,
        Definition~\ref{def:assignment}\itemref{def:assignment:item:notporous}
        forbids ``differentiating'' in $f\circ h$ at $(0,z)$ in the direction of $e_1$.
        However, we have to allow for differentiating in the direction of $e_2$:
        If $U$ is a complete Hadamard derivative domain assignment of $f\circ h$,
        then testing by $\testfn\fcolon \R \to \R^2$,
        \(
        \testfn(
        z
        )=(0,z)
        \) we see
        that \(
        e_2 \in U(\testfn(
        z
        ))
        \) for almost every \(
        z
        \in \R
        \).
\end{example}

\bigbreak

        We also take
        care about measurability of derivative assignments which we define
        presently in the spirit of~\cite{MP2016}, see
        also~\cite[Example~5.4]{MP2016}.
\begin{definition}\label{def:assignment-meas}
        Let $Y,Z$ be Banach spaces, $\Edomf\subset Y$, $f\fcolon\Edomf\to Z$ a function and
        $(U,f^\bd)$
        an assignment pair for $f$.
        We say that 
        $(U,f^\bd)$
        is
        \emph{Borel measurable}
        if
\begin{enumerate}[(i)]
\item\label{def:AM:item:G1}
        $S=\left\{(y,u)\setcolon y\in \Edomf,\ u\in U(y)\right\}$ is  Borel in $\Edomf\times Y$, and 
\item\label{def:AM:item:M}
        the map $(y,u)\mapsto f'_\HforHadam(y;u)=f^\bd(y)(u)$ is a Borel map on $S$.
         \eodhere
\setcounter{saveenum}{\value{enumi}}
\end{enumerate}
\end{definition}
\begin{remark}\label{rem:assMeasNB}
        If $\Edomf \subset Y$ is not Borel, we use the relative topology,
        so that, for example, $M\subset \Edomf \times Y$ is Borel in $\Edomf \times Y$
        if and only
        if
        there is a Borel set $B \subset Y\times Y$ such that $M=B\cap ( \Edomf \times Y )$.
        This however is not an important case because
        if~$f$ is a Lipschitz function then $f$
        can be always extended to a closed set 
        $\overline F$.
\end{remark}
\begin{remark}\label{rem:MiG2}
        If $Z$ is separable,
        then
        conditions~\itemref{def:AM:item:G1}
        and~\itemref{def:AM:item:M} 
        of Definition~\ref{def:assignment-meas}
        imply that
\begin{enumerate}[(i)]
\setcounter{enumi}{\value{saveenum}}
\item\label{def:AM:item:G2}
$P=\left\{(y,u,f^\bd(y)(u))\setcolon y\in \Edomf,\ u\in U(y)\right\}$ is Borel  in $\Edomf\times Y\times Z$.
\end{enumerate}
        Indeed, by~\cite[Proposition~12.4]{Kechris}, 
        the
        set $P$ is a Borel subset of $S\times Z$.
\end{remark}
\begin{remark}\label{rem:assMeas}
\begin{inparaenum}[(a)]
\item\label{rem:assMeas:item:G1iG2}
        If
        $Y$ and $Z$ are separable
        (hence standard Borel spaces,
        see~\cite[Definition~12.5]{Kechris})
        and
        $\Edomf$ is a Borel subset of $Y$
        then
        condition~\itemref{def:AM:item:G1} required in
        Definition~\ref{def:assignment-meas} follows
        from~\itemref{def:AM:item:G2}
        of Remark~\ref{rem:MiG2}.
        Indeed,
        by~\itemref{def:AM:item:G2},
        $P$
        is a Borel subset of
        $Y\times Y\times Z$
        and
        by \cite[Theorem~18.10]{Kechris},
        as every section 
        $P_{(y,u,\cdot)}$
        of
        $P$
        is a singleton or an empty set,
        we have that
        its projection on the standard Borel space
        $Y \times Y$
        is Borel.

\item\label{rem:assMeas:item:G1G2iM}
        We also note that if
        $Y$, $Z$
        are separable,
        $\Edomf$ is a Borel subset of $Y$
        and \itemref{def:AM:item:G1} is satisfied, then
        $S$ is a standard Borel space
     by~\cite[Corollary~13.4]{Kechris};
hence
        by~\cite[Theorem~14.12]{Kechris},
        the conditions~\itemref{def:AM:item:M} and~\itemref{def:AM:item:G2} 
        are equivalent.

\par

\item\label{rem:assMeas:item:PARTiM}
        By Lemma~\ref{l:HadamMeas},
        \itemref{def:AM:item:M} 
        of Definition~\ref{def:assignment-meas} 
        is true for any separable $\Edomf\subset Y$, any
        continuous function $f\fcolon \Edomf \to Z$ and any Hadamard derivative assignment $(U,f^\bd)$
        for~$f$.
            This is because
            in~\itemref{def:AM:item:M}, the derivative is restricted to a set where we have the thickness property and the derivative exists,
            i.e., a subset of $M$ defined by~\eqref{eq:defMBorelon} in  Lemma~\ref{l:HadamMeas}.

\item\label{rem:assMeas:item:SUM}
        If
        $Y$ and $Z$ are separable,
        $\Edomf$ is a Borel subset of $Y$
        and $f$ is Lipschitz
        then
        \itemref{def:AM:item:M} is always true
        (cf.~\itemref{rem:assMeas:item:PARTiM})
        and
        \itemref{def:AM:item:G1}
        is equivalent to \itemref{def:AM:item:G2},
        see~\itemref{rem:assMeas:item:G1G2iM}.
        Hence in the particular case $\Edomf = Y$,
        Definition~\ref{def:assignment-meas} agrees with the definition
        of Borel measurable assignment of~\cite{MP2016} (see page 4717,  following~\cite[Example~5.4]{MP2016}) 
        where condition \itemref{def:AM:item:G2} 
        of Remark~\ref{rem:MiG2}
        is used.
\end{inparaenum}
\end{remark}

The following definition allows us 
  to talk about measurability of an assignment
        with no apparent function associated. In this case the assignment
        usually originates from the distance function to a subset of separable
        Banach space, in which case
        Remark~\ref{rem:assMeas}\itemref{rem:assMeas:item:PARTiM} applies.
\begin{definition}\label{def:assignment-meas-nof}
        Let $Y$ be a separable Banach space.  We say that
        an assignment (that is, a map)
        $U\fcolon \Edomf\subset Y \to 2^Y$
        is
        \emph{Borel measurable}
        if
        $S=\left\{(y,u)\setcolon y\in \Edomf,\ u\in U(y)\right\}$ is  Borel in $\Edomf\times Y$.
\end{definition}

        We now state and prove
        another
        topological fact about Borel measurability; we use it
        later when we check the Borel measurability of the composition
        of Hadamard derivative assignments, see
        Proposition~\ref{p:chain}.

\begin{lemma}\label{l:pair0}
   Let $D$ be a topological space and
   $Y$ a separable Banach
   \textup(or a second countable topological\textup) space.
   If $h,L\fcolon D \to Y$ are Borel measurable, then $(h,L)\fcolon
   D
   \to Y
   \times Y
   $ is Borel measurable, too.
\end{lemma}
\begin{proof}
  Preimages of open rectangles
  $(h,L)^{-1}(G_1\times G_2) = h^{-1}(G_1) \cap L^{-1}(G_2)$
  are obviously Borel.
  Using the second countability of $Y$, the same is true for every open set.
  The proof is then finished by the induction over the Borel hierarchy.
  See also \cite[Theorem 31.VI.1, p.~382]{Kuratowski}.
\end{proof}

        The following
        example
        shows that sometimes the freedom of choice of the domain assignment values $U(x)$ could lead to
        assignments
        that are not Borel-measurable.

\begin{example}\label{e:composition3}
        Using the cylindrical coordinates,
        let $h\fcolon \R^3 \to \R^2$
        be defined by
        $h(r\cos \alpha, r\sin\alpha,t) = (r^2 \sin 2\alpha,t)$
        ($r\ge 0, \alpha \in \R$).
        Hence $h(0,0,t)=(0,t)$.
        Furthermore, let
        $f\fcolon M \subset \R^2 \to \R$,
        where $M= \{ (0,t) \setcolon t\in \R \}$,
        be defined by $f(y,t)=5t$.

        Then we can calculate the following Hadamard derivatives:
        $h'((0,0,t); a e_1 + b e_2 + c e_3) = c e_2 \in \R^2$ for every $a,b,c\in \R$,
        $f'((0,t); c e_2) = 5 c e_1$.
        So we can expect $(f\circ h)'((0,0,t);a e_1 + b e_2 + c e_3 ) = 5 c e_1$
        at least in the sense that $5 c e_1$ is a ``perfect'' candidate for this derivative.
        However $f \circ h$ is a function defined only on the set
        $
        D\coloneq
        \{(r\cos \alpha, r\sin\alpha,z) \setcolon \sin 2\alpha = 0 \}
        =
        \{ (x,y,z) \setcolon x,y,z\in \R,\ xy=0 \}
        $
        and
        Definition~\ref{def:assignment}\itemref{def:assignment:item:notporous}
        forbids ``differentiating''
        of
        $f\circ h$ at $(0,0,t)$ in
        directions other than those in $D$.
        Assuming $U$ to be a complete Hadamard derivative domain assignment for $f\circ h$,
        let us explore what should be the values of $U(0,0,t)$ for $t\in \R$.

        The domain of the composition of the linear maps that represent the Hadamard derivatives of $h$ and $f$
        was noted to be $\R^3$ which would suggest $U(0,0,t) = \R^3$.
        However, $U(0,0,t) \subset D$ is required by Definition~\ref{def:assignment}\itemref{def:assignment:item:notporous}.
        Choosing one of
        \begin{equation}\label{eq:Dsubset}
         \{ (x,y,z) \setcolon x,z\in \R,\ y=0 \},
         \qquad
         \{ (x,y,z) \setcolon y,z\in \R,\ x=0 \}
        \end{equation}
        for $U(0,0,t)$ suffers from ambiguity --- if we choose from the two options randomly,
        we cannot hope $U$ would be Borel measurable.
        The choice $U(0,0,t)=\{0\}$ ($t\in \R$) leads to an assignment that is not complete,
        as can be seen using test mapping $\testfn(t) = (0,0,t)$, $t\in \R$.
        Besides a clever ad-hoc choice from~\eqref{eq:Dsubset} we are left with one remaining possibility:
        $U(0,0,t)=\{(0,0,z)\setcolon z\in \R\}$.
        And this is the choice selected in~\eqref{eq:BDcomposition}
        and
        \eqref{eq:Ucomposition} below.
\end{example}

 When defining a Hadamard derivative assignment for a composition $f\circ h$,
 we naturally compose the linear mappings that represent the derivatives of $h$ and $f$.
 However, Example~\ref{e:composition3} suggests we need to restrict the domain
 of this composition
 to a suitable subset of the contingent (tangent, cf.\ Remark~\ref{r:contingent}) cone of $\dom(f\circ h)$.
 \begin{definition}\label{def:aregTan}
 Let $Y$ be a
 Banach space and $M$ a subset of $Y$.
 An assignment $R \fcolon M \to 2^Y$ is called
 a \emph{regularized tangent for~$M$}
 if
 \begin{enumerate}
 \item\label{def:aregTan:cllin}
 for every $y\in M$, $R(y)$ is a closed linear subspace of $Y$,
 \item\label{def:aregTan:nporous}
 for every $y\in M$ and $u\in R(y)$, $M$ is not porous at~$y$ in the direction of~$u$,
 \item \label{def:aregTan:complete}
         If $M$ is separable
         then
 $R$ is complete,
 \item\label{def:aregTan:Borel}
         If $M$ is separable
         then
         $R$ is Borel measurable on $M$ in the sense
          of Definition~\ref{def:assignment-meas-nof}.
 \eodhere
 \end{enumerate}
 \end{definition}

Throughout the rest of this section and in Section~\ref{sec:chain_rule}, we assume that $\regtan M$ is a regularized tangent for~$M$,
and write $\regtan_y M$ for $\regtan M(y)$, $y\in M$.

\begin{definition}\label{def:circ}
Let $X$, $Y$, $Z$ be Banach spaces, $H\subset X$, $\domf\subset Y$,  
        $h\fcolon H\to Y$ and $f\fcolon \domf \to Z$
        be arbitrary functions, and
        let
        $\domainHzeroWasOne = \dom (f\circ h) = h^{-1}(\domf)$.

        If
        $h^\bd$
        and
        $f^\bd$
        are
        Hadamard derivative assignments
        for $h$ and $f$, respectively, then 
        we define
        \begin{equation}\label{eq:BDcomposition}
            (f^\bd \circ h^\bd )
            (x)
          =
            f^\bd(h(x)) \circ  h^\bd(x)
\circ \identity_{ \textstyle  \regtan_x \domainHzeroWasOne }\ ,
            \qquad x\in \domainHzeroWasOne
            ,
        \end{equation}
and call it \emph{the composition of derivative assignments $f^\bd$ and $h^\bd$}.
Here the composition with the identity operator on $\regtan_x \domainHzeroWasOne $ in the right-hand side of~\eqref{eq:BDcomposition} restricts the natural domain of $f^\bd(h(x)) \circ  h^\bd(x)$ to a linear subspace of $\regtan_x \domainHzeroWasOne$.
\end{definition}
\begin{remark}\label{r:circ}
        Notice that the derivative assignment $f^\bd\circ h^\bd$ as defined above depends not only on $f^\bd$ and $h^\bd$ but also on
        the inner function $h$.
        The set $\domainHzeroWasOne$ can be obtained as
        $\domainHzeroWasOne=\dom (f\circ h)
        =
        h^{-1}(\dom f^\bd)
        $.

        When the Hadamard derivative assignments $f^\bd$, $h^\bd$ are given in the form of domain assignments $U^f$, $U^h$, we
        use $U^f \circ U^h$ to denote the domain assignment corresponding to $f^\bd\circ h^\bd$.
        Although dependence on $h$ is suppressed by this notation, the function
        still needs to be specified.
        The formula for $U^f \circ U^h$ is 
        \begin{align}\label{eq:Ucomposition}
(U^f \circ U^h) (x)
          &=
           \dom
            (
              (f^\bd \circ h^\bd )
              (x)
            )
         \\
         \notag
          &=
                \bigl
                \{
                                u \in U^h(x)
                              \setcolon
                                h^\bd(x) (u)
                                \in U^f(h(x))
                \bigr
                \}
           \cap
                \regtan_x \domainHzeroWasOne
         \\
         \notag
          &=
                \bigl
                \{
                                u \in U^h(x)
                              \setcolon
                                h'_\HforHadam(x;u) \in U^f(h(x))
                \bigr
                \}
           \cap
                \regtan_x \domainHzeroWasOne
                ,
                \qquad
                x\in \domainHzeroWasOne.
           \eodhere
        \end{align}
\end{remark}

We now state one of the main propositions of this paper.
Together with the results of Theorem~\ref{t:MP52} and Proposition~\ref{p:Borel1}, 
which prove the existence of complete and measurable derivative assignments, 
and 
 Definition~\ref{def:assignmentPhi}
        and
Theorem~\ref{thm:disttan}
which ascertain that $\regtanex \domainHzeroWasOne$ is an example of a regularized tangent for $\domainHzeroWasOne$,
Proposition~\ref{p:chain-orig} provides an `almost everywhere' Chain Rule formula for
pointwise
Lipschitz functions defined on subsets of Banach spaces. 
In the context of finite-dimensional spaces and their subsets, the
requirements relating to separability and the Radom-Nikod\'ym property in all above statements and in
Proposition~\ref{p:chain-orig} are automatically satisfied.

\jkrestatable{P:chain}{%
        Let
        $X$, $Y$, $Z$ be Banach spaces,
        $H\subset X$, $\domf\subset Y$,
        $h\fcolon H\to Y$ and $f\fcolon \domf \to Z$
		functions.
        Let $\domainHzeroWasOne = \dom (f\circ h) = h^{-1}(\domf)$.

        Let
        $f^\bd$
        and
        $h^\bd$
        be
        Hadamard derivative assignments
        for $f$ and $h$, respectively.
        Then

\begin{inparaenum}[\textup\bgroup(1)\egroup]
\par\noindent
\item\label{p:chain:deras}
        the composition $f^\bd \circ h^\bd$
        is a Hadamard derivative assignment
        for $f\circ h$.

\par\noindent
\item\label{p:chain:complete}
        If $H_0$ is separable,
        $h$ is
        pointwise
        Lipschitz and
$f^\bd$ and $h^\bd$ are complete, then the composition $f^\bd \circ h^\bd$
        is a complete Hadamard derivative assignment
        for $f\circ h$.

\par\noindent
\item\label{p:chain:measurable}
        If both
		$H_0$
                and $h(H_0)$ are separable,
		and if
        the function
        $h|_{H_0}\fcolon H_0\to Y$
        and the derivative assignments
        $f^\bd$ and $h^\bd$
        are Borel measurable, then $f^\bd \circ h^\bd $
        is Borel measurable too.
\end{inparaenum}
}
\begin{proposition}[Chain Rule for Hadamard derivative assignments, Proposition~\ref{p:chain}]\label{p:chain-orig}
\restateWithLPrefix{P:chain}{Nintro}
\end{proposition}
We prove this proposition later, in Section~\ref{sec:chain_rule}.
We note here that Proposition~\ref{p:chain-orig} is valid independently of the choice of a regularized tangent $\regtan \domainHzeroWasOne$ used in~\eqref{eq:BDcomposition}. 
\begin{remark}\label{rem:completeNotEQ}
  
        For a Lipschitz function $f\fcolon \R^n \to \R^n$, let $U^f$
        denote a complete derivative domain assignment for $f$.
Then
        $U^{h_1\circ h_2}$ is a complete
        derivative domain
        assignment
        for any two Lipschitz functions $h_1,h_2\fcolon \R^n \to \R^n$.
        Also, the composition
        $U^{h_1} \circ U^{h_2}$ is a complete derivative assignment by
        Proposition~\ref{p:chain-orig}.

It would be desirable to define a complete derivative assignment so that 
$U^{h_1} \circ U^{h_2}$ and $U^{h_1\circ h_2}$ agree.
       However, we  show that this 
is not possible even when $n=1$.
 		More precisely, we show that
          there is no operator $f\mapsto U f$ 
          assigning to
          each
          Lipschitz real function~$f$ on the real line
          a complete derivative domain assignment $U f$
          (for $f$) satisfying,
          for all Lipschitz functions $f,h\fcolon \R\to\R$,
        \begin{equation}\label{eq:circEQ}
                U f \circ  U h = U ( f \circ h)
          .
        \end{equation}
        Indeed, assume $U$ is
        such an operator.
        Then
        $U (\identity_\R )$ is required to be complete, where $\identity_\R$ is the identity function.
        Testing the completeness (cf.~Definition~\ref{def:completeAssignment})
        using the
        Lipschitz
         function $g\fcolon \R\to\R$, $g(t)=t$,
        we see that $U ( \identity_\R )  (x_0) = \R$ for almost every $x_0\in \R$.
        To obtain a contradiction, we will show that, for every $x_0\in \R$,
        $U ( \identity_\R )  (x_0) = \{ 0 \}$.

        Fix $x_0\in \R$.
        Then
        functions
        $ h(x) = ( 1 + \chi _ { \{ x > x_0 \} } ) ( x - x_0 ) $
        and $ f(y) = x_0 + ( 1 - 1/2 \chi _ { \{ y > 0 \} } ) y $
        are not differentiable at $x_0$ and $0$, respectively.
        Hence $U h(x_0) = \{ 0 \}$ and $U f(0) = \{ 0 \}$.
        The two functions
        compose to $f(h(x))=x$
        witnessing that $U ( \identity_\R )  (x_0) = \{ 0 \}$
        as claimed.
\end{remark}

\section{Differentiation and derived sets: composition with a test mapping}\label{s:P52}
We aim to construct a derivative assignment and prove that it is complete, 
see Definitions~\ref{def:assignment} and~\ref{def:completeAssignment}.
The latter will be done in Theorem~\ref{t:MP52}, which therefore
extends~\cite[Proposition~5.2]{MP2016}.
Since the completeness 
will be tested by a composition with a function $g \fcolon G\subset X \to Y$,
the first step is to relate the derivatives of $f \fcolon \Edomf \subset Y \to Z$
to the derivatives of the composition $f\circ g$.
This is straightforward if $\Edomf=Y$ but 
not in the general case, so the present section is devoted to this topic.
The use of Hadamard derivatives and 
directional (non-)porosity will be vital.

Recall $\mathcal L_1$ is the collection of exceptional sets defined in Definition~\ref{def:L1null}.
After a simple Lemma~\ref{l:cruleHdd}
we show
in Lemma~\ref{l:step12} below
that the derivatives $(f\circ g)'(x;\cdot)$ and $f'(g(x); g'(x;\cdot))$ coincide away from an exceptional set.

\begin{lemma}
\label{l:cruleHdd}
        Let $X$, $Y$, $Z$ be Banach spaces,
        $G\subset X$,
        $\Edomf\subset Y$,
        $\fngA\fcolon G\to Y$ an arbitrary function and
        $f\fcolon \Edomf\to Z$ a Lipschitz function.
        Let $x\in \gpreimage\coloneq\fngA^{-1}(\Edomf)\subset G$ and $e\in X$.
        Assume that Hadamard derivatives $L\coloneq (f\circ g)'(x;e)$ and
        $g'(x;e)$ exist, and $\gpreimage$ is not porous at $x$ in the
        directions $\pm e$.
        Then
        $L$ is a Hadamard derivative of $f$ at $g(x)$ in the direction of $g'(x;e)$.
\end{lemma}
\begin{remark}\label{r:three-der-uniq}
Under the assumptions of Lemma~\ref{l:cruleHdd},
Hadamard derivatives $(f\circ g)'(x;e)$, $g'(x;e)$ and $f'(g(x);g'(x;e))$ are unique if they exist.

Indeed, by
Remark~\ref{r:thickandporous}, $\gpreimage$ is thick at $x$ in the direction of $e$.
By Remark~\ref{r:thickandmap}, $g(\gpreimage) \subset \Edomf$ is thick at $g(x)$ in the direction of $g'(x;e)$.
Remark~\ref{r:thickandunique} then brings the conclusion.
\end{remark}
\begin{proof}[Proof of Lemma~\ref{l:cruleHdd}]
Assume that
$L$ is
not
a
Hadamard
derivative of $f$ at $\fngA(x)\in \Edomf$ in the direction of $\fngA'(x;e)\in Y$.
Then there is $\varepsilon >0$ and sequences
$u_k \to \fngA'(x;e)$
and $t_k \to 0$, $t_k \neq 0$,
such that
$y_k\coloneq\fngA(x)+t_k u_k$ is in~$\Edomf$ and
\begin{equation}\label{eq:eqa}
        \norm{\frac{f(y_k)-f(\fngA(x))}{t_k} - L} > \varepsilon
\end{equation}
for every $k\in \N$.
If $e=0$, let $x_n=x$ and $k_n=n$ for all $n\ge1$. Otherwise,
let $k_0=0$ and 
for each $n\ge1$,
since $\gpreimage$ is neither $1/n$-porous at $x$ in the direction $e$,
nor in the direction $-e$,
there is 
$k_n>k_{n-1}$ 
such that we can find a point
$x_n \in \gpreimage\cap B(x+t_{k_n} e, t_{k_n}\norm{e}/n)$.
Let $v_n = (x_n - x)/t_{k_n}$. Note that $\norm{v_n -e} \le \norm e /n$.
Since $L$ is the Hadamard derivative of $f\circ \fngA$ at $x$ in the direction $e$,
we have
\begin{equation}\label{eq:fgder}
        \lim_{n\to\infty} \norm{\frac{f(\fngA(x_n))-f(\fngA(x))}{t_{k_n}} - L} = 0
\end{equation}
and since $\fngA'(x;e)$ is the Hadamard derivative of $\fngA$ at $x$ in the direction $e$,
we have
\[
        \lim_{n\to\infty} \norm{\frac{\fngA(x_n)-\fngA(x)}{t_{k_n}} - \fngA'(x;e)} = 0
        .
\]
Since 
$u_{k_n} \to \fngA'(x;e)$,
the last equation is equivalent to
\[
        \lim_{n\to\infty} \norm{ \frac{\fngA(x_n)-y_{k_n} }{t_{k_n}} } = 0
        .
\]
Since $f$ is Lipschitz on $\Edomf$, this implies
\begin{equation}\label{eq:gder3}
        \lim_{n\to\infty} \norm{ \frac{f(\fngA(x_n))-f(y_{k_n}) }{t_{k_n}} } = 0
        .
\end{equation}
Combining \eqref{eq:fgder} and \eqref{eq:gder3} we obtain
\begin{equation}\label{eq:fgder2}
        \lim_{n\to\infty} \norm{\frac{f(y_{k_n})-f(\fngA(x))}{t_{k_n}} - L} = 0
\end{equation}
which contradicts \eqref{eq:eqa}.
\end{proof}
\begin{lemma}
\label{l:step12}
       Let $X$, $Y$ and $Z$
       be Banach spaces,
such that $X$ is separable,
and assume
$G\subset X$,
$
\Edomf\subset Y$.
Let $\fngA\fcolon G\to Y$ and $f\fcolon \Edomf\to Z$ be Lipschitz functions.
Assume that $F\withrnpindex\subset F$ is such that $f(\Edomf\withrnpindex) \subset Z\withrnpindex\subset Z$, where $Z\withrnpindex$ is a Banach subspace of $Z$ with the Radon-Nikod\'ym property. 

Then
there is an $\mathcal L_1$ null set $Q\subset X$ such that,
for every $x\in \fngA^{-1}(\Edomf\withrnpindex)\setminus Q$ and $e\in X$,
if 
Hadamard derivative
$\fngA'(x;e)$ exists
then
Hadamard derivatives $(f\circ \fngA)'(x;e)$ and $f'(\fngA(x); \fngA'(x;e))$
exist and are equal to the same vector.
\end{lemma}
\begin{proof}
Let $f\withrnpindex=f|_{\Edomf\withrnpindex}$.
Let $\gpreimage\coloneq\fngA^{-1}(\Edomf\withrnpindex)\subset G$ be the domain of $f\withrnpindex\circ \fngA$.
Let $Q_1$ be the set of points $x\in\gpreimage$ such that $f\withrnpindex\circ
\fngA$ is not Hadamard differentiable at $x$. As $Z\withrnpindex$ has the
Radon-Nikod\'ym property, by Definition~\ref{def:L1null}
we conclude that $Q_1$ is $\mathcal L_1$ null.
Let $Q_2$ be the set of points $x\in \gpreimage$ at which $\gpreimage$ is directionally porous.
As $Q_2$ is a directionally porous set,
we also have that $Q_2$ is $\mathcal L_1$ null, see Remark~\ref{r:porousAreNull}.

Let $Q=Q_1\cup Q_2$, an $\mathcal L_1$ null set; assume
$x\in \gpreimage\setminus Q=\gpreimage\setminus (Q_1\cup Q_2)$, $e\in X$.
Since $x\in \gpreimage\setminus Q_1$, we conclude that
Hadamard derivative $L\coloneq(f\withrnpindex \circ \fngA)'(x;e)$ exists.

        Recall that $\gpreimage\subset \fngA^{-1}(\Edomf)$ is not directionally porous at~$x$.
        From this, 
        by Lemma~\ref{lem:MZ:derext},
         we first deduce that also $(f  \circ \fngA)'(x;e) = L$ exists.
        If, in addition, Hadamard derivative $\fngA'(x;e)$ exists then, by Lemma~\ref{l:cruleHdd},
        $L$ is a Hadamard derivative of $f$ at $g(x)$ in the direction of $g'(x;e)$.
Note that
the Hadamard derivatives are unique by Remark~\ref{r:three-der-uniq}.
\end{proof}

The following lemma may be viewed as an approximation analogue of Lemma~\ref{l:step12},
with Hadamard derivatives replaced by the corresponding Hadamard derived sets.

\begin{lemma}\label{l:upanddown}
   Suppose that $X$, $Y$ and $Z$ are Banach spaces,
      $f\fcolon \Edomf\to Z$ 
   and
   $\fngB\fcolon G\to Y$ are Lipschitz, where
   $G\subset X$ and
   $\Edomf\subset Y$.

   Then,
   for every
   $x\in \gpreimage=\fngB^{-1} (\Edomf)\subset G$
   and $e\in X$,
   if
   the
   Hadamard derivative
   $\fngB'(x;e)$
   exists,
   then
    the Hadamard derived sets
    \[
    \hadamderset[\delta,\omega_1] f(\fngB(x), \fngB'(x;e))
    ,
    \qquad
    \hadamderset[\delta,\omega_2] (f\circ \fngB) (x,e)
    \qquad
    \qquad
    (\delta,\omega_1,\omega_2 >0)
    \]
    are close to each other in the following sense:

\begin{inparaenum}[\textup\bgroup(1)\egroup]
\item \label{item:upanddown:first}
 For every $\omega_1>0$ there are
                $\delta_0>0$, $\omega_0>0$ such that
                \begin{equation}\label{eq:inclusion1}
                    \hadamderset[\delta,\omega_0] (f\circ \fngB) (x,e)
                    \subset
                    \hadamderset[\delta,\omega_1] f(\fngB(x), \fngB'(x;e))
                \end{equation}
                for every $\delta \in (0, \delta_0)$.

\item  \label{item:upanddown:second}             
If $\gpreimage$, additionally, is not porous at $x\in\gpreimage$ in the direction of $e$, 
                then
        for all
        $\varepsilon>0$,
        $\varepsilon_0>0$
        and $\omega_1 >0$,
        there is $\delta_0 > 0$
        such that
        \begin{equation}\label{eq:inclusion2lemma}
                    \hadamderset[\delta,\omega_1] f(\fngB(x), \fngB'(x;e))
                    \subset
                    \hadamderset[\delta,\varepsilon_0] (f\circ \fngB) (x,e)
                    +
                    B(0_Z, \omega_1 \Lip f + \varepsilon)                   
        \end{equation}
for every
        $\delta \in (0,\delta_0)$.
\end{inparaenum}
        \end{lemma}
\begin{proof}
                \itemref{item:upanddown:first}
                Let $x\in \gpreimage$ and $e\in X$.
                Assume that the Hadamard derivative $\fngB'(x;e)$ exists.

                Fix $\omega_1> 0$ arbitrarily.
                We find
                $\delta_0>0$, $\omega_0>0$ such that
                \begin{equation}\label{eq:fromDeriv}
                        \norm
                        {
                                       \frac{  \fngB(x+t u) - \fngB(x)  }{ t }
                            -
                                   \fngB'(x;e)
                        }
                        < \omega_1
                \end{equation}
                for every $t\in (-\delta_0,0)\cup(0, \delta_0)$ and $u\in B(e,\omega_0)$
                such that $x+tu \in G$.

                If
                $\delta \in (0, \delta_0)$
                and
                $w\in \hadamderset[\delta,\omega_0] (f\circ \fngB)(x,e)$
                then there are 
                $t \in 
                (0,\delta_0)$ and $u\in B(e,\omega_0)$
                such that $x + t u \in \gpreimage$ and
                \[
                     w
                  =
                      \frac{
                             f(\fngB(x+t u))-f(\fngB(x))
                           }{t}
                   .
                \]
                Denoting $v=\frac{  \fngB(x+t u) - \fngB(x)  }{ t }$ we conclude from \eqref{eq:fromDeriv}
                \[
                    w
                  =
                      \frac{
                             f(\fngB(x) +t v))-f(\fngB(x))
                           }{t}
                  \in
                    \hadamderset[\delta,\omega_1] f(\fngB(x), \fngB'(x;e))
                  .
                \]

                \itemref{item:upanddown:second}
                Denote
                $K\coloneq\Lip f$.
                Let
                $x\in \gpreimage$
                and $e\in X\setminus \{0\}$.
                Assume that
                Hadamard derivative
                $\fngB'(x;e)$ exists.
                Assume also that
                $\gpreimage$ is not porous at $x$ in the direction of~$e$.

                Assume there are
                $
                 \varepsilon,
                 \varepsilon_0,
                 \omega_1
                 >0$
                and a sequence $\delta_k \downto 0$ such that
        \begin{equation*}
                    \hadamderset[\delta_k,\omega_1] f(\fngB(x), \fngB'(x;e))
                    \not\subset
                    \hadamderset[\delta_k,\varepsilon_0] (f\circ \fngB) (x,e)
                    +
                    B(0_Z, K \omega_1  + \varepsilon)
                    .
		\end{equation*}
                Then there are
                $w_k\in
                    \hadamderset[\delta_k,\omega_1] f(\fngB(x), \fngB'(x;e))
                $
                such that
        \begin{equation}\label{eq:eqA}
                \dist(w_k,
                    \hadamderset[\delta_k,\varepsilon_0] (f\circ \fngB) (x,e)
                )\ge K \omega_1 + \varepsilon
                .
        \end{equation}
                Find
                $t_k \in 
                (0,\delta_k)$ and $u_k\in B(\fngB'(x;e), \omega_1)$
                such that
                $\fngB(x)+t_k u_k \in \Edomf$
                and
                \[
                  w_k
                  =
                      \frac{
                             f(\fngB(x)+t_k u_k)-f(\fngB(x))
                           }{t_k}
                  .
                \]

Let $p_k=x+t_k e$.
Fix $n \in \N$ for a short while.
Since $\gpreimage$ is not $1/n$-porous at $x$ in the direction of $e$,
there are $k_n\ge n$ and
$x_n \in \gpreimage\cap B(p_{k_n}, t_{k_n}{\norm{e}}/n)$.
Let $v_n = (x_n - x)/t_{k_n}$. Note that $\norm{v_n -e} <  {\norm{ e }}/ n $.
Thus
                \[
                     \tilde w_{k_n}
                   \coloneq
                      \frac{
                             f(\fngB(x +t_{k_n} v_n))-f(\fngB(x))
                           }{t_{k_n}}
                  \in
                    \hadamderset[\delta_{k_n}, \norm{ e }/n] (f\circ \fngB) (x,e)
                    .
                \]
                Notice that
\begin{align}
                \norm{
                      \tilde w_{k_n}
                      -
                      w_{k_n}
                     }
          &=
                \norm{
                      \frac{
                             f(\fngB(x +t_{k_n} v_n))
                             -
                             f(\fngB(x)+t_{k_n} u_{k_n})
                           }{t_{k_n}}
                     }
          \le
                K
                \norm{
                      \frac{
                             \fngB(x +t_{k_n} v_n)
                             -
                             \fngB(x)
                                 - t_{k_n} u_{k_n}
                           }{t_{k_n}}
                     }
          \notag
          \\
          &\le
                K
                \norm{
                      \frac{
                             \fngB(x +t_{k_n} v_n)
                             -
                             \fngB(x)
                           }{t_{k_n}}
                      -
                      \fngB'(x;e)
                     }
              +
                K
                \norm{
                      u_{k_n}
                      -
                      \fngB'(x;e)
                     },
          \notag
\end{align}
        where the first term goes to zero
        as $n\to \infty$
        and the second is bounded from above by $K \omega_1$.
	This contradicts \eqref{eq:eqA}, hence \eqref{eq:inclusion2lemma} is proved.
\end{proof}

\section{Assignments $\UPhi f$ and $\EAPhi f$}\label{sec:Uphi}
In this section we will always assume $Y,Z$ to be Banach spaces, $\Edomf\subset Y$ and 
$f\fcolon \Edomf\to Z$ to be a Lipschitz function.
We also assume $\Phi\subset\Lip_L(Z)$,
the collection of real-valued $L$-Lipschitz functions on $Z$, where $L>0$. The most important case for us will be $L=1$.
Many results proved in this case can then be generalised using Remark~\ref{r:change_Phi}\itemref{item:KPhi} below.

For each $\varphi\in\Lip_L(Z)$, $T\subset Z$, $z\in Z$ and $\varepsilon>0$ we follow \cite{MP2016} and denote
\begin{align}
\label{eq:def:Bphi}
B_\varphi(T,\varepsilon)&=
\{z\in Z\setcolon \varphi(z)<\sup\varphi|_{_T}+\varepsilon\},\\
B_\varphi(z,\varepsilon)&=
B_\varphi(\{z\},\varepsilon).
\notag
\end{align}

\begin{lemma}\label{l:Bphi}
For $\varphi,\tilde\varphi\in\Lip_1(Z)$, $S\subset T\subset Z$, $z\in Z$ and $\varepsilon,\eta>0$, we have
      \begin{align}
				\label{eq:phi:ball}
                B(z,\varepsilon)
                &\subset
				B_\varphi(z, \varepsilon );
				\\
                \label{eq:phi:monot}
				S\subset
                B_\varphi( S , \varepsilon )
                &\subset
                B_\varphi( T , \varepsilon );
            \\
				\label{eq:phi:repeat}
				B_\varphi(B_\varphi(S,\varepsilon),\eta)
				&\subset
				B_\varphi(S,\varepsilon+\eta);
            \\
                \label{eq:phi:addin}
                B_\varphi( S + B(0_Z, \eta), \varepsilon )
                &\subset
                B_\varphi( S , \varepsilon + \eta);
            \\
                \label{eq:phi:addout}
                B_\varphi( S , \varepsilon ) + B(0_Z, \eta)
                &\subset
                B_\varphi( S , \varepsilon + \eta);
            \\
                \label{eq:phi:phi}
                B_{\tilde \varphi}( S , \varepsilon ) \cap T
                &\subset
                B_{\varphi}( S , \varepsilon+2\sup_{T}\abs{\varphi-\tilde\varphi})\cap T.
\end{align}
\end{lemma}
\begin{proof}
We verify the least trivial assertion~\eqref{eq:phi:phi}.
Assume $z\in B_{\tilde \varphi}( S , \varepsilon ) \cap T$;
then $z\in T$ implies $\varphi(z)\le\tilde\varphi(z)+\sup_{T}\abs{\varphi-\tilde\varphi}$,
and $z\in B_{\tilde \varphi}( S , \varepsilon ) $ implies $\tilde\varphi(z)<\sup_{S}\tilde\varphi+\varepsilon\le\sup_S \varphi+\sup_{T}\abs{\varphi-\tilde\varphi}+\varepsilon$. It follows that
$\varphi(z)<\sup_S\varphi+2\sup_{T}\abs{\varphi-\tilde\varphi}+\varepsilon$.
\end{proof}

\begin{definition}\label{def:phi}
        We say that $\Phi\subset \Lip_L(Z)$ is \emph{separable} if it is
        separable in the topology of the uniform convergence on bounded sets
        in $Z$.
        We call $\Phi\subset \Lip_L(Z)$ \emph{admissible} if it is separable
        and satisfies the following conditions:
\begin{equation}\label{eq:cond1}
   \text{\relax
        for every $z\in Z$ and $\tau > 0$ there are $\varphi\in \Phi$
and $\varepsilon > 0$ such that 
$B_\varphi(z,\varepsilon) \subset B(z,\tau)$}
\end{equation}
and
\begin{equation}\label{eq:cond2}
   \text{\relax
     if $\varphi\in\Phi$ and $z\in Z$ then $\varphi(\cdot - z) \in \Phi$.}
  \eodhere
\end{equation}
\end{definition}

\begin{remark}\label{r:ZsepPHIsep}
\begin{inparaenum}[(i)]
\item\label{r:ZsepPHIsep:item:Ur}
        If $Z$ is separable, then $\Phir=\Phir(Z)\coloneq \{ \norm{\cdot-z} \setcolon z\in Z\}\subset\Lip_1(Z)$
        is an example of
        a separable and
        admissible $\Phi\subset \Lip_1(Z)$.
\par
\item
        On the other hand,
        if $\Phi\subset\Lip_1(Z)$ is separable
        and satisfies~\eqref{eq:cond1}
        then $Z$ is necessarily separable.
Indeed, fix $\{\varphi_n \setcolon n\in\N \}$ dense in $\Phi$.
It is enough to find a countable dense subset of 
$
\BBB\coloneq B_Z(0,K),
$
where $K>0$ is fixed.
For every $n\in \N$ and $q\in \Q$
such that $q > \inf \varphi_n(\BBB)  \in \R $
find $z_{n,q}\in \BBB$ such that $\varphi_n(z_{n,q})<q$.
Then $\{z_{n,q}\}$ is a countable dense subset of $\BBB$.
To show this, fix $z\in \BBB$ and $\tau \in (0,K - \norm {z})$.
Using \eqref{eq:cond1},
find $\varphi\in \Phi$ and $\varepsilon > 0$ such that
$B_\varphi(z,\varepsilon) \subset B(z,\tau)$.
Then find $n$ such that $\norm{\varphi-\varphi_n}_{\sup, \BBB} < \varepsilon / 4$.
Then $ \abs{\varphi(z)-\varphi_n(z)} < \varepsilon/4$
and
$B_{\varphi_n}(z,\varepsilon/4) \cap \BBB \subset B_{\varphi}(z,3\varepsilon/4) \subset B(z,\tau)$, cf.\ also~\eqref{eq:phi:phi}.
By choosing any $q\in \Q \cap (\varphi_n(z), \varphi_n(z) + \varepsilon/4)$
we have $z_{n,q} \in B_{\varphi_n}(z, \varepsilon/4 ) \cap \BBB \subset B(z,\tau)$.
This shows that $\{z_{n,q}\}$ is a countable dense subset of $\BBB$.
\par\item
If $Y$ is assumed to be separable,
$\Edomf\subset Y$
and $Z$ is not separable, then $Z$ can be replaced by the closed linear span of $f(\Edomf)$, which is a separable
subspace of $Z$. The results of the present paper then can be used in this case.
\eodhere
\end{inparaenum}
\end{remark}

\begin{definition}\label{def:Y}
Let $Y$ be a Banach space. 
Define $[Y]$ to be the map that assigns to each $y\in Y$ 
the full space $Y\in2^Y$ as its value.
\end{definition}
\begin{remark}\label{rem:Ycomplete}
        It is straightforward to verify that the assignment $[Y]$ is a complete Borel measurable 
        assignment, see Definitions~\ref{def:completeAssignment}
        and~\ref{def:assignment-meas-nof}.
\end{remark}

    The following definition is inspired
    by \cite[Proposition 5.1(vi)]{MP2016}
    and
    by the notion of a regular point,
    see \cite[Section 6.2]{LPT}.

\begin{definition}[Assignment $\UPhiplus f(y)$]\label{def:assignmentPhi}
Let $Y,Z$ be Banach spaces, $f\fcolon\Edomf\subset Y\to Z$ be a function and $\Phi\subset\Lip_L(Z)$, where $L>0$.
Let $R\fcolon\Edomf\to 2^Y$ be an assignment. 
Assume that $\Edomf$ is thick at $y$ in the direction $u$ for every $y\in\Edomf$ and $u\in R(y)$.

For $y\in\Edomf$, we define $\UPhiplusR f(y)$ as the set of those directions
$u\in R(y)$
for which
Hadamard derivative $f'(y; u)$ exists
and which have the property
that for every
$\varepsilon >0$,
$\varphi\in\Phi$
and
$v\in
R(y)
\setminus\{0\}
$,
there exists $\delta_0>0$ such that,
for every
$\delta, \omega\in (0,\delta_0)$
and $\sigma\in \{\pm 1\}$,
\begin{align}\label{eq:UPhiRegularity}
            \hadamdersetplus[\delta,\omega] f(y, \sigma u + v)
        \subset 
             f'(y; \sigma u)
            +
             B_\varphi( \hadamdersetplus[
             \delta, \varepsilon
             ] f(y,v), \varepsilon)
        .
\end{align}

Let $M\subset Y$ and $d_M=\dist(\cdot,M)\fcolon Y\to\R$ be the distance function. 
Let        $\Phir=\Phir(\R)$ be as in Remark~\ref{r:ZsepPHIsep}\itemref{r:ZsepPHIsep:item:Ur}.
Then we define
\begin{equation}\label{eq:regtanex:def}
        \regtanex_y M=\regtanex M(y)\coloneq\UPhirplusY d_M(y)
        \qquad\text{ for all } y\in M
\end{equation}
and
\begin{equation}\label{eq:UPhidef}
	\UPhiplus f=
		\UPhiplusArg{\regtanex\Edomf}f.
	\eodhere
\end{equation}
\end{definition}

The definition of $\UPhiplus f$ in~\eqref{eq:UPhidef} requires that
$\Edomf$  is thick at $y$ for all $y\in \Edomf$ 
in
directions $u \in \regtanex \Edomf(y)$. We check it in Lemma~\ref{l:nporous}. 

\begin{remark}\label{r:disttan:eq}
	For $y\in M\subset Y$, we have
	$\regtanex M(y)=\UPhirplus d_M(y)$. Indeed, by Definition~\ref{def:assignmentPhi}, 
	$\regtanex M(y)=\UPhirplusY d_M(y)=\UPhirplus d_M(y)$.
\end{remark}
\begin{remark}\label{r:regtestdir0}
	Note that in Definition~\ref{def:assignmentPhi},
	``$v\in R(y)
	\setminus\{0\}
	$''
	can be replaced by
	``$v\in R(y)
	$''.
	Indeed, if $v=0$
	then
	$
	0 \in
	\hadamderset[
	\delta, \varepsilon
	] f(y,v)
	$.
	Hence, if Hadamard derivative $f'(y;u)$ exists, then 
	\eqref{eq:UPhiRegularity} for $v=0$ can be deduced from the definition of Hadamard derivative,
	cf.\ also Remark~\ref{r:HDsetandder}
	and~\eqref{eq:phi:ball}.
\end{remark}

\begin{lemma}\label{l:nporous}
	Let $Y$ be a Banach space and $M\subset Y$.
	Then, for every $y\in M$ and $u\in\regtanex M(y)$,
	\begin{enumerate}[\textup\bgroup(1)\egroup]
		\item\label{l:nporous:der}\label{l:nporous:__B}
		Hadamard derivative $d_M'(y;u)$ exists,
		\item\label{l:nporous:nporous}
		$M$ is not porous at $y$ in the direction of $\pm u$,
		\item\label{l:nporous:thick}\label{l:nporous:____E}
		$M$ is thick at $y$ in the direction of $\pm u$.  
	\end{enumerate}
\end{lemma}
\begin{proof}
	The existence of the Hadamard directional derivative
	of the distance function is a requirement of Definition~\ref{def:assignmentPhi}.
	By Remark~\ref{r:derandporous}, condition
	\itemref{l:nporous:der} easily leads to \itemref{l:nporous:nporous}.
	Then, \itemref{l:nporous:thick} 
	follows
	by  Remark~\ref{r:thickandporous}.
\end{proof}

\begin{example}\label{ex:disttan:full}
	If $M=Y$
	then
	corresponding distance function $d_M$ is the zero function, hence, by Definition~\ref{def:assignmentPhi},
$\UPhirplusY d_M$
is	identically equal to $Y$.
	Thus, $\regtanex_y Y=Y$ for every $y\in Y$.
\end{example}

        Note that 
        \begin{align}
               \label{eq:HARsubDisttan}
               \UPhiplusR f(y) &\subset R(y) 
                ,
               \\
               \noalign{in particular,
        using also Example~\ref{ex:disttan:full},
               }
               \label{eq:HAsubDisttan}
               \UPhiplus f(y) &\subset \regtanex \Edomf(y) 
                .
        \end{align}

\begin{remark}\label{r:change_Phi}
	\begin{inparaenum}[(a)]
		\item\label{item:KPhi}        In Definition~\ref{def:assignmentPhi}, $\UPhi f$ (and $\UPhiplusR f$)
		does not change when $\Phi$ is replaced by 
		$c\Phi=\{ c \varphi \setcolon \varphi \in \Phi\}$,
		where $c>0$ is arbitrary.
		Indeed,
		if $\varphi$ is replaced by $c \varphi$ and
		$\varepsilon$ by $c \varepsilon$
		then
		the set $B_\varphi(S,\varepsilon)$
		in \eqref{eq:UPhiRegularity}
		gets replaced by
		\( B_{c \varphi}(S,c \varepsilon) =
		B_\varphi(S,\varepsilon) \).
		
		\item\label{item:ZvPhi}
		Also, if $\tilde Z$ is a Banach space,
		$Z\subset\tilde Z$ and $\tildePhi\subset\Lip_L(\tilde Z)$ is such that 
		$\Phi=\{\tilde\varphi|_Z\setcolon\tilde\varphi\in\tildePhi\}$,
		then for any Lipschitz
		$f\fcolon\Edomf\to Z$ we have $\UPhi f=\Hindexedbyfnspace{\tildePhi} f$.
		Indeed, in Definition~\ref{def:assignmentPhi} both the left-hand side
		$\hadamdersetsym[\delta,\omega] f(y, \sigma u + v)$ and
		$f'(y;\sigma u)$ lie in $Z$, hence only values of~
		$\tilde\varphi$
		on~$Z$ matter.
	\end{inparaenum}
\end{remark}

        In the above definition, the most important case is the one when
        $R=[Y]$, $\Edomf=Y$ and $Z$ are separable, and $f\colon \Edomf\to Z$ is Lipschitz. 
   If this is the case
   and~$\Phi\subset\Lip_1(Z)$ is admissible,
   we show in Lemma~\ref{l:UMPequalsUH}    that $\UPhiplus f$ equals $\UPhiMP f$ defined in
   \cite[Proposition~5.1(vi)]{MP2016}.
   However, we may not always guarantee that the collection of Lipschitz functions is admissible and that the codomain space has
   the
   Radon-Nikod\'ym property (which is used in~\cite{MP2016} to prove completeness of $\UPhiMP$); we
   will consider Hadamard derivative assignments with $\tildePhi$ and $\tilde Z$ without these restrictions.

\let\SpellOn(
\begin{lemma}\label{l:extphi}
	Let $Z_0\subset \tilde Z$ be Banach spaces,
	$\PhiOnSmallest\subset\Lip_{L}(Z_0)$, where $L>0$.
	Then there exists $\tildePhi \subset\Lip_{2L}(\tilde Z)$,
        separable if $\PhiOnSmallest$ is separable, 
	such that
	
\begin{inparaenum}[\textup\bgroup(1)\egroup]
\item\label{extphi:one}
	$\{ \tilde\varphi | _ {Z_0} \setcolon \tilde\varphi
	\in \tildePhi \} = \PhiOnSmallest$.
	
\item\label{extphi:two}
	if $z_0 \in Z_0$ and $\varphi_1, \varphi_2 \in \PhiOnSmallest$
	satisfy
	$\varphi_1(z) = \varphi_2(z - z_0)$ for $z\in Z_0$,
	then there exist $\tilde\varphi_1, \tilde\varphi_2 \in \tildePhi$ such that
	$\tilde\varphi_1 | _ {Z_0} = \varphi_1$ and
	$\tilde\varphi_1(z) = \tilde\varphi_2(z - z_0)$ for $z\in \tilde Z$.
	
\item\label{extphi:three}
	if $\PhiOnSmallest$ satisfies~\eqref{eq:cond2} for all $z\in Z_0$,
	then
	$\tildePhi$ satisfies~\eqref{eq:cond2} for all $z\in Z_0$ too, meaning that
	\begin{equation}\label{eq:cond2Z0}
		\text{\relax
			if $\tilde\varphi\in\tildePhi$ and $z_0\in Z_0$ then $\tilde\varphi(\cdot - z_0) \in \tildePhi$.}
	\end{equation}
\end{inparaenum}
\end{lemma}
\begin{proof}
	For every $\varphi \in \PhiOnSmallest$, consider its
	$2L$-Lipschitz extension $e(\varphi)\fcolon
	\tilde Z
	\to \R$
	of $\varphi$
	defined in the spirit of the well-known formula
	by
	$e(\varphi)(z) =
	\inf
	\{ g_{q,\varphi}(z) \setcolon q\in 
	Z_0
	\}$
	where
	$g_{q,\varphi}(z)=\varphi(q)
	+
	2L\norm{q-z}$
	for 
	$z\in \tilde Z$,
	see also the proof of \cite[Lemma~1.1]{BL}.
	Define $\tildePhi = e(\PhiOnSmallest)\subset\Lip_{2L}(\tilde Z)$.
	
	Fix
	$z\in \tilde Z$.
	Let
	$d=\dist_{\tilde Z}(z,Z_0)$
	and find $z_0\in Z_0$ with $\norm{z-z_0} \le 2d \le 2\norm z$.
	Then for every $q\in Z_0$ with $\norm{q}\ge11\norm z$ we have  
	$\norm{q-z_0}\ge
	11\norm z-\norm{z-z_0}-\norm z\ge
	10d-2d=8d$, 
	so that
\[	\norm{q-z_0} \ge 8d\ge 4\norm{z_0-z}.
\]	
This implies,
	using $\Lip(\varphi)\le L$,
	\begin{equation}\label{eq:Borel:ball}        
		g_{q,\varphi}(z)
		\ge
		(\varphi(z_0)
		-
		L \norm{q-z_0})
		+
		2L
		(\norm{q-z_0} -  \norm{z_0 - z})
		\ge
		\varphi(z_0)
		+
		2L\norm{z_0 - z}
		= g_{z_0,\varphi}(z),
	\end{equation}         
	so that
	$e( \varphi)(z) =
	\inf
	\{ g_{q,\varphi}(z) \setcolon q\in B_{Z_0}(0, 11\norm z) \}$.

	From the definition of $g_{q,\varphi}$ we can
	also
	deduce
	that for every $\varphi_1, \varphi_2 \in \PhiOnSmallest$, $R>0$ and $q\in B_{Z_0}(0,11R)$,
	\[
	\abs
	{g_{q,\varphi_1}(z_1)- g_{q,\varphi_2}(z_1)}
	=
	\abs{  \varphi_1(q) - \varphi_2(q) }
	\le \norm{\varphi_1-\varphi_2}_{B_{Z_0}(0,11R)}
	\qquad
	\text{for any}
	\quad
	z_1\in B_{\tilde Z}(0,R).
	\]
	This implies
	$
	\norm{e(\varphi_1) - e(\varphi_2) }_{B_{\tilde Z}(0,R)}
	\le
	\norm{\varphi_1 - \varphi_2 }_{B_{Z_0}(0,11R)}
	$ whenever $\varphi_1, \varphi_2 \in \PhiOnSmallest$.
	
	Therefore
	\(
	e
	\fcolon\PhiOnSmallest\to\Lip_{2L}(\tilde Z)
	\fcolon
	\varphi \mapsto e(\varphi)
	\)
	is a continuous map and,
		if $\PhiOnSmallest$ is assumed to be separable, then
	its image 
\begin{equation}\label{eq:extensionPhi}
	\tildePhi \coloneq 
	e
	(\PhiOnSmallest)
\end{equation}	
is separable in the topology of the uniform convergence on bounded sets.
	
	Obviously, $e$ commutes with the translation of $\varphi \in \PhiOnSmallest$ by any $z_0 \in Z_0$,
        from which
        \itemref{extphi:two} and~\itemref{extphi:three}
        follow.
\end{proof}

\begin{lemma}\label{l:extf}
	Let $Y,\Zmedium$ be Banach spaces,
	$\Flextf\subset Y$ and
	$f\fcolon \Flextf \to \Zmedium$ Lipschitz.
	There exist Banach space $\tilde Z \supset f(\Flextf)$,
        separable if\/ $Y$ is separable,
	and Lipschitz function $\tilde f\fcolon Y\to \tilde Z$ such that
	$\tilde f$ is an extension of~$f$
	and $\Lip \tilde f = \Lip f$.
\end{lemma}
\begin{proof}
	Let $\Gamma=B_{\Zmedium^*}(0,1)$ 
	and let $j\fcolon \Zmedium\to l_\infty(\Gamma)$ be the linear isometric embedding of $\Zmedium$ into
	$l_\infty(\Gamma)$ defined by $j(z)(e^*)=e^*(z)$ for all $z\in \Zmedium$ and
	$e^*\in\Gamma$, cf.~\cite[Proof of Proposition~1.1]{Zip}.
	From now on, let us identify $\Zmedium$ with $j(\Zmedium) \subset l_\infty(\Gamma)$. Note that
	the Banach space structures of $\Zmedium$ and $j(\Zmedium)$ coincide, so this
	identification does not affect differentiability and the values of
	derivatives of $f$.
	
	Let $K=\Lip f$.
	By \cite[Lemma~1.1(ii)]{BL}
	with $\omega(t) = Kt$ for $t \ge 0$,
	$K$-Lipschitz
	function 
	$f\fcolon \Flextf\to l_\infty(\Gamma)$
	admits a
	$K$-Lipschitz
	extension $\extendedf \fcolon Y\to l_\infty(\Gamma)$ defined by
\begin{equation}\label{eq:extensionBL}
\extendedf (y)(e^*)=\inf\Bigl\{f(y')(e^*)+Kd(y,y')\setcolon y'\in Y\Bigr\},
\qquad 
e^*\in\Gamma.
\end{equation}	
	Let 
\begin{equation}\label{eq:extensionBL_codomain}
	\tilde Z=\closedSpan \extendedf (Y)\subset l_\infty(\Gamma);
\end{equation}        then
        $\tilde Z\supset f(\Flextf)$
	and $\extendedf \fcolon Y\to\tilde Z$ is a
	$K$-Lipschitz extension of~$f$.
Note that
	$\tilde Z$ is a separable Banach space if $Y$ is separable.
\end{proof}

    In the following definition, we need that $\tilde Z$, $\tilde f$ and $\tildePhi$ have the properties from Lemma~\ref{l:extf}
    and Lemma~\ref{l:extphi}, respectively.
    For the formal reasons
    we fix these objects by the specific formalae
    from the proofs of the two results.

\begin{definition}[Hadamard derivative assignment $\EAPhi f(y)$]\label{def:assignmentEA}
Let $Y,Z$ be Banach spaces,
$f\fcolon\Edomf\subset Y\to Z$ be a Lipschitz function and $\Phi\subset\Lip_L(Z)$.
Let $\tilde Z$ and $\tilde f\fcolon Y\to\tilde Z$ be as in~\eqref{eq:extensionBL_codomain} and~\eqref{eq:extensionBL}.
Let $Z_0 = \closedSpan f(\Edomf)\subset\tilde Z$,
        $\Phi_0 = \Phi|_{Z_0}$ be the collection of restrictions
        and let $\tildePhi\subset\Lip_{2L}(\tilde Z)$
        be as in~\eqref{eq:extensionPhi}.
We define
\begin{equation}\label{eq:def:EA}
       \EAPhi f(y) =
           \Hindexedby{\tildePhi, [Y] } \tilde f (y)
           \cap
           \regtanex \Edomf(y),
           \qquad
           y\in \Edomf
           .
           \eodhere
\end{equation}
\end{definition}
\begin{remark}\label{r:onehalf}
It will be convenient, in some places, to use the following line as an equivalent definition of $\EAPhi f$:	
\begin{equation}\label{eq:def:EA:half}
	\EAPhi f(y) =
	\Hindexedby{\frac12\tildePhi, [Y] } \tilde f (y)
	\cap
	\regtanex \Edomf(y),
	\qquad
	y\in \Edomf
	.
	\eodhere
\end{equation}
We can do so by Remark~\ref{r:change_Phi}\itemref{item:KPhi}.
\end{remark}	
\begin{remark}\label{r:HAisH}
        If $f$ is defined on the whole space (that is, $\Edomf=Y$) then $\EAPhi f$ coincides with $\UPhi f$.
	Indeed, in such a case $\tilde f=f$, $\tilde Z = Z_0 = \closedSpan f(Y) \subset Z$,
	$\tildePhi = \Phi_0=\Phi |_{Z_0}=\Phi |_{\tilde Z}$
        so that by Remark~\ref{r:change_Phi}\itemref{item:ZvPhi},  
	\(
	\Hindexedby{ \tildePhi, [Y] } \tilde f (y)
	=
           \UPhi f (y)
	\).
	Finally, $\regtanex \Edomf(y)=Y$ for every $y\in Y$ by Example~\ref{ex:disttan:full}.
\end{remark}

\begin{remark}
In Proposition~\ref{p:UPhiIsAssignment} we show that under suitable additional assumptions, $\EAPhi f$ is indeed a Hadamard derivative assignment, as in Definition~\ref{def:assignment}.	
\end{remark}

\begin{lemma}\label{l:nporousHA}
       Let $Y$ be a Banach space and $M\subset Y$.
        If $Z$ is a Banach space and $f\colon M \subset Y \to Z$
        is a function, then \itemref{l:nporous:__B}--\itemref{l:nporous:____E} of Lemma~\ref{l:nporous} hold true
        for 
        every $y\in M$ and $u\in\UPhi f(y)$.
	The same holds true for
        every $y\in M$ and $u\in\EAPhi f(y)$.
\end{lemma}
\begin{proof}
Follows from Lemma~\ref{l:nporous} by~\eqref{eq:HAsubDisttan} and~\eqref{eq:def:EA}.
\end{proof}

We will need the following easy result in the proof of the forthcoming Lemma~\ref{l:UPhibydense}.
\begin{lemma}\label{l:balls}
If $r_1,r_2>0$, $M_i\subset B(0_Z,r_i)$ for $i=1,2$ and $M\subset Z$ is such that
$M_1\subset M+M_2$ then $M_1\subset (M\cap B(0_Z,r_1+r_2))+M_2$.
\end{lemma}
\begin{proof}
        For any $z_1\in M_1$ we have $z_1=z+z_2$ with $z\in M$ and $z_2\in M_2$.
        Then
        $\norm z \le \norm{z_1}+\norm{z_2} $.
\end{proof}
In the following lemma we show that it is possible to relax the requirements of Definition~\ref{def:assignmentPhi} by replacing respective sets for variables by their dense subsets
without changing 
the values of~$\UPhi f$.
We prove this statement for Lipschitz functions $f$ defined on the whole space $\Edomf=Y$.

\begin{lemma}\label{l:UPhibydense}
        Let $Y,Z$ be Banach spaces, $f\fcolon Y\to Z$ be a
        Lipschitz function,
        $\Phi_0\subset\Phi\subset\Lip_1(Z)$ and $Y_0\subset Y$.
        Assume $\Phi_0$ is dense in $\Phi$ in the topology of the uniform convergence on
        bounded sets 
in $Z$
and 
$Y_0$ is dense in $Y$.
Let further $Q=\{1/n\setcolon n\ge1\}$.

        For $y\in\Edomf$, we define $U(y)$ as the set of those directions
$u\in Y$
        for which
        Hadamard derivative $f'(y; u)$ exists and which have the property
        that for every
        $\varepsilon\in Q$,
\def\jkAnnotinit{\color{green!50!black}[[ BorelComplex Borel via changed def ]]}%
        $\varphi\in\Phi_0$
        and
\def\jkAnnotinit{\color{green!50!black}[[ BorelComplex Borel via changed def ]]}%
        $v\in Y_0 \setminus\{0\} $,
        there exists $\delta_0>0$ such that,
        for every
        $\delta, \omega\in (0,\delta_0)\cap Q$
        and $\sigma\in \{\pm 1\}$,
        \eqref{eq:UPhiRegularity} holds.
        Then $U(y)=\UPhi f(y)$.
\end{lemma}

\begin{proof}
        Assume that $f$ is $K$-Lipschitz, $K>0$.
        We only need to show that
        $U(y)\subset\UPhi f(y)=\UPhiplusY f(y)$ as
        the reverse inclusion is trivial.
        Let $y\in Y$ and $u\in U(y)$;
        then Hadamard derivative $f'(x;u)$ exists. 
        Let $\varepsilon>0$ and  $\varphi\in \Phi$ be arbitrary, and
        let $v\in Y\setminus\{0\}$. 
         We will show
         that there is $\delta _ 0 > 0$ such that for every $\delta, \omega \in (0, \delta_0)$,
         $\sigma\in\{\pm1\}$
         \begin{equation}\label{eq:wanttoprove1}
             \hadamderset[\delta,\omega] f(y, \sigma u + v)
             \subset
             f'(y;\sigma u)
             + B_\varphi(\hadamderset[ \delta, \varepsilon] f(y, v), \varepsilon)
             .
         \end{equation}
         Find $p\in \N$, $\tilde v \in Y_0$ such that
         $p > 2(1 +  4K) / \varepsilon$
         and
         $\norm{ \tilde v - v } < 1/p$,
         and choose $\tilde \varphi \in \densePhi$ such that
         \begin{equation}\label{eq:sup78}
            \sup_{ z\in \closure B  (0_Z,3\varrho)}
                  \abs{ \tilde \varphi (z) - \varphi (z) }
           < 1/(2p),
         \end{equation}
         where $\varrho= (\norm u + \norm v + 5) K$.
        Since
        $u\in U(y)$, there exists $\delta_1\in(0,1)$ such that         
         \begin{equation}\label{eq:dense}
             \hadamderset[\delta,\omega] f(y, \sigma u + \tilde v)
             \subset
             f'(y;\sigma u)
             + B_{\tilde\varphi}(\hadamderset[ \delta, 1/p] f(y, \tilde v),
                               1/p
                               )
         \end{equation}
for every $\delta,\omega\in(0,\delta_1)\cap Q$ and $\sigma\in\{\pm1\}$.
         Now we find a positive integer $m\ge p$ such that 
\begin{equation}\label{eq:cond-m}         
         (2+\norm u + \norm v) / m < 1/p
         \end{equation}
         and let $\delta_0 = \min( 
         1/m, 1/p, \delta_1)$.

         Let $\delta,\omega \in (0, \delta_0)$ be given.
         We fix
         integers
         $q \ge m$
         and
         $s \ge m$
         such that
         $\delta \in [1/(q+1), 1/q)$,
         $\omega \in [1/(s+1), 1/s)$
         and
         let
         $\delta_q \coloneq 1/(q+1)\in Q$,
         $\omega_s \coloneq 1/(s+1)\in Q$.
         Using first Lemma~\ref{l:HDbyDelta} together with $1-\frac{\delta_q}{\delta}<1/q$,
         then
         Lemma~\ref{l:HDbyHD}\itemref{HDbyHD:two} with $\Edomf=Y$, and, finally, $q,s\ge m$ and~\eqref{eq:cond-m},
         we conclude that
 \begin{align}
         \notag
M_1\coloneq
         \hadamderset[\delta,\omega] f(y, \sigma u + v )
         &
         \subset
             \hadamderset[\delta_q,\omega] f(y, \sigma u + v )
          +B(0_Z,2K  ( \norm u + \norm v+\omega) / q
          )
       \\\notag
         &
         \subset
             \hadamderset[\delta_q,\omega_s] f(y, \sigma u + v )
           +B(0_Z,
             2K  ( \norm u + \norm v+\omega) / q
             + 2K /s
          )
       \\\notag
         &
         \subset
             \hadamderset[\delta_q,\omega_s] f(y, \sigma u + v )
           +B(0_Z,
             2K  ( \norm u + \norm v+2) / m
          )
       \\\notag
         &
         \subset
             \hadamderset[\delta_q,\omega_s] f(y, \sigma u + v )
           +B(0_Z,
             2K / p
          )
         .
 \end{align}
         Since 
         $\norm{ \tilde v - v} < 1/p$,
         we apply again Lemma~\ref{l:HDbyHD}\itemref{HDbyHD:two} to get that
         the latter set is contained in
         \(          \hadamderset[\delta_q,\omega_s] f(y, \sigma u + \tilde v )
        +B(0_Z,5K /p)          
         \).
        Since $\delta_q,\omega_s\in Q$,
        we now use~\eqref{eq:dense} to get that, letting
        $M=B_{\tilde\varphi}(\hadamderset[\delta_q, 1/p] f(y, \tilde v),
             1/p
        )$, we have
\begin{align*}
        M_1&\subset
         \hadamderset[\delta_q,\omega_s] f(y, \sigma u + \tilde v )
        +B(0_Z,5K /p)\\
&\subset
             f'(y;\sigma u)
             + B_{\tilde\varphi}(\hadamderset[\delta_q, 1/p] f(y, \tilde v),
             1/p
             )
             +B(0_Z,5K /p)\\
             &=
             f'(y;\sigma u)
             + M
             +B(0_Z,5K /p).
\end{align*}
   By Remark~\ref{r:HDlip},
   \begin{align*}
         &M_1 =
         \hadamderset[\delta,\omega] f(y, \sigma u + v )
         \subset
         B(0_Z, (\norm u + \norm v + 1/s) K)
         \subset
         B(0_Z, \varrho)
         \\
		 &M_2\coloneq f'(y;\sigma u)
		 +B(0_Z,5K/p)
		 \subset
		 B(0_Z,K(\norm u+5/p))
         \subset
         B(0_Z, \varrho)
         ,
   \end{align*}
   so applying Lemma~\ref{l:balls}, we get
\begin{align}\label{eq:ballconclusion1}   
         \hadamderset[\delta,\omega] f(y, \sigma u + v )
         =M_1
         \subset
             f'(y;\sigma u)
             + \bigl(
			M
             \cap B(0_Z,2\varrho)\bigr)
             +B(0_Z,5K /p).
\end{align}
  
We now apply~\eqref{eq:phi:phi} with $S=\hadamderset[\delta_q,1/p] f(y, \tilde v)$ and $T=B(0_Z,2\varrho)$, to get
       \[
		M\cap B(0_Z,2\varrho)\subset    
        B_{\varphi}
           \bigl(
               \hadamderset[\delta_q, 1/p] f(y, \tilde v), 
             2/p
           \bigr)
        .
   \]
We note that~\eqref{eq:phi:phi} is indeed applicable as, by Remark~\ref{r:HDlip},
\[
S\subset B(0_Z,K(\norm{\tilde v}+1/p))\subset B(0_Z,K(\norm{v}+2/p))\subset T.
\]
   Therefore,~\eqref{eq:ballconclusion1}
   gives
\begin{equation}\label{eq:almostfinal}
         \hadamderset[\delta,\omega] f(y, \sigma u + v )
         \subset
              f'(y; \sigma u)
             +
             B_{\varphi}
                \bigl(
                    \hadamderset[\delta_q, 1/p] f(y, \tilde v),
             2/p
                \bigr)
    +B(0_Z, 5 K / p)
          .
          \end{equation}
 By
 $\norm{ \tilde v - v} < 1/p$ and Lemma~\ref{l:HDbyHD}\itemref{HDbyHD:two},
 following by~\eqref{eq:phi:addout} and
 $p>2(1+4K)/\varepsilon$,
 the right-hand side of~\eqref{eq:almostfinal} is contained in
 \begin{align*}
         &
			 f'(y;\sigma u)
             +
             B_{\varphi}
                \bigl(
                    \hadamderset[\delta_q, 1/p ] f(y,v),
             2/p +
                    3 K/p
                \bigr)
             + 
             B(0_Z,
             5 K  / p
             )
          \\
          &\qquad
          \subset
			 f'(y;\sigma u)
             +
             B_{\varphi}
                \Bigl(
                      \hadamderset[\delta_q, 1/p] f(y, v),
                           2/p  +
                            8 K  / p
                \Bigr)
          \\
          &\qquad
          \subset
			 f'(y;\sigma u)
             +
             B_{\varphi} ( \hadamderset[\delta, \varepsilon] f(y, v),
                           \varepsilon
                         )
.
 \end{align*}
        Combining the above and~\eqref{eq:almostfinal}, we get~\eqref{eq:wanttoprove1}.
        Hence, we proved that~\eqref{eq:wanttoprove1}
        is true for every $\delta,\omega \in (0,\delta_0)$.
        By Definition~\ref{def:assignmentPhi}
        this implies
        $u\in \UPhi f(y)$,
          which finishes the proof of
          $U(y)=\UPhi f(y)$.
\end{proof}

\bigbreak
\bigbreak

Our goal now would be to         
         show, in Proposition~\ref{p:UPhiIsAssignment},
        that $\EAPhi f(y)$ is a Hadamard derivative assignment. Then in
        Theorem~\ref{t:MP52} we show that
        it is
        complete, and in Proposition~\ref{p:Borel1}
        that it is Borel measurable. 
Our first aim is to show that $\EAPhi f(y)$ is a closed linear subspace of $Y$, see the forthcoming 
  Lemma~\ref{l:linearity} and Lemma~\ref{l:UPhi-closed-cnts}.

\begin{lemma}\label{l:linearity}
        Let $Y,
        Z\withrnpindex\subset
        Z$ be Banach spaces, $
        \Edomf\withrnpindex\subset
        \Edomf\subset Y$ and 
        $f\fcolon \Edomf\to Z$ a Lipschitz function
        such that $f(\Edomf\withrnpindex) \subset Z\withrnpindex$.
        Suppose $\Phi \subset \Lip_1(Z)$
        obeys condition~\eqref{eq:cond2} for all $z\in Z\withrnpindex$
        and
        assume that $\{ \varphi|_{Z\withrnpindex} \setcolon \varphi \in \Phi \} \subset \Lip_1(Z\withrnpindex)$
        is admissible.

\begin{inparaenum}[\textup\bgroup(1)\egroup]
\item \label{one:lin}
        Let $R\fcolon \Edomf\to 2^Y$ be an assignment 
        such 
        that $R(y)$  is a linear subspace of\/~$Y$ for every $y
        \in \Edomf\withrnpindex
        $.
        Assume also that $\regtanex \Edomf\withrnpindex (y)$ is a linear space for every $y\in \Edomf\withrnpindex$.
        Then for every $y\in \Edomf\withrnpindex$,
        $\UPhiplusR f(y)
        \cap \regtanex \Edomf\withrnpindex (y)
        $ is a linear subspace of\/~$Y$
        and $f'(y; u)$ depends linearly on $u\in \UPhiplusR f(y)
        \cap \regtanex \Edomf\withrnpindex (y)
        $.

\item\label{newtwo:lin} 
For any $\Edomf_{**}\subset Y$,
$\regtanex \Edomf_{**}(y)$ is a linear subspace of\/~$Y$, for all $y\in \Edomf_{**}$.

\item\label{two:lin}
        For every $y\in \Edomf\withrnpindex$,
        $\UPhi f(y)
        \cap \regtanex \Edomf\withrnpindex (y)
        $ is a linear subspace of\/~$Y$
        and
        Hadamard derivative
        $f'(y; u)$
        depends linearly on $u\in \UPhi f(y)
        \cap \regtanex \Edomf\withrnpindex (y)
        $.
\end{inparaenum}
\end{lemma}
\begin{proof}
We start the proof by quickly showing that~\itemref{one:lin} implies~\itemref{newtwo:lin} and~\itemref{two:lin}.
 	
        \itemref{one:lin}$\Rightarrow$\itemref{newtwo:lin}:
Use~\itemref{one:lin} with $R=[Y]$,
with a specific function $f=d_{\Edomf_{**}}$, 
with $Y$ instead each of  $\Edomf\withrnpindex$ and $\Edomf$,
with $Z\withrnpindex = Z \coloneq \R$ 
and $\Phir$ instead of $\Phi$.
It is important to recall that $\regtanex Y = [Y]$ by Example~\ref{ex:disttan:full} so that,
for every $y\in Y$, $\regtanex Y(y)$ is a linear space,
which is
a precondition for using~\itemref{one:lin}.
With $f=d_{\Edomf_{**}}$,
we obtain that $\bigl(\regtanex\Edomf_{**}(y)\bigr)\cap 
\regtanex Y(y)=
\regtanex\Edomf_{**}(y)\cap Y=
\regtanex\Edomf_{**}(y)$ is a linear subspace of~$Y$ for every
$y\in \Edomf_{**}$, 
cf.\ also Remark~\ref{r:ZsepPHIsep}\itemref{r:ZsepPHIsep:item:Ur}.
	
        \itemref{one:lin}$\Rightarrow$\itemref{two:lin}:
Assume that~\itemref{one:lin} is true; then we can use~\itemref{newtwo:lin} with $\Edomf_{**}\coloneq \Edomf\withrnpindex$. 
Therefore, $\regtanex \Edomf\withrnpindex (y)$ is a linear subspace of~$Y$.
        Then,  applying~\itemref{one:lin}  with
        $R=\regtanex\Edomf$,
        we obtain~\itemref{two:lin}.

        We now prove~\itemref{one:lin}.
        Let $y\in \Edomf\withrnpindex$.
        Let $u,v\in \UPhiplusR f(y)
        \cap \regtanex \Edomf\withrnpindex (y)
        $.
        Note that
        by Lemma~\ref{l:nporous}\itemref{l:nporous:nporous},
        $\Edomf\withrnpindex$
        is not porous
        at~$y$ in the directions of~$u$, $v$
        and $u+v \in \regtanex \Edomf\withrnpindex (y)$.
        We have $u, v \in R(y)$ by~\eqref{eq:HARsubDisttan}, and $u+v\in R(y)$ as
        the latter is a linear subspace of~$Y$.
        By Definition~\ref{def:assignmentPhi} we know that
$f'(y;u)$ and $f'(y;v)$ exist.
        We show that
        Hadamard derivative $f'(y; u+v) = f'(y;u) + f'(y;v)$ exists.

        Let $f\withrnpindex = f | _ { \Edomf\withrnpindex }$.
        As $\Edomf\withrnpindex$ is thick at $y$ in directions $u$ and $v$, we have that
\begin{equation}\label{eq:deriv-v}
    f'(y;u)
=
f\withrnpindex'(y;u)
\in Z\withrnpindex
\qquad\text{and}\qquad 
        f'(y;v)
        =
        f\withrnpindex'(y;v)
        \in Z\withrnpindex
		.
\end{equation}		
        Given $\tau>0$,
        we can use 
        the first property of
        admissibility of $\Phi
        |_{Z\withrnpindex}
        $~\eqref{eq:cond1} to
        choose $\varphi \in \Phi$ and $\varepsilon>0$ such that
        $B_\varphi(f\withrnpindex'(y;v), 2\varepsilon)
        \cap Z\withrnpindex
        \subset B
        _ {Z\withrnpindex }
        (f'(y;v), \tau)$.

        Since Hadamard derivative $f'(y;v)$ exists,
        we may fix
        $\delta_1, \omega_1 \in(0,\varepsilon)$ such that
        \(
         \hadamderset[\delta_1,\omega_1] f(y,v)
         \subset
         B
         ( f'(y;v), \varepsilon )
        \)
        and
        \(
         \hadamderset[\delta_1,\omega_1] f(y,-v)
         \subset
         B
         ( f'(y;-v), \varepsilon )
        \),
        see
        Remark~\ref{r:HDsetandder}.

Since $u\in \UPhiplusR f(y)$,
there is $\delta_0\in(0,\delta_1)$ such that,
for every
$\delta,\omega \in(0,\delta_0)$
and $\sigma \in \{\pm 1\}$,
\begin{align*}
        \hadamderset[\delta,\omega]f(y,\sigma u+v)
        \cap Z\withrnpindex
        &\subset
        f'(y;\sigma u)+B_\varphi(\hadamderset[\delta,\omega_1]f(y,v),\omega_1)
        \cap Z\withrnpindex
        \\
        &\subset
        f'(y;\sigma u)+B_\varphi( B_\varphi(f'(y;v), \varepsilon),\varepsilon)
        \cap Z\withrnpindex
        \\
        &\subset
        f'(y;\sigma u)+
        B_{Z\withrnpindex} 
        ( f'(y;v), \tau)
        =
        B_{Z\withrnpindex} 
        ( f'(y;\sigma u)+f'(y;v), \tau)
        ,
\end{align*}
where in the second inclusion we used~\eqref{eq:phi:repeat} of Lemma~\ref{l:Bphi}, so
        in particular
\begin{align*}
        \hadamderset[\delta,\omega]f\withrnpindex(y,\sigma u+v)
        \subset
        \hadamderset[\delta,\omega]f(y,\sigma u+v)
        \cap Z\withrnpindex
        \subset
        B_{Z\withrnpindex}( f'(y;\sigma u)+f'(y;v), \tau)
        .
\end{align*}
        The same is true also when $v$ is replaced by $-v$, though we might
        need to replace $\delta_0$ by a smaller
        number for this purpose.
        Since we obtain such $\delta_0$ for every $\tau > 0$, we deduce
        using
        Remark~\ref{r:HDsetandder}
        that
        Hadamard derivative $f\withrnpindex'(y; u+v)$
        exists and is equal to $f'(y;u) + f'(y;v)$. The uniqueness of the
        derivative is ensured by
        Remark~\ref{r:thickandunique}.
        Since $F\withrnpindex$ is not porous at~$y$ in the direction of $u+v$,
        we have by
        Lemma~\ref{lem:MZ:derext}
        that also
        $f'(y; u+v)=f\withrnpindex'(y; u+v)$ exists.

\medbreak

        We now show that $\UPhiplusR f(y)
        \cap \regtanex \Edomf\withrnpindex (y)
        $ is a linear subspace of $Y$.

        If \(u\in\UPhiplusR f (y) 
        \cap \regtanex \Edomf\withrnpindex (y)
        \) and $c\in\R\setminus\{0\}$,
        then
        $u \in R(y)$,
        $cu \in R(y)$
        and it is easy to see that
        \(cu\in\UPhiplusR f (y) \).
        We also have
        $0\in \UPhiplusR f(y)$.
        Indeed,
        $0\in R(y)$ by the assumption of linearity of $R(y)$
        and
        $f'_\HforHadam(y;0)=0$ always for a Lipschitz function $f$,
        which makes \eqref{eq:UPhiRegularity} with $u=0$ follow trivially from \eqref{eq:HDMonotOmega}
        and~\eqref{eq:phi:monot}.

        Now it remains
        to prove that
        $u+v\in \UPhiplusR f(y)$.
        For this we only need to show that
        for any fixed $\e>0$,
        $w\in R(y) \setminus \{ 0 \}$
        and $\varphi_0\in\Phi$ there exists $\delta_0>0$ such that, for all
        $\delta,\omega \in(0,\delta_0)$ and $\sigma \in\{\pm1\}$,
        \[
                \hadamderset[\delta,\omega]f(y,\sigma (u+v)+w)
                \subset
                f'(y;\sigma (u+v))+B_{\varphi_0}(\hadamderset[\delta,\e]f(y,w),\e)
                .
        \]

        For $\sigma \in \{\pm 1\}$, let
        $\varphi_{
        \sigma
        }
        (x)=\varphi_0(x-f'(y;
        \sigma
        v))$,
        $x\in Z$;
        then $\varphi_{
        \sigma
        }
        \in\Phi$ by
        the assumption that $\Phi$ satisfies \eqref{eq:cond2} with $z\coloneq f'(y;\sigma v)\in Z\withrnpindex$, 
        which is satisfied due to~\eqref{eq:deriv-v}.
        Note that obviously for any $\sigma \in\{\pm1\}$,
\begin{equation}\label{eq:shiftedB}
          B_{\varphi_{
        \sigma
        }} ( f'(y;\sigma v) + A, \eta ) = f'(y;\sigma v) + B_{\varphi_0} ( A, \eta)
\end{equation}
        for any $A\subset Z$ and $\eta>0$.

Since
$v\in \UPhiplusR f(y)$, we apply Definition~\ref{def:assignmentPhi}  with $\varphi=\varphi_0$
to find $\delta_1 \in(0,\e/2)$ such that for all $\delta,\omega\in(0,2\delta_1)$ and $\sigma \in\{\pm1\}$,
the following holds:
\begin{align}
        \label{phi0}
        &\hadamderset[\delta,\omega]f(y,\sigma v+w)
        \subset
        f'(y;\sigma v)+B_{\varphi_0}(\hadamderset[\delta,\e/2 ]f(y,w),\e/2)
        .
\end{align}
Similarly, since
$u\in \UPhiplusR f(y)$, we
find $\delta_0 \in(0,\delta_1)$ such that for all
$\delta,\omega\in(0,\delta_0)$ and $\sigma \in\{\pm1\}$,
\begin{equation}        
\label{phi1}
\begin{split}
        \hadamderset[\delta,\omega]f(y,\sigma u+(v+w))
        &\subset
        f'(y;\sigma u)+B_{\varphi_{ 1 }}(\hadamderset[\delta,\delta_1]f(y,v+w),\delta_1)
        \text{ and}\\
               \hadamderset[\delta,\omega]f(y,\sigma u+(-v+w))
        &\subset
        f'(y;\sigma u)+B_{\varphi_{
        -1
         }}(\hadamderset[\delta,\delta_1]f(y,-v+w),\delta_1).
\end{split}        
\end{equation}

        Let $\delta,\omega\in(0,\delta_0)$ and $\sigma \in\{\pm1\}$. Then using
        $\delta_1 <\e/2$, applying \eqref{phi1}, \eqref{phi0},
        \eqref{eq:shiftedB}
        and, finally, \eqref{eq:phi:repeat}, we get
\begin{align*}
        \hadamderset[\delta, \omega]f(y,\sigma (u+v)+w)
        &\subset
        f'(y;\sigma u)+B_{\varphi_{
        \sigma
        }}(\hadamderset[\delta,\delta_1]f(y,\sigma v+w),\e/2)\\
        &\subset
        f'(y;\sigma u)+B_{\varphi_{
        \sigma
        }}(f'(y;\sigma v)+B_{\varphi_0}(\hadamderset[\delta,\e
        ]f(y,w),\e/2),\e/2)\\
        &=
        f'(y;\sigma u)+f'(y;\sigma v)+B_{\varphi_0}(B_{\varphi_0}(\hadamderset[\delta,\e]f(y,w),\e/2),\e/2)\\
        &\subset
        f'(y;\sigma u)+f'(y;\sigma v)+B_{\varphi_0}(\hadamderset[\delta,\e]f(y,w),\e).
\end{align*}
        Therefore,
        $u+v\in \UPhiplusR f(y)$ and
        $\UPhiplusR f(y)
        \cap \regtanex \Edomf\withrnpindex (y)
        $ is a linear subspace of $Y$.
        We already proved that $f'(y; u)$ depends additively on $u\in \UPhiplusR f(y)
        \cap \regtanex \Edomf\withrnpindex (y)
        $;
        the homogeneity is obvious.
\end{proof}

\begin{lemma}\label{l:UPhi-closed-cnts}
        Let $Y,
        Z$ be Banach spaces, $
        \Edomf\subset Y$ and 
        $f\fcolon \Edomf\to Z$ a Lipschitz function.
        Suppose $\Phi \subset \Lip_1(Z)$.

\begin{inparaenum}[\textup\bgroup(1)\egroup]
\item \label{one:cl}
        Let $R\fcolon \Edomf\to 2^Y$ be an assignment such that
        $\Edomf$ is not porous at~$y$ in the direction of~$u$ for
        every $y\in\Edomf$ and $u\in R(y)$, and that $R(y)$  is a
        closed linear subspace of~$Y$ for every $y\in\Edomf$.

        Then $\UPhiplusR f(y)
        $ is a closed subset of~$Y$ and
        $f'(y;\cdot)$ is continuous on $\UPhiplusR f(y)
        $, for every
        $y\in\Edomf$.

        \item\label{two:cl}
        $\UPhi f(y)$ is closed and $f'(y;\cdot)$ is continuous on $\UPhi f(y)$, for every $y\in\Edomf$.

        \item\label{three:cl}
        $\regtanex \Edomf(y)$ is a closed linear subspace of $Y$ for each $y\in \Edomf$.
\end{inparaenum}
\end{lemma}
\begin{proof}
        We first show that~\itemref{one:cl}
        implies~\itemref{two:cl}
        and~\itemref{three:cl}.
        Indeed, $\Phir(\R)$ is admissible by
        Remark~\ref{r:ZsepPHIsep}\itemref{r:ZsepPHIsep:item:Ur}.
        Hence we may use~\itemref{one:cl}
        with $R=[Y]$,
        $Z= \R$,
        $f=d_\Edomf$
        and
        $\Edomf$ replaced by $Y$
        to see
        that $\regtanex\Edomf(y)$ is a closed
         subset of $Y$, for each $y\in F$. In addition, by Lemma~\ref{l:linearity}\itemref{newtwo:lin}, we have that $\regtanex\Edomf(y)$ is a linear
         subspace of $Y$ for each $y\in F$.
         Finally,
        Lemma~\ref{l:nporous}\itemref{l:nporous:nporous} allows us to
        apply~\itemref{one:cl} again, now with 
        $R(y)= \regtanex\Edomf(y)$,
        $\Edomf$ and $\Phi$
        to get~\itemref{two:cl}.

        Now we prove~\itemref{one:cl}.
        Assume that $f$ is $K$-Lipschitz, $K>0$.
        We use Remark~\ref{r:thickandporous} to make sure $\Edomf$ is thick at $y\in\Edomf$ in all directions $u\in R(y)$.
        We now show that $\UPhiplusR f(y)$ is closed in $Y$.

We first note that $0\in\UPhiplusR f(y)$. Indeed, as $y\in\Edomf$ and $R(y)$ is a linear subspace of $Y$, we have $0\in R(y)$. Moreover, $f'_\HforHadam(y;0)=0$ as $f$ is Lipschitz; so~\eqref{eq:UPhiRegularity} is satisfied with $u=0$, by~\eqref{eq:HDMonotOmega} and~\eqref{eq:phi:monot}.

Suppose $\bar u \in \closure{ \UPhiplusR f(y) }
        \setminus \{0\}
        $.
For $n\in \N$, let $u_n \in  \UPhiplusR f(y)
$ be such that $u_n\to\bar u$.
Without loss of generality we may assume $u_n\ne0$ for all $n\ge1$.
Let $z_n=f'(y; u_n)$ be the Hadamard derivative, $n\in \N$.
        We have $u_n\in R(y)$ by~\eqref{eq:HARsubDisttan}
        and $\bar u\in R(y)$ as $R(y)$ is closed.
        Hence $\Edomf$ is neither porous at~$Y$ in the direction
        of~$u$ nor in the direction of~$u_n$ by the assumptions.
The sequence $\{u_n\}\subset R(y)$ is Cauchy and hence, by Lemma~\ref{l:DerByDer},
$\{z_n\}$ is also Cauchy. Hence there is $\bar z\in Z$ such that $z_n\to \bar z$.

        Fix $\eta_0>0$.
        We can
        choose
        an $m$ with $\norm{z_m - \bar z} < \eta_0$
        and $\norm{u_m - \bar u} < \eta_0$.
        Then we find $\eta \in (0,\eta_0)$ such that,
        for all $\delta,\omega \in (0, 2\eta)$,
        \begin{equation}\label{eq:45}
        \hadamderset[\delta, \omega] f(y,u_m) \subset  B(z_m, \eta_0)
        \qquad\text{and}\qquad
        \hadamderset[\delta, \omega] f(y,-u_m) \subset  B(-z_m, \eta_0)
        .
        \end{equation}
Applying Lemma~\ref{l:HDbyHD} we find $\delta_0\in (0,2\eta)$ such that,
for every $\delta, \omega \in (0,\delta_0)$,
        \[
        \hadamderset[\delta, \omega] f(y,\bar u)
        \subset
        \hadamderset[\delta, \eta] f(y,u_m) + B(0_Z,(\norm{u_m - \bar u}+\omega+\eta)K).
        \]
Combining that with~\eqref{eq:45} and $\norm{z_m - \bar z} < \eta_0$
we obtain
        \[
        \hadamderset[\delta, \omega] f(y,\bar u)
        \subset
        B(z_m, \eta_0) + B(0_Z,(\norm{u_m - \bar u}+\omega+\eta)K )
        \subset
        B(\bar z, \eta_0 + 4 \eta_0 K+\eta_0 )
        \]
for every $\delta, \omega \in (0,\delta_0)$.
Since,
        for similar reasons,
        the same inclusion holds true for
        $-\hadamderset[\delta, \omega] f(y, - \bar u)$,
        and since
        $\eta_0>0$ was arbitrary, this proves that the Hadamard derivative of $f$ at $y$ in the direction of $\bar u$ exists and is equal to $\bar z$.
        Moreover, $f'(y;-\bar u)$ exists too and is equal to $-f'(y;\bar u)$.
        Recall that $\bar u \in R(y)$.
\smallbreak

Fix $v\in R(y)\setminus\{0\}$,
$\varepsilon >0$, $\varphi\in \Phi$.
To show that $\bar u \in  \UPhiplusR(f,y) $,
it is enough to show there is $\delta_0>0$ such that,
for every
$\delta, \omega\in (0,\delta_0)$ 
and $\sigma\in\{\pm1\}$
\begin{equation}\label{eq:need55}
            \hadamderset[\delta,\omega] f(y, \sigma\bar u + v)
        \subset
             f'(y; \sigma\bar u)
            +
             B_\varphi( \hadamderset[
             \delta, \varepsilon
             ] f(y,v), (3K+2) \varepsilon),        
\end{equation}
see~\eqref{eq:UPhiRegularity},
\eqref{eq:HDMonotOmega}
and~\eqref{eq:phi:monot}.
Note that since $R(y)$ is a linear space, it is enough to prove~\eqref{eq:need55} for $\sigma=+1$ only.

Since
$R(y)$
is a closed linear space 
and 
$u_n \in \UPhiplusR(f,y) \subset R(y)$,
$v \in
R(y)$ for all $n\in\N$,
we conclude that both $\bar u+v$ and $u_n+v$ belong to
$R(y)$,
and hence
$\Edomf$ is not porous at~$y$
in these directions.
We choose $k\in \N$ such that
$\norm{u_k - \bar u} < \varepsilon$
and
$\norm{z_k - \bar z} < \varepsilon$.
Since $u_k \in  \UPhiplusR(f,y) $,
there is $\delta_1 \in (0,\varepsilon/2)$ such that,
for every
$\delta\in (0,2\delta_1)$,
\begin{equation}\label{eq:have55}
            \hadamderset[\delta,\delta_1] f(y, u_k + v)
        \subset
             f'(y; u_k)
            +
             B_\varphi( \hadamderset[
             \delta, \varepsilon
             ] f(y,v), \varepsilon)
        .
\end{equation}
By Lemma~\ref{l:HDbyHD} used with $\eta=\delta_1$
there is $\delta_0\in (0, \delta_1)$ such that
        for every $\delta \in (0, \delta_0)$ and $\omega>0$,
        \begin{equation}\label{eq:111}
        \hadamderset[\delta, \omega] f(y,\bar u+v)
        \subset
        \hadamderset[\delta, \delta_1] f(y,u_k+v) + 
        B(0_Z,(\norm{u_k - \bar u}+\omega+\delta_1)K )
        .
        \end{equation}
Note that $\norm{z_k - \bar z} < \varepsilon$ means
        \begin{equation}\label{eq:112}
        f'(y;u_k) \in f'(y;\bar u) + B(0_Z, \varepsilon)
        .
        \end{equation}
        Combining \eqref{eq:111}, \eqref{eq:112}, \eqref{eq:have55} and~\eqref{eq:phi:addout} we obtain~\eqref{eq:need55}
        for every $\delta, \omega\in (0,\delta_0)$,
        hence $\bar u \in  \UPhiplusR(f,y) $.

        Finally, for
        the
        continuity of $f'(y;\cdot)$ on $\UPhiplusR f(y)\setminus\{0\}$, recall that $z_n=f'(y;u_n) \to \bar z=f'(y;\bar u)$. 
Recall also that 
$0\in \UPhiplusR f(y)$. It remains to note that
the continuity of $f'(y;\cdot)$ at~$0$ follows from the Lipschitz property of~$f$.
\end{proof}
\begin{remark}\label{r:closedlinear}
By
Lemma~\ref{l:linearity}
and Lemma~\ref{l:UPhi-closed-cnts} we have that
if $\Phi$ is admissible 
and $f$ is a Lipschitz function
then
$\UPhi f(y)$ is a closed linear subspace of $Y$ for every $y\in\Edomf$.
To see this we let
$\Edomf\withrnpindex=\Edomf_{**}=\Edomf$, 
$Z\withrnpindex=Z$ and use~\eqref{eq:HAsubDisttan}.
\end{remark}
\begin{proposition}\label{p:UPhiIsAssignment}
	Let $Y$, $Z$ be Banach spaces, $\Edomf \subset Y$ and $\Phi \subset \Lip_1(Z)$  admissible.
	Assume
	$f\fcolon \Edomf\to Z$
	is a Lipschitz function.
	Then $\EAPhi f$ is a Hadamard derivative domain assignment
	for~$f$.
\end{proposition}
\begin{proof}
	We show that $\EAPhi f=
           \Hindexedby{\frac12 \tildePhi, [Y] } \tilde f (y)
\cap
\regtanex \Edomf(y)$ defined for $y\in \Edomf$, see Remark~\ref{r:onehalf}, 
	satisfies all properties required in Definition~\ref{def:assignment}.
Let us start with $\tilde Z$ and $\tilde f\fcolon Y\to\tilde Z$ as in~\eqref{eq:extensionBL_codomain} and~\eqref{eq:extensionBL}, see Lemma~\ref{l:extf}.
Let also $Z_0 = \closedSpan f(\Edomf)\subset\tilde Z$,
$\Phi_0 = \Phi|_{Z_0}$ be the collection of restrictions
and let $\tildePhi\subset\Lip_{2}(\tilde Z)$
be as in~\eqref{eq:extensionPhi}, see also Lemma~\ref{l:extphi}.

	For every $y\in \Edomf$, $\EAPhi f(y)$ is a closed linear subspace of $Y$ by 
	Lemma~\ref{l:linearity} and Lemma~\ref{l:UPhi-closed-cnts} applied with $(\Edomf\withrnpindex,\Edomf,\Phi,f)\coloneq (\Edomf,Y,\frac12\tildePhi,\tilde f)$.
	If $u\in \EAPhi f(y)$,
	then $u\in \regtanex \Edomf(y)$, so by
        Lemma~\ref{l:nporous}\itemref{l:nporous:thick},
        $\Edomf$~is thick at~$y$ in the direction of~$u$, which
	verifies property~\itemref{def:assignment:item:notporous} of Definition~\ref{def:assignment}.
	For
	property~\itemref{def:assignment:item:hadamDer},
	note that by the
	Definition~\ref{def:assignmentPhi} used for $\Hindexedby{\frac12 \tildePhi, [Y] } \tilde f (y)$, the derivative
$\tilde f'(y;u)$ exists. Therefore $f'(y;u)$ also exists as $f=\tilde f|_F$.
	What remains to prove is property~\itemref{def:assignment:item:contlin}.
	Linearity of $L_y=f'(y;\cdot)$ follows from Lemma~\ref{l:linearity}
	and its continuity from Lemma~\ref{l:UPhi-closed-cnts}.
\end{proof}

The following easy lemma shows that linearity of Hadamard derivative combined with the thickness of $\Edomf$ has very strong implications for the derivative assignment.

\begin{lemma}\label{l:easyHA}
	Let $Y,Z$ be Banach spaces,
	$f\fcolon F\subset Y \to Z$ a function.
	\begin{enumerate}
		\item \label{l:easyHA:item:one}
		Assume $\Phi \subset \Lip_1(Z)$ and $R\fcolon F\to 2^Y$ is an assignment
		such that 		
		$F$ is thick at~$y$ in the direction of~$u$
		for every $y\in F$ and every $u \in R(y)$.
		If,
		for a fixed $y\in F$,
		$R(y)$ is a linear space
		and
		Hadamard derivative
		$f'(y;\cdot)$
		is a linear, not necessarily continuous, map on $R(y)$,
		then $\UPhiplusR f(y) = R(y)$.
		\item \label{l:easyHA:item:two}
		If $y\in F$ and $f'(y;\cdot)$
		is a linear map on
		$\regtanex F(y)$,
		then $\UPhi f(y) = \regtanex F(y)$.
		\item \label{l:easyHA:item:three}
		If $M\subset Y$ is a set, $y\in M$
		and, for every $u\in Y$,
		$M$ is not porous at~$y$ in the direction of~$u$,
		then $\regtanex M(y) = Y$.
	\end{enumerate}
\end{lemma}
\begin{proof}
	\itemref{l:easyHA:item:one}
	Let us fix $y\in F$ 
	such that
	$L(u)=f'(y;u)$
	is a linear map on the linear subspace $R(y)$.
	The inclusion $\UPhiplusR f(y) \subset R(y)$ follows by Definition~\ref{def:assignmentPhi}.
	To prove the reverse inclusion,
	let $u\in R(y)$, $\varepsilon>0$, $\varphi\in\Phi$ and $v\in R(y)\setminus \{0\}$
	be given.
	Since $R(y)$ is linear, we have $\sigma u +v \in R(y)$ for all $\sigma \in \{\pm 1\}$, so the Hadamard derivative	 
	$f'(y;\sigma u
	+v
	) = L(\sigma u
	+v
	)$
	exists.	Therefore, we can find $\delta_0 > 0$ such that
	\begin{equation*}
		\hadamderset [\delta,\omega] f(y, \sigma u
		+v
		) \subset B(L(\sigma u
		+v
		), \varepsilon / 2 )
	\end{equation*}
	for all $\sigma \in \{ \pm 1 \}$
	and
	$\delta, \omega \in (0, \delta_0)$,
	see Remark~\ref{r:HDsetandder}.
	Since $L(
	v
	) =f'(y;v)\in  \closure{ \hadamderset [\delta,\varepsilon] f(y, v) } $
	by Remark~\ref{r:derinHA},
	we have,
	for every $\sigma \in \{\pm 1\}$
	and $\delta, \omega \in (0, \delta_0)$,
	\begin{align*}
		\hadamderset [\delta,\omega] f(y, \sigma u + v )
		&\subset
		B(L(\sigma u + v), \varepsilon / 2 )
		\\
		&\subset
		L(\sigma u) + \closure{ \hadamderset [\delta,\varepsilon] f(y, v) }  + B(0, \varepsilon / 2)
		\\
		&\subset
		f'(y; \sigma u) + B_\varphi ( \hadamderset [\delta,\varepsilon] f(y, v), \varepsilon )
	\end{align*}
	where we 
	used~\eqref{eq:phi:ball} in
	the last inclusion.
	Hence $u\in\UPhiplusR f$ by Definition~\ref{def:assignmentPhi}.
	
	\itemref{l:easyHA:item:two}
	Recall that $\regtanex F(y)$ is a linear subspace of $Y$ by Lemma~\ref{l:linearity}\itemref{newtwo:lin}. Hence~\itemref{l:easyHA:item:two} 
	follows immediately from~\itemref{l:easyHA:item:one},
	cf.\ Definition~\ref{def:assignmentPhi}
	and Lemma~\ref{l:nporous}\itemref{l:nporous:thick}.

	\itemref{l:easyHA:item:three}        
	This 
	follows by the application of~\itemref{l:easyHA:item:one} and~\eqref{eq:regtanex:def}
	together with Remark~\ref{r:derandporous}.
\end{proof}

\bigbreak

        For the case of 
        Lipschitz
        functions $f\fcolon \Edomf=Y\to Z$ defined on the whole space, we may conclude
        $\UPhi f(\cdot)=\EAPhi f(\cdot)$, see Remark~\ref{r:HAisH}, and
        in Lemma~\ref{l:UMPequalsUH} below we compare this to $\UPhiMP(f,\cdot)$ defined in~\cite{MP2016}.
So
for the rest of this section we assume $\Edomf=Y$.
We follow~\cite[\S2]{MP2016} and define  
the $\delta$-derived set $\derset[\delta] f(y,e)$  of $f$ at $y\in Y$ in the direction
        of $e\in Y$ by
\begin{equation}\label{eq:dersetDEF}
\derset[\delta] f(y,e)=
\Bigl\{\frac{f(y+te)-f(y)}{t}:0<t<\delta\Bigr\}
.
\end{equation}

\begin{lemma}\label{l:inclusions}
        Let $Y,Z$ be Banach spaces  and
        $f\fcolon Y\to Z$ be $K$-Lipschitz,
        $K>0$.
        Then
        \begin{align}
                \label{eq:DinHD}
                        \derset[\delta] f(y,e)
                        &\subset
                        \hadamderset[\delta,\omega] f(y,e),
                \\
                \label{eq:HDinD}
                        \hadamderset[\delta,\omega] f(y,e)
                        &\subset
                        \derset[\delta] f(y,e)
                        + B(0_Z, K \omega),
                \\
                \label{eq:HDinHD}
                        \hadamderset[\delta,\omega] f(y,e)
                        &\subset
                        \hadamderset[\delta,\bar \omega] f(y,e)
                        + B(0_Z, K \abs{\omega-\bar\omega})
        \end{align}
        for every $y,e\in Y$ and $\delta,\omega,\bar\omega>0$.
   \end{lemma}
\begin{proof}
        The first inclusion follows from Definition~\ref{def:Hadamard-derived-set} and~\eqref{eq:dersetDEF}. 
        
        For the second, let
        $v=(f(y+ t u)-f(y))/t \in  \hadamderset[\delta,\omega] f(y,e)$
        be associated to $t\in (0,\delta)$ and $u\in B(e,\omega)$.
        Then the inequality
        \[
               \norm{
                \frac{ f(y+ t u)-f(y) }{t}
                -
                \frac{ f(y+ t e)-f(y) }{t}
               }
             \le
                \frac{ Kt\norm{u-e} }{t}
             <
                K \omega
        \]
        concludes the proof of~\eqref{eq:HDinD}.
        	 
        The third inclusion is obvious if $\omega \le \bar\omega$, cf.~\eqref{eq:HDMonotOmega}.
        In
        the
        case $\omega>\bar\omega$, let
        $v=(f(y+ t u)-f(y))/t \in  \hadamderset[\delta,\omega] f(y,e)$
        be associated to $t\in (0,\delta)$ and $u\in B(e,\omega)$.
        Let $\bar u= e+\frac{\bar\omega}{\omega}(u-e)$. Then $\bar u \in B(e, \bar\omega)$
        and
        \[
               \norm{
                \frac{ f(y+ t u)-f(y) }{t}
                -
                \frac{ f(y+ t \bar u)-f(y) }{t}
               }
             \le
                \frac{ Kt\norm{u-\bar u} }{t}
             \le
                K\abs{\omega-\bar \omega}
                \frac{ \norm{u-e} }{\omega}
             <
                       K\abs{\omega-\bar \omega}
        \]
        proves~\eqref{eq:HDinHD}.
\end{proof}

\begin{lemma}\label{l:UMPequalsUH}
Let $Y,Z$ be separable and $\Phi\subset\Lip_1(Z)$ be admissible.
             If $f\fcolon Y\to Z$ is a
             Lipschitz function
             defined on the whole space
             then $\UPhiplus f(\cdot) = \EAPhi f(\cdot)=
             \UPhiMP (f,\cdot)$.
             In other words, in such a case
             the Hadamard derivative assignments from Definition~\ref{def:assignmentPhi} and Definition~\ref{def:assignmentEA}
             agree
             with the derivative assignment $\UPhiMP (f,\cdot)$ of\/ \cite[Proposition~5.1(vi)]{MP2016}.
\end{lemma}
\begin{proof}
We note that $\UPhi f=\EAPhi f$, see Remark~\ref{r:HAisH}.
        Recall that,
        by Example~\ref{ex:disttan:full},
        $\regtanex_y Y=Y$ for any $y\in Y$.

        Assume that $f$ is $K$-Lipschitz, $K>0$, and fix any $y\in Y$.
        Assume first that $u\in \UPhiMP(f,y)$.
        Fix $\varepsilon>0$, $\varphi \in \Phi$, $v\in Y$
        and find $\delta_0$ as in \cite[Proposition~5.1(vi)]{MP2016}.
        Then, for every $\delta\in(0,\delta_0)$, $\omega\in(0,\varepsilon/K)$ and $\sigma \in\{\pm1\}$,
        we use Lemma~\ref{l:inclusions}, \cite[Proposition~5.1(vi)]{MP2016} and Lemma~\ref{l:Bphi} to get
        \begin{align*}
                \hadamderset[\delta, \omega] f(y,\sigma u+v)
                &\subset
                \derset[\delta]f (y,\sigma u+v) + B(0_Z, K\omega)
                \\
                &\subset
                f'(y;\sigma u)
                + B_\varphi( \derset[\delta]f (y,v), \varepsilon )
                + B(0_Z, K\omega)
                \\
                &\subset
                f'(y;\sigma u)
                + B_\varphi( \derset[\delta]f(y,v), \varepsilon + K\omega )
                \\
                &\subset
                f'(y;\sigma u)
                + B_\varphi( \hadamderset[\delta,\varepsilon]f(y,v), \varepsilon + K\omega )
                                \\
                &\subset
                f'(y;\sigma u)
                + B_\varphi( \hadamderset[\delta,2\varepsilon]f(y,v), 2\varepsilon),
        \end{align*}
        where we use \eqref{eq:HDinD} for the first inclusion, followed by \eqref{eq:phi:addout},  \eqref{eq:DinHD} and \eqref{eq:phi:monot},
        and finally
        \eqref{eq:HDMonotOmega}
        and \eqref{eq:phi:monot}
        for the last three inclusions.
        We conclude 
        that $u\in \UPhiplus(f,y)$.

        Let now $u\in \UPhiplus(f,y)$.
        Fix any $\varepsilon>0$, $\varphi \in \Phi$, $v\in Y$
        and
        find $\delta_0$ as in Definition~\ref{def:assignmentPhi}.
        Choose any
        $\omega\in(0,\delta_0)$
        and observe that,
        for every $\delta\in(0,\delta_0)$ and  $\sigma \in\{\pm1\}$, by Lemma~\ref{l:inclusions}, Definition~\ref{def:assignmentPhi} and Lemma~\ref{l:Bphi},
        \begin{align*}
                \derset[\delta] f(y,\sigma u+v)
                &\subset
                \hadamderset[\delta,\omega]f (y,\sigma u+v)
                \\
                &\subset
                f'(y;\sigma u)
                + B_\varphi( \hadamderset[\delta,\varepsilon]f (y,v), \varepsilon )
                \\
                &\subset
                f'(y;\sigma u)
                + B_\varphi( \derset[\delta]f(y,v) + B(0_Z, K\varepsilon) , \varepsilon  )
                \\
                &\subset
                f'(y;\sigma u)
                + B_\varphi( \derset[\delta]f(y,v) , \varepsilon + K\varepsilon ),
        \end{align*}
where we use \eqref{eq:DinHD} for the first inclusion, followed by \eqref{eq:HDinD} and \eqref{eq:phi:monot} for the penultimate and \eqref{eq:phi:addin} for the last inclusions. Finally, since $K$ is a fixed constant, this shows that $u\in \UPhiMP(f,y)$.
\end{proof}

\begin{remark}
\label{rem:MPcomplete}
        Let $Y$, $Z$ be separable.
   For $y\in Y$, let $\UPhiMP(f,y)\subset Y$ denote
   the value at $y$ of the derivative assignment
   for~$f$ from \cite[Proposition~5.1(vi)]{MP2016}, as in Lemma~\ref{l:UMPequalsUH} above.
        Then
        $\UPhiMP(f,\cdot) = \UPhiplus f(\cdot)=\EAPhi f(\cdot)$
        is a complete derivative assignment
        in the sense of \cite[p.~4711]{MP2016}
        by \cite[Proposition~5.2]{MP2016}.
 Later we will prove that $\UPhiMP(f,\cdot)$ is complete in the
        sense of Definition~\ref{def:completeAssignment}, see Theorem~\ref{t:MP52}.
         Compared to \cite{MP2016}, 
where the test mappings were required to be defined on the whole space $X$,
         our definition of completeness allows the test mappings
        $\testfunX\fcolon G\subset X\to Y$
        to have proper subsets of space $X$ as their domains.
        On the other hand, we require the exceptional set to belong to
        (possibly different) ideal $\lone$.
\end{remark}

\section{Completeness of Hadamard derivative assignment $\EAPhi f$}\label{s:step3}

        The goal of this section is to prove, in Theorem~\ref{t:MP52}, that
        \( \EAPhi f\)
        is a complete assignment,
        for a Lipschitz function $f\colon\Edomf\subset Y\to Z$.
        We divide the task into smaller
        steps. We first aim to decide when
        $\UPhiplusR f$
        is a complete assignment
        for any complete assignment $R$, see Lemma~\ref{l:MP52}. To do so, for every Lipschitz function
        $g\colon G\subset X\to Y$ we need to identify an exceptional set
        $N\in\lone(X)$ and then to verify that the
        condition~\eqref{eq:propofN}
        is satisfied for all $x$
        outside the set $N$. Therefore, the bulk of this section is devoted to
        the work which helps us prepare these steps.

In this section we always have $X,Y,Z$ Banach spaces and $\Phi\subset\Lip_1(Z)$ separable, see Definition~\ref{def:phi}.

\begin{lemma}\label{l:MP2.21}
 Suppose 
 $K>0$,
 $\Edomf\subset Y$.
 Let
 $f\fcolon \Edomf\to Z$ be a $K$-Lipschitz map,
 $v,w\in Y$, $\varphi \in \Lip_1(Z)$
 and
 $\varepsilon,\omega_3, \omega_4, \omega_5 \in (0,\infty)$.
 Then there is a $\sigma$-$v$-porous set $P\subset Y$ such that
 for every $y\in \Edomf\setminus P$,
 there exists $\delta_0>0$ such that
 the inclusion
 \begin{equation}\label{eq:MP2.8}
        \hadamderset[\delta,\omega_3] f(y,v+w)
        \subset
        \hadamderset[\delta,\omega_4] f(y,v)
        +
        B_\varphi(\hadamderset[\delta,\omega_3+\omega_4+\omega_5] f(y,w), \varepsilon)
 \end{equation}
 holds for all
 $\delta\in (0,\delta_0)$.
\end{lemma}
\begin{proof}
        Note first that if $v=0$, then we can take $P=\emptyset$. Indeed, \eqref{eq:MP2.8} then holds for all $y\in\Edomf$
        and $\delta>0$,
        using
        $0\in \hadamderset[\delta,\omega_4] f(y,0)$,
        Lemma~\ref{l:Bphi} and~\eqref{eq:HDMonotOmega}.

        In what follows, assume $v\ne0$.
        Note that $ \hadamderset[\delta,\omega_3+\omega_4+\omega_5] f(y,w) \subset B(0_Z, (\norm{w}+\omega_3+\omega_4+\omega_5)K) $ for any $y\in \Edomf$,
        cf.~Remark~\ref{r:HDlip}.
        Recall that $\varphi$ is $1$-Lipschitz, so it is bounded on bounded sets, thus
        the function $
                \delta
                \mapsto
                \phi_y(\delta)
                \coloneq
                                \sup\bigl\{\varphi(z)\setcolon z\in \hadamderset[\delta,\omega_3+\omega_4+\omega_5] f(y,w) )\bigr\}$
        is bounded from below. Since $\phi_y(\cdot)$ is also non-decreasing, we have that
                $c_{y} \coloneq \inf_{\delta>0}  \phi_y(\delta)
        \in\R $ is finite.
Hence,
        for every $y\in \Edomf$
        there is a rational 
        $\tau>0$ such that 
        for every $\delta \in (0, \tau)$, it holds that
        \[
                c_{y} \le \phi_y(\delta)
        < c_{y} + \varepsilon / 4.\]
        This implies that for every $y\in\Edomf$ there is a rational 
        $c\le c_{y}$ 
        and a rational $\tau>0$ such that for every $\delta \in (0, \tau)$
\begin{equation}
                c   \le 
                \phi_y(\delta)
                < c + \varepsilon / 2.
                \label{eq:cplusEps}
        \end{equation}
        For every rational $c$ and every rational $\tau>0$,
        let $P_{c,\tau}$ be the set of all $y\in \Edomf$ which have property~\eqref{eq:cplusEps}
        for every $\delta \in (0, \tau)$ and for which \eqref{eq:MP2.8} fails
        for arbitrarily small $\delta>0$. Let $P$ be the union of $P_{c,\tau}$ over all such $c,\tau$.
        It suffices to show that each $P_{c,\tau}$ is porous in the direction of~$v$.
For this, 
let
        $
               \eta
             \coloneq
               \min(
                 \omega_4
                 ,
                 \omega_5
                 ,
                 \varepsilon / (4K)
               )
               .
        $
        To prove
        the
        porosity of $P_{c,\tau}$ it is enough to consider any $y\in P_{c,\tau}$ and                
        show that
        for every $t_0>0$ there is $t\in (0,t_0)$ such that
\begin{equation}\label{eq:hole}
        B( y + t v, t \eta / 2 ) \cap P_{c,\tau} = \emptyset
        .
\end{equation}
        
        So assume $t_0>0$ is fixed.
        By the definition of $P_{c,\tau}$, there is
        $\delta \in (0, \min(t_0, \tau))$
        for which
        \eqref{eq:cplusEps} is true
        and
        \eqref{eq:MP2.8} fails.
        Find $t\in(0,\delta)$ and
        $\hat w\in B(w,\omega_3)$,
        with $y+t(v+\hat w) \in \Edomf$ and
        \[
                        \frac{
                                f(y+t (v+\hat w)) - f(y)
                        }{ t}
                \notin
                        \hadamderset[\delta,\omega_4] f(y,v)
                        +
                        B_\varphi(\hadamderset[\delta,\omega_3+\omega_4+\omega_5] f(y,w), \varepsilon)
                .
        \]
        Assume
        that~\eqref{eq:hole} is
        not
        true.
        Since $P_{c,\tau} \subset \Edomf$, we have
        $B( y + t v, t \eta / 2 ) \cap \Edomf \neq \emptyset$,
so
        we can find
        $\hat v\in B( v, \eta / 2 )$
        with $\bar y\coloneq y + t \hat v \in \Edomf$.
        Since
        $
                        \bigl(
                                f(y+t \hat v) - f(y)
                        \bigr) / t
                \in
                        \hadamderset[\delta,\omega_4] f(y,v)
        $, we can infer that
        \[
                        \bigl(
                                f(y+t (v+\hat w)) - f(\bar y)
                        \bigr) / t
                \notin
                        B_\varphi(\hadamderset[\delta,\omega_3+\omega_4+\omega_5] f(y,w), \varepsilon),
        \]
        which means
        \begin{equation}
        \label{eq:phi_y_bar}
                \varphi
                \Bigl(
                        \bigl(
                                f(y+t (v+\hat w)) - f(\bar y)
                        \bigr) / t
                \Bigr)
             \ge
                \varepsilon +
                \phi_y(\delta)
             \ge
                \varepsilon + c.
        \end{equation}
Let 
$\bar P$ be the set of all $x\in \Edomf$ such that
\begin{equation}\label{eq:phi_x}
\sup\bigl\{\varphi(z)\setcolon z\in \hadamderset[\delta, \omega_3+\omega_4 + \eta] f(x,w)\bigr\} 
< c+ \varepsilon /2.
\end{equation}
We show that
$\bar P$
 is disjoint from $B(\bar y, t\eta)$.
If there is
$x\in B(\bar y, t \eta ) \cap \bar P$,
then
\begin{equation}\label{eq:567}
	\norm{ \frac{f(y+t(v+\hat w)) - f(\bar y) }{ t } - \frac{f( y+t(v+\hat w)) - f( x) }{ t }     }
	\le \frac{ K \norm{x-\bar y} }{t} < \frac{ \varepsilon }{ 2 }.
\end{equation}
Using~\eqref{eq:phi_y_bar}, followed by~\eqref{eq:567} and, finally, by
\[\frac{
	y+t(v+\hat w) - x } {t} 
=\hat w+(v-\hat v)+\frac{\bar y-x}{t}
\in
B(w, \omega_3+\eta/2 + \eta )
\subset
B(w, \omega_3+\omega_4 + \eta )
\] and~\eqref{eq:phi_x},
we
infer
that
\[
c + \varepsilon
\le
\varphi\left( \tfrac { f(y+t(v+\hat w)) - f(\bar y) }{ t } \right)
<
\varphi\left( \tfrac { f( y+t(v+\hat w)) - f( x) }{ t } \right)
+ \tfrac{ \varepsilon }{ 2 }
<
c + \varepsilon,
\]
which is a contradiction proving that $B(\bar y, t \eta ) \cap \bar P=\emptyset$.

As $\eta\le\omega_5$, we
have that
$P_{c,\tau}\subset\bar P$ due to \eqref{eq:cplusEps}
and~\eqref{eq:HDMonotOmega},
so that 
$\bar P\cap B(\bar y, t\eta)=\emptyset$ gives 
$P_{c,\tau}\cap B(\bar y, t\eta)=\emptyset$
which, as
$\hat v\in B( v, \eta / 2 )$,
implies \eqref{eq:hole}, a contradiction. 
\end{proof}
\begin{lemma}\label{l:MP2.22}
        Suppose that $\Phi\subset \Lip_1(Z)$ is separable.
Suppose
        that $V$ and $W$ are separable subsets of $Y$,
        $\Edomf\subset Y$
        and that $f\fcolon \Edomf \to Z$ is a Lipschitz mapping.
        Then there is a $\sigma$-$V$-directionally porous
        set $P\subset Y$ such that
        for every $\varphi \in \Phi$, $\varepsilon>0$, $p\in \Edomf\setminus P$,
        $v\in V$, $w\in W$,
        $\omega>0$ and $\eta>0$
        one may find $\delta_0 > 0$ such that
        \begin{equation}\label{eq:inMP2.22}
                \hadamderset[\delta,\omega]  f(p, v + w)
                \subset
                \hadamderset[\delta,\eta]  f(p, v )
                +
                B_\varphi(
                \hadamderset[\delta,\omega+\eta]  f(p, w)
                ,
                \varepsilon
                )
        \end{equation}
        for every $\delta \in (0,\delta_0)$.
\end{lemma}
\begin{proof}
        Let $K>0$ be such that $f$ is $K$-Lipschitz.
        Choose countable dense sets $V_1\subset V$, $W_1\subset W$ and $\Phi_1\subset \Phi$.
		For
        $v_1\in V_1$,
        $w_1\in W_1$,
        $\varphi_1\in \Phi_1$
        and
        $\varepsilon_1,\bar\omega, \tau \in \Qplus$,
        let $P_{v_1, w_1, \varphi_1,\varepsilon_1,\bar\omega,\tau}$ be the
        $\sigma$-$v_1$-porous set
        from Lemma~\ref{l:MP2.21}
        used
        for $v_1$, $w_1$, $\varphi_1$, $\varepsilon_1$ in place of $v$, $w$, $\varphi$, $\varepsilon$ and
        with $\omega_3=\bar\omega$, $\omega_4=\omega_5=\tau$. 
        We show that the statement of the present lemma holds with
        \[
        P= \bigcup_{\substack{v_1\in V_1,
                w_1\in W_1,
                \varphi_1\in \Phi_1\\
                \varepsilon_1,\bar\omega,\tau \in \Qplus}
        }
        P_{v_1, w_1, \varphi_1,\varepsilon_1,\bar\omega,\tau}
        \]
        which is clearly a $\sigma$-$V$-porous set.

        Let 
        $\varphi\in \Phi$,
        $\varepsilon>0$,
        $p\in \Edomf\setminus P$,
        $v\in V$, $w\in W$
        and
        $\omega, \eta \in (0,\infty)$.
        Let $\varepsilon_1\in\Q\cap(0,\varepsilon/3)$,
        $\tau \in \Q \cap (0, \eta/5)$
        and
        $\bar \omega \in \Q \cap (\omega + \tau, \omega + 2 \tau)$.
        Find
        $v_1\in V_1$,
        $w_1\in W_1$,
        $\varphi_1\in \Phi_1$
        such that
        $\norm{v - v_1} < \tau/2$,
        $\norm{w - w_1} < \tau/2$
        and
        $\sup_{z\in B(0_Z, (\bar\omega+3\tau+\norm{w})K)} \abs{ \varphi(z) - \varphi_1(z) } < \varepsilon_1$.
        Let $\delta_0 > 0$ be given by Lemma~\ref{l:MP2.21}
        for $v_1$, $w_1$, $\varphi_1$, $\varepsilon_1$ in place of $v$, $w$, $\varphi$, $\varepsilon$ and
        $\omega_3 = \bar \omega $,
        $\omega_4=\omega_5=\tau$.

        Assume $\delta \in (0, \delta_0)$.
        Then using \eqref{eq:coneapprox} and~\eqref{eq:HDMonotOmega} we get
\[                        \hadamderset[\delta,\omega]  f(p, v + w)
  \subset
                        \hadamderset[\delta,\omega+\tau]  f(p, v_1 + w_1)
                 \subset
                        \hadamderset[\delta,\bar \omega]  f(p, v_1 + w_1).\]
        Since $p\in\Edomf\setminus P$ does not belong to the set $P_{v_1, w_1, \varphi_1,\varepsilon_1,\bar\omega,\tau}$, obtained by Lemma~\ref{l:MP2.21},
        we see that 
                        \[\hadamderset[\delta,\bar \omega]  f(p, v_1 + w_1)
                        \subset
                        \hadamderset[\delta,\tau]  f(p, v_1 )
                        +
                        B_{\varphi_1}(
                                \hadamderset[\delta,\bar \omega+2\tau]  f(p, w_1) , \varepsilon_1 ).\] 
        Applying \eqref{eq:coneapprox}, followed by~\eqref{eq:phi:phi} with
        $T=B(0_Z,(\bar\omega+3\tau+\norm{w})K)$ provided by
        Remark~\ref{r:HDlip},
        we get that 
\begin{align*}
                        \hadamderset[\delta, \tau]  f(p, v_1 )
                        +
                        B_{\varphi_1}(
                                \hadamderset[\delta,\bar \omega+2\tau]  f(p, w_1) ,\varepsilon_1 )                
                 &\subset
                        \hadamderset[\delta, 2\tau]  f(p, v )
                        +
                        B_{\varphi_1}(
                                \hadamderset[\delta,\bar \omega+3\tau]  f(p, w) ,  \varepsilon_1)
                \\
                \subset \hadamderset[\delta, 2\tau]  f(p, v )
                        +
                        B_{\varphi}(
                                \hadamderset[\delta,\bar \omega+3\tau]  f(p, w) , 3\varepsilon_1)
                &\subset
                        \hadamderset[\delta,\eta]  f(p, v )
                        +
                        B_{\varphi}(
                                \hadamderset[\delta,\omega+\eta]  f(p, w) , \varepsilon )
        \end{align*}
        where~\eqref{eq:HDMonotOmega} and~\eqref{eq:HDMonotDelta} was used in the last inclusion. 
        This sequence of inclusions proves \eqref{eq:inMP2.22}.
\end{proof}

        The following lemma explains how
        the
        function $\tildegD(x\oplus y)= \fngD(x) + y$,
        where $\fngD\fcolon G\subset X \to Y$,
        carries non-porosity down to its domain.
        We will use this lemma to help us construct the exceptional subset of $X$.
\begin{lemma}\label{l:porpor}
        Let 
        $\fngD\fcolon G\subset X \to Y$ be a function, $\Edomf\subset Y$.
        Let $\tildegD\fcolon G\oplus Y\subset X\oplus Y\to Y$ be defined by $\tildegD(x\oplus y)= \fngD(x) + y$.
        Assume that $x\in G$, $\fngD(x)\in \Edomf$, $v\in X$ and that
        Hadamard derivative
        $\bar v \coloneq \fngD'(x; v)$ exists.
        Then if  $w\in Y$ and $\Edomf$ is not porous at $\fngD(x)$ in the direction of $\bar v + w$,
        $G$ is not porous at $x$ in the direction of $v$
        then $\Edomfpreimage \coloneq \tildegD^{-1}(\Edomf)$ is not porous at $x\oplus 0$ in the direction of $v\oplus w$.
\end{lemma}
\begin{proof}
        If $\Edomfpreimage $ is porous at $x\oplus 0$ in the direction of $v\oplus w$, then there exists
        $c>0$ and a sequence $t_n \downto 0$ such that
        \begin{equation}\label{eq:Fpor111}
                B(x\oplus 0 + t_n ( v\oplus w), c t_n) \cap \Edomfpreimage  = \emptyset
        \end{equation}
        for every $n\in \N$.
        Since $\Edomf$ is not
        porous
        at $\fngD(x)$ in the direction of $\bar v+w$,
        there are $n_1\in \N$
        and a sequence
        $y_n \in Y$ ($n\ge n_1$)
        such that $ \emptyset \neq B(  \fngD(x) + t_{n} ( \bar v + w)   , c t_{n} / 3 ) \cap \Edomf
        \ni y_n$
        for every $n\ge n_1$.
        Since $\bar v = \fngD'(x;v)$, there
        are
        $\delta \in (0, c/3)$ and
        $n_2 \ge n_1$ such that
        $\norm{ \fngD( x + t_{n} \hat v )
                - \fngD(x) - t_{n} \bar v } \le
        c t_{n} / 3$
        for every $n\ge n_2$
        and $\hat v \in B(v,\delta)$.
        Since $G$ is not porous at $x$ in the direction of $v$,
        there are
        $n_3 \ge n_2$ and a sequence
        $\hat v_n\in B(v, \delta)$ ($n\ge n_3$)
        such that $x + t_n \hat v_n \in G$
        for every $n\ge n_3$.
        Let $A_n = ( x + t_{n} \hat v_n) \oplus ( y_{n} - \fngD( x + t_{n} \hat v_n) )$.
        Then
        \begin{align*}
        \norm { A_n - ( (x \oplus 0) + t_{n} ( v \oplus w) ) }
        \hspace{-3cm}
        &
        \hspace{3cm}
        =
        \norm{ t_n ( \hat v_n - v) \oplus 0 }
        +
        \norm {
                0
                \oplus
                (
                  y_{n} - \fngD(x + t_{n} \hat v_n) - t_{n} w
                )
              }
        \\ &
        \le
        t_n \delta
        +
        \norm {
                y_{n} - (\fngD(x) + t_{n} \bar v + t_{n} w)
                }
        +
        \norm {
                \fngD(x) + t_{n} \bar v - \fngD(x + t_{n} \hat v_n)
                }
        \\ &
        < 
        c t_{n} / 3  +
        c t_{n} / 3  +
        c t_{n} / 3
        \end{align*}
        for every $n\ge n_3$.
        Since $\tildegD(A_n) = y_{n} \in \Edomf$ we have, for every $n\ge n_3$,
        \[A_n \in \Edomfpreimage  \cap B(x \oplus 0 + t_{n} (v \oplus w), c t_{n} ),\]
        which contradicts~\eqref{eq:Fpor111}.
\end{proof}
\begin{lemma}\label{l:MP2.23}
		Let $V,G\subset X$ and $W,\Edomf\subset Y$.
        Assume that $\Phi \subset \Lip_1(Z)$, $V$ and $W$ are separable.
                Suppose that
        $\fngF\fcolon G\to Y$ and
        $f\fcolon \Edomf\to Z$ are Lipschitz.
        Then there is a $\sigma$-$V$-directionally porous set $P\subset X$
        such that for every
        $\varphi \in \Phi$,
        $\varepsilon > 0$,
        $x\in \fngF^{-1}(\Edomf) \setminus P$,
        $\omega_1>0$,
        $\omega_2>0$,
        for every direction $v\in V$
        for which
        Hadamard derivative
        $\bar v\coloneq\fngF'(x; v)$ exists
        and
        for every
        $w\in W$
        such that
        $\Edomf$ is not porous at $y=\fngF(x)$ in the direction of
        $(\bar v+w)$,
        there exists $\delta_0>0$
        such that
        \begin{equation}\label{eq:inMP2.23}
                \hadamderset[\delta,\omega_1]  f(y, \bar v + w)
                \subset
                \hadamderset[\delta,\omega_2]  f(y, \bar v )
                +
                B_\varphi(
                \hadamderset[\delta,\omega_2]  f(y, w)
                ,
                \omega_1 \Lip f +
                \varepsilon
                )
        \end{equation}
        for every $\delta \in (0,\delta_0)$.
\end{lemma}
\begin{proof}

Let $K=\Lip f\ge0$.
        Let $\tildegF\fcolon G\oplus Y\to Y$ be defined by $\tildegF(x\oplus y) = \fngF(x) +y$,
        and let $h=f\circ \tildegF$.
        Let us identify
        $X$ with $X\oplus \{0\}$ and $Y$ with $\{0\}\oplus Y$,
        and allow ourselves to write, for $x\in X$, the expression $\tilde g(x)$ meaning $\tilde g(x\oplus 0)=g(x)$.
        Hence
        we
        identify
         $V$ and $W$
         with
        subsets of 
        $X\oplus \{0\}\subset X\oplus Y$
        and $\{0\}\oplus Y\subset X\oplus Y$
         respectively.
        Let
        $\Edomfpreimage =
        \tildegF^{-1}(\Edomf)
        \subset
        G \oplus Y
        \subset X\oplus Y$.
        Using
        Lemma~\ref{l:MP2.22},
        with $h \fcolon \Edomfpreimage  \subset X\oplus Y \to Z$
        in place of $f\fcolon \Edomf \to Z$,
        we find a $\sigma$-$V$-directionally porous set
        $Q\subset X\oplus Y$ such that for every
        $\varphi \in \Phi$, $\varepsilon>0$,
        $ p \in \Edomfpreimage  \setminus Q$,
        $v\oplus0=v\in V$, $0\oplus w=w\in W$,
        $\omega>0$ and $\eta>0$
        one may find
        $\delta_*>0$
        such that
        \begin{equation}\label{eq:inMP2.22usedFORh}
                \hadamderset[\delta,\omega]  h(p , v \oplus w)
                \subset
                \hadamderset[\delta,\eta/2]  h(p , v )
                +
                B_\varphi(
                \hadamderset[\delta,\omega+\eta/2]  h(p , w)
                ,
                \varepsilon / 2
                )
        \end{equation}
        for every $\delta \in (0,\delta_*)$.
        Below
        we will use~\eqref{eq:inMP2.22usedFORh} for $p=x\oplus 0 \in \Edomfpreimage \setminus Q$.

        Let 
        $P_0\subset X$
        be the set of all
        $x\in G$ such that 
        $G$ is porous at $x$ in the direction of $v$
        for some $v\in V\setminus\{0\}
        $.
        Note that
        \(   P\coloneq P_0 \cup \{x\in X\setcolon x\oplus 0 \in Q\} =
        P_0 \cup
        \proj_X(Q\cap (X\oplus \{0\}))
        \)
        is a $\sigma$-$V$-directionally porous subset of $X$.

        Let
        $\varphi\in \Phi$, $\varepsilon>0$, $x\in \fngF^{-1}(\Edomf)\setminus P$, $v\in V$, $w\in W$,
        $\omega_1>0$, $\omega_2>0$.
        Assume
        Hadamard derivative
        $
        \bar v \coloneq
        \fngF'(x; v)
        $
        exists.
        Also assume
        that
        $\Edomf$ is non-porous at $y=\fngF(x)$ in the direction of  
        $(\bar v + w)$.

        Note that,
        since $x,v\in X$
        and $w\in Y$, we have that $\tildegF(x) = \fngF(x) = y$
        and
        Hadamard derivative $\tildegF'(x; a v \oplus b w) = a \fngF'(x;v) +  b w$ exists
        for any $a,b \in \R$.
        In particular, $\tildegF'(x;v)= \bar v$, $\tildegF'(x;w)=w$ and $\tildegF'(x;v\oplus w)=\bar v + w$.
Below
we will apply Lemma~\ref{l:upanddown}\itemref{item:upanddown:first} with $e\in\{v\oplus 0,0\oplus w\}$, and Lemma~\ref{l:upanddown}\itemref{item:upanddown:second} with $e=v\oplus w$.

        First we apply Lemma~\ref{l:upanddown}\itemref{item:upanddown:first}
        with
        $\tildegF\fcolon
        G
        \oplus Y\to Y$
        in place of
        $\fngB$
        and
        with
        $V\oplus W\subset X\oplus Y$
        in place of
        $V\subset X$,
                to find
                $\delta_2>0$, $\eta>0$ such that,
                for
                each
                $e\in \{v\oplus0,\ 0\oplus w\}$,
                    \begin{equation}\label{eq:xA2}
                    \hadamderset[\delta,\eta] (f\circ \tildegF) (x,e)
                    \subset
                    \hadamderset[\delta,\omega_2] f(\fngF(x), \tildegF'(x;e))
                    \end{equation}
                for every $\delta \in (0, \delta_2)$.

Recall
that
        $\Edomf$ is not porous at $\fngF(x)$ in the direction
        of $(\bar v + w)$
        and that $x\notin P_0$ so that $G$ is not porous at $x$
        in the direction of \(
        v
        \).
        Thus, by Lemma~\ref{l:porpor}, we conclude that $\Edomfpreimage$ is not porous at 
        $x=x\oplus 0$ 
        in the direction of $v\oplus w$. We now apply Lemma~\ref{l:upanddown}\itemref{item:upanddown:second} with $e=v\oplus w$
to obtain
        $\delta_3\in (0, \delta_2)$
        such that, for every
        $\delta \in (0,\delta_3)$,
        \begin{align}\label{eq:xB2}
                    \hadamderset[\delta,\omega_1] f(\tildegF(x), \tildegF'(x;v\oplus w))
                    \subset
                    \hadamderset[\delta,\eta/2] (f\circ \tildegF) (x,v\oplus w)
                    +
                    B(0_Z, K \omega_1 + \varepsilon / 6)
                    .
        \end{align}

        Apply~\eqref{eq:inMP2.22usedFORh} with $\omega=\eta/2$ and
        $ p=x=x\oplus 0 \in H\setminus (P\oplus \{0\})\subset H\setminus (Q\cap(X\oplus \{0\}))$,
        thus $p\in H\setminus Q$,
        to
         find $\delta_1\in (0,\delta_3)$ such that
        \begin{equation}\label{eq:inMP2.22honX}
                \hadamderset[\delta,\eta/2]  h(x, v \oplus w)
                \subset
                \hadamderset[\delta,\eta/2]  h(x, v )
                +
                B_\varphi(
                \hadamderset[\delta,\eta]  h(x, w)
                ,
                \varepsilon / 2
                )
        \end{equation}
        holds
        for every $\delta \in (0,\delta_1)$.

        Using 
        \eqref{eq:xB2} and
        \eqref{eq:inMP2.22honX}
        in the first and second inclusions,
        \eqref{eq:xA2} in the two subsequent inclusions,
       and~\eqref{eq:phi:addout} for the final inclusion,
        we get,
        for every $\delta \in (0, \delta_1)$,
        \begin{align*}
                \hadamderset[\delta,\omega_1] &f(\fngF(x), \fngF'(x;v) + w)
                \\
                & =
                \hadamderset[\delta,\omega_1] f(\tildegF(x), \tildegF'(x;v\oplus w))
                \\
&\subset
                    \hadamderset[\delta,\eta/2] (f\circ \tildegF) (x,v\oplus w)
                    +
                    B(0_Z, K \omega_1  + \varepsilon /6)
                \\
&\subset
                        \hadamderset[\delta,\eta]  h(x, v )
                        +
                        B_\varphi(
                        \hadamderset[\delta,\eta]  h(x, w)
                        ,
                        \varepsilon / 2
                        )
                    +
                        B(0_Z, K \omega_1 + \varepsilon /6)
                \\
&\subset
                    \hadamderset[\delta,\omega_2] f(\fngF(x), \bar v)
                        +
                        B_\varphi(
                        \hadamderset[\delta,\eta]  h(x, w)
                        ,
                        \varepsilon / 2
                        )
                    +
                        B(0_Z, K \omega_1 + \varepsilon /6)
                \\
&\subset 
                    \hadamderset[\delta,\omega_2] f(\fngF(x), \bar v)
                        +
                        B_\varphi(
                        \hadamderset[\delta,\omega_2] f(\fngF(x), w)
                        ,
                        \varepsilon / 2
                        )
                    +
                        B(0_Z, K \omega_1 + \varepsilon /6)
                \\
                &\subset
                    \hadamderset[\delta,\omega_2] f(\fngF(x), \bar v)
                        +
                        B_\varphi\Bigl(
                        \hadamderset[\delta,\omega_2] f(\fngF(x), w)
                        ,
                        K \omega_1 +\varepsilon / 2 + \varepsilon /6\Bigr)
                .
                \qedhere
        \end{align*}
\end{proof}
\medbreak

       We are about to establish the completeness of Hadamard derivative domain assignment
       $\UPhi f$ for functions defined on a subset of a Banach space, cf.\ Definition~\ref{def:completeAssignment}.

\begin{lemma}\label{l:MP52}
        Let\/ $Y$, $
        Z\withrnpindex \subset
        Z$ be
        Banach spaces.
        Assume
        that
        $Z\withrnpindex$
        has the Radon-Nikod\'ym property
        and that\/
        $Y$
        and $\Phi\subset\Lip_1(Z)$ are separable.
        Let $\Edomf$ be a subset of\/ $Y$
        and
        $f\fcolon \Edomf \to Z$ a Lipschitz function.
        Assume that $\Edomf\withrnpindex$ is a subset of $\Edomf$ and $f(\Edomf\withrnpindex) \subset Z\withrnpindex$.

        Let $R\fcolon \Edomf\to 2^Y$ be a complete assignment such that
        $\Edomf$
        is not porous at $y$ in the direction of $u$ for
        every $y\in\Edomf$
        and $u\in R(y)$, and that $R(y)$  is a
        linear subspace of\/ $Y$ for every $y\in\Edomf\withrnpindex$.

        Then
        $y \in \Edomf\withrnpindex \mapsto \UPhiplusR f(y)$
        is a complete assignment. 
\end{lemma}
\begin{proof}
        Let $X$ be a separable Banach space,
        $G\subset X$ a set
        and $\testfunY\fcolon G\to Y$ a Lipschitz function.
        By Lemma~\ref{l:step12}, there is an $\mathcal L_1$ null set $Q\subset X$
        such that,
        for every $x\in \testfunY^{-1}(\Edomf\withrnpindex)\setminus Q$ 
        and every $e\in X$ for which
        Hadamard derivative
        $\testfunY'(x; e)$ exists,
        it holds that 
        Hadamard derivative
        $
         f'(\testfunY(x) ; \testfunY'(x;e))
        $
        also exists.

        Since $R(y)$ is a complete
        assignment, 
        there is an $\mathcal L_1$ null set $P_0$ such that
        for every $x\in
        \testfunY^{-1}(\Edomf\withrnpindex)
        \setminus P_0$ and $v\in X$,
        if
        Hadamard derivative
        $\bar v = \testfunY'(x; v)$ exists then $\bar v \in R(\testfunY(x))$.

        For each $\sigma\in\{-1,+1\}$,
        by Lemma~\ref{l:MP2.23} there exists
        a $\sigma$-directionally porous set $P_\sigma\subset X$
        such that for every
        $\varphi \in \Phi$,
        $\varepsilon > 0$,
        $x\in \testfunY^{-1}(\Edomf) \setminus P_\sigma$,
        $\omega_1>0$,
        $\omega_2>0$,
        for every direction $v\in V\coloneq Y$
        for which
        Hadamard derivative
        $\bar v\coloneq\testfunY'(x; v)$
        exists
        and for every
        $w\in Y$
        such that $\Edomf$ is not porous at $y\coloneq\testfunY(x)$ in the direction of
        $\sigma \bar v +  w$,
        there exists $\delta_\sigma=\delta_\sigma(\varphi, x, v, w, \varepsilon, \omega_1, \omega_2)>0$
        such that
        \begin{equation}\label{eq:inMP2.23-copy}
                \hadamderset[\delta,\omega_1]  f(y,\sigma \bar v + w)
                \subset
                \hadamderset[\delta,\omega_2]  f(y,\sigma \bar v )
                +
                B_\varphi(
                \hadamderset[\delta,\omega_2]  f(y, w)
                ,
                \omega_1 \Lip f +
                \varepsilon
                )
        \end{equation}
        for every $\delta \in (0,\delta_\sigma)$.

        Assume
        $x\in \testfunY^{-1}(\Edomf\withrnpindex) \setminus (Q\cup P_{-1}\cup P_{+1}\cup P_0)$
        and $v\in X$
        are such that $\bar v \coloneq \testfunY'(x;v)$ exists.
        Let $y=\testfunY(x)$.
        Since $x\notin P_0$, we have $\bar v \in R(y)$.
        Since $x\notin Q$, Hadamard derivative $f'(y;\bar v)$ exists.
        Fix
        $\varepsilon > 0$,
        $\varphi \in \Phi$,
        and 
        also
        $w\in R(y)$.
        Then, as $R(y)$ is a linear subspace and $\bar v,w\in R(y)$, we get
        $\sigma\bar v,
        \sigma\bar v + w\in R(y)$, and hence
        $\Edomf$
        is not porous at $y$ in the directions of
        $\sigma \bar v$
        and
        $\sigma\bar v + w$, for both $\sigma\in\{-1,+1\}$.
        Let $\omega_1>0$ be such that $\omega_1\Lip f<\varepsilon$.
        Numbers
        $\delta_{0}>0$
        and $\omega_2\in (0, \varepsilon)$
        can be chosen small enough to guarantee
        \[
        \hadamderset[\delta,\omega] f(y, \sigma \bar v)
        \subset
        f'(y;\sigma \bar v)
        +
        B(
        0_Z,
        \varepsilon)
        \]
        whenever $\omega \in (0,\omega_2]$, $\delta\in (0, \delta_{0})$ and $\sigma\in\{-1,+1\}$,
        see Remark~\ref{r:HDsetandder}.
        Let $\Delta_0 = \min(\omega_1, \omega_2, \delta_{0}, \delta_{-1},\delta_{+1})$.
        Then \eqref{eq:inMP2.23-copy}, 
        together with properties~\eqref{eq:HDMonotOmega} and~\eqref{eq:phi:addout}, 
        implies that
        \begin{align}
        \notag
                \hadamderset[\delta,\omega]  f(y, \sigma\bar v + w)
                \subset
                \hadamderset[\delta,\omega_1]  f(y, \sigma\bar v + w)
                &\subset
                f'(y; \sigma\bar v)
                +
                B_\varphi(
                \hadamderset[\delta,\omega_2]  f(y, w)
                ,
                3 \varepsilon
                )
        \\
        \notag
                &\subset
                f'(y;\sigma \bar v)
                +
                B_\varphi(
                \hadamderset[\delta,3 \varepsilon]  f(y, w)
                ,
                3 \varepsilon
                )
        \end{align}
        for every $\delta, \omega \in (0, \Delta_0)$ and $\sigma\in\{-1,+1\}$.
We conclude that
        indeed $\testfunY'(x;v) \in \UPhiplusR(f,\testfunY(x))$,
        cf.~Definition~\ref{def:assignmentPhi}
        and Remark~\ref{r:thickandporous},
        thus
        $y \in \Edomf\withrnpindex \mapsto \UPhiplusR f(y)$
        is complete by Definition~\ref{def:completeAssignment}.
\end{proof}
\begin{theorem}\label{t:MP52}
        Let $Y$, $Z$ be
        Banach spaces.
        Assume
        that
        $Z$ has the Radon-Nikod\'ym property
        and that
        $Y$
        and $\Phi\subset\Lip_1(Z)$ are separable.
        Let $H\subset Y$
        and
        $h\fcolon H \to Z$ be a Lipschitz function.

Then $\EAPhi h$, $\UPhi h$ 
and
$\regtanex H$ are complete assignments.
\end{theorem}

\begin{proof}
Recall that 
$\Phir=\Phir(\R)\subset\Lip_1(\R)$,
is separable and admissible by Remark~\ref{r:ZsepPHIsep}\itemref{r:ZsepPHIsep:item:Ur}, and that $[Y]$,
defined by $[Y](y)=Y$ for every $y\in Y$,
 is a complete assignment by Remark~\ref{rem:Ycomplete}.
       Using Lemma~\ref{l:MP52} with 
$\Edomf\withrnpindex=\Edomf=Y$,
$Z=\R$, 
         $R=[Y]$, 
        with $\Phir$, 
        and with $f(y)=d_H(y)\coloneq\dist(y,H)$ defined for all $y\in Y$, we
        obtain that
        $\UPhirplusY d_H$
        is a complete assignment. By Remark~\ref{r:restriction:complete}, its
        restrictions are
        complete;
        hence
        $\regtanex H$,
        its restriction to $H$,
        is
        also
        complete,
        cf.~\eqref{eq:regtanex:def}.

Using Lemma~\ref{l:MP52} with $(F\withrnpindex,Z\withrnpindex,R)\coloneq(F,Z,\regtanex)$, we conclude that $\UPhi h$ is complete.

Let $\tilde Z$ and $\tilde f\fcolon Y\to\tilde Z$ be as in~\eqref{eq:extensionBL_codomain} and~\eqref{eq:extensionBL} of Lemma~\ref{l:extf}, used with $(f,\Edomf)\coloneq (h,H)$.
Let $Z_0 = \closedSpan h(H)\subset\tilde Z$,
$\Phi_0 = \Phi|_{Z_0}$ be the collection of restrictions
and let $\tildePhi\subset\Lip_{2}(\tilde Z)$
be as in~\eqref{eq:extensionPhi}
of Lemma~\ref{l:extphi}.
        We may now apply Lemma~\ref{l:MP52} with $(F\withrnpindex,F,f,R)\coloneq(H,Y,\tilde f,[Y])$ to get that 
        $y\in H \mapsto  \Hindexedby{\frac12 \tildePhi, [Y] } \tilde f (y)$
        is complete. As $\regtanex H$ is also complete, by application of Remark~\ref{r:cap_complete}, Definition~\ref{def:assignmentEA} and Remark~\ref{r:onehalf} we get that $\EAPhi h$ is complete.
\end{proof}

In the next section we show that one cannot remove the assumption that $Y$ is separable from Theorem~\ref{t:MP52}, cf.\ Example~\ref{e:nonsepar2}.

\section{Borel measurability of $\EAPhi f$}\label{s:Borel}

        This section is devoted to the proof, in Proposition~\ref{p:Borel1}, that $\EAPhi f$ is Borel measurable
        whenever $Y,Z$ are Banach spaces, $Y$ is separable,
        $f\fcolon\Edomf\subset Y\to Z$ is Lipschitz and $\Phi\subset\Lip_1(Z)$
        is separable, see Definition~\ref{def:phi}.

        We start by showing that one cannot omit the assumption that $Y$ is separable.

\begin{example}\label{e:nonsepar2}
        Let $\spaceExample $ be a non-separable Banach space.
        Under the assumption $\cantor=\omega_1$ 
        we show that there
        is a closed set $H\subset Y$ such that
        $\regtanex H$
is neither complete
nor Borel measurable. We get it as a particular case of a more general construction, of the closed set $H$ and function $f(\cdot)=\dist(\cdot,H)$ for which 
        $\UPhi f=\EAPhi f$        are
        neither complete
        nor Borel measurable,
       whenever 
$\Phi \subset \Lip_1(\R)$
contains at least one of the following functions:
$\varphi_1(z) = \abs{z - 1}$,
$\varphi_2(z) = -z$ or    
$\varphi_3(z) = -\abs{z}$, $z\in\R$. 
By Remark~\ref{r:disttan:eq},
to justify the statement about $\regtanex H$ we only need to recall that $\Phi=\Phir$ contains  $\varphi_1$, see Remark~\ref{r:ZsepPHIsep}\itemref{r:ZsepPHIsep:item:Ur}.

In fact, we show that for every $\setTexample\subset \R$ of cardinality at most $\omegaone$
and any unit vector $\vectorv_0\in Y$ there is a closed set $H\subset\spaceExample$,
such that 
\begin{equation}\label{eq:setTexample}
\R e_0 \subset H
\qquad\text{and}\qquad 
	\{ x \in \R \setcolon \vectorv_0
\notin
\UPhiplus f(x \vectorv_0) \}
=
\setTexample
\end{equation} 
 for any $\Phi$ containing one of the three functions as above.
        Here we have $\UPhi f=\EAPhi f$ by Remark~\ref{r:HAisH}.

Assuming this claim, we presently define a very specific $\setTexample$,
of cardinality at most $\omegaone$, for which $\UPhi f$ is neither Borel measurable nor complete.
Indeed, one of the properties of $\setTexample\subset\R$ will be that it is not Borel. Using the notation for $S$ from Definition~\ref{def:assignment-meas},  note that~\eqref{eq:setTexample} implies identity 
$S\cap (\R e_0\times \{e_0\})=(\R\setminus \setTexample)e_0
\times \{e_0\}$,
which gives that $S$ is not Borel in $Y\times Y$, meaning that $\UPhi f$ is not Borel measurable.
Furthermore, another property of $\setTexample$ will be that it has a
       positive outer Lebesgue measure.
To
show that $\UPhi f$ is not complete, we  refer to Definition~\ref{def:completeAssignment} 
with
test mapping  $g\fcolon \R\to Y$,
$g(t) = t \vectorv_0
\in H$, which fails to satisfy~\eqref{eq:propofN} due to~\eqref{eq:setTexample}.

We quickly explain how $\setTexample$ could be constructed.
Let
        $\{B_\alpha\}_{\alpha<\omegaone}$ be an enumeration
        of all Borel subsets
        of~$[0,1]$
        and
        $\{z_\alpha\}_{\alpha<\cantor}$ be an enumeration
        of~$[0,1]$.
Let 
$
\setTexample_1 =
\setTexample_{\text{diag}}
= \{ z_\alpha \setcolon \alpha < \omegaone \text{ and } z_\alpha \notin B_\alpha \}$.
Assuming $\omega_1 = \cantor$, choose        
        a set
        $\setTexample_2\subset [2,3]$
        of positive outer Lebesgue measure
        of cardinality at most $\omegaone$.
Then $\setTexample_1 \cup \setTexample_2$                
is a non-Borel 
subset 
of $\R$ 
of cardinality at most $\omegaone$ and of positive outer Lebesgue measure.

Let $\setTexample\subset\R$ be any subset of cardinality at most $\omega_1$, $\spaceExample$ non-separable Banach space, $e_0\in S_\spaceExample$ and $\Phi\cap\{\varphi_1,\varphi_2,\varphi_3\}\ne\emptyset$.
We now aim to define a closed set $H\subset \spaceExample$ such that~\eqref{eq:setTexample} is satisfied.  

We first define
        $\{\vectorv _\alpha\}_{1\le \alpha< \omega_1} \subset S_\spaceExample =\{y\in\spaceExample\setcolon \norm{y}=1\}$ 
        such that,
        for every 
        $1\le \gamma < \omega_1$,
\begin{equation}\label{eq:faraway}
        \norm{ \vectorv  - \vectorv _\gamma } > 99/100
        \qquad
        \text{ whenever }
        \vectorv \in \closedSpan\  \{ \vectorv _\alpha \setcolon 0\le \alpha < \gamma \}
        .
\end{equation}
Note that in~\eqref{eq:faraway} we include $e_0$ which has been fixed in advance.
        If 
        $1\le \gamma < \omega_1$ and
        $\{\vectorv _\alpha\}_{0\le \alpha< \gamma}$ are
        already defined, 
        let $\spanExample_\gamma =\closedSpan\ \{ \vectorv _\alpha \setcolon 0\le \alpha < \gamma \}$.
        As $\spanExample_\gamma$ is separable and $\spaceExample $
        is
        non-separable, there exists $z\in \spaceExample \setminus \spanExample_\gamma $.
        Find $g\in \spaceExample ^*$ separating $z$ and $\spanExample_\gamma $. As $\spanExample_\gamma $ is a linear space, we have $g=0$ on $\spanExample_\gamma$
        and $g(z)\neq 0$.
        Let $\vectorv _\gamma \in S_\spaceExample $ be such that $g(\vectorv _\gamma) > 99/100 \norm{g} $.
        As $g = 0 $ on $\spanExample_\gamma$ we have~\eqref{eq:faraway}.

        Now we prove that
        for every 
pair of
distinct $\alpha,\beta \in[1,\omega_1)$, for all
$a, b\in \R \setminus \{ 0 \}$
and $v\in \Spanofset {\vectorv_0}$
\begin{align}\label{eq:faraway2}
        \norm{a \vectorv _\alpha - b \vectorv _\beta 
        +
        v }
        >
        ( \abs{a} + \abs{b} ) \cdot 3/10.
\end{align}

        By \eqref{eq:faraway}, assuming 
        $a,b\ne0$ and,
        without loss of generality, $\alpha<\beta$,
		we get
                \(
                \norm{\frac ab \vectorv _\alpha - \vectorv _\beta 
                +
                	 \frac 1b v }
                \ge
                A \coloneq
                99/100
                \).
                On the other hand
                we also have 
                $\norm{ \vectorv _\alpha 
                +
                	  \frac 1a v } > 99/100$
                and so
                \(
                \norm{a \vectorv _\alpha - b \vectorv _\beta 
                +
                	  v  }
                 \ge 
                \norm{a \vectorv _\alpha 
                +
                	 v } - \norm { b \vectorv _\beta }
                >
                B
                \)
                where
                \(
                 B \coloneq
                \abs a \cdot 99/100 - \abs b
                \).
                Using a convex combination with coefficients $2/3+1/3=1$ we
                obtain
                \(
                \norm{a \vectorv _\alpha - b \vectorv _\beta   
                +
                	v  }
                 >  2/3 \cdot A\abs b  + 1/3 \cdot B 
				\ge
				32/100 \cdot \abs b + 33/100 \cdot \abs a
                  \),
                  from which~\eqref{eq:faraway2} follows.

        Similarly, we show that
        for every $1 \le \alpha<\omega_1$, for all 
        $a\in \R \setminus\{0\} $
        and $v\in \Spanofset {\vectorv_0}$
\begin{equation}\label{eq:spanaway2}
        \norm{a \vectorv _\alpha 
        +
        	 v } > ( \abs{a} + \norm v ) \cdot 3/10.
\end{equation}
        Indeed,
        by \eqref{eq:faraway},
        $\norm{ \vectorv _\alpha 
        +
        	 \frac 1a v } > 9/10$.
        On the other hand,
        \(
        \norm{a \vectorv _\alpha  
        +
        	  v  }
        \ge
        \norm v - \abs a
        \).
        Adding the $(7\abs a/10)$-multiple of the first inequality and $3/10$-multiple of the second, 
        we obtain~\eqref{eq:spanaway2}.

\medbreak

        Recall that
        $\setTexample\subset \R$
        has cardinality at most $\omegaone$.
        Let $\{ x_\alpha \}_{1\le\alpha < \Gamma}$ be an enumeration of $\setTexample$,
        where $\Gamma \le \omegaone$.
        For
        $1 \le \alpha<\Gamma$,
        let 
        \begin{equation*}
        	C_\alpha = \bigcup _{s\in \{\pm 1\}} \, \bigcup _{t\in(0,\infty)} B(st \vectorv_\alpha, t/10) \subset \spaceExample.
        \end{equation*}
        In this construction, we always index the cones by
        ordinals
        from \( 
        [1,\Gamma) =
        \{ \alpha \setcolon 1 \le \alpha < \Gamma \} \).

        If $y_1 \in C_\alpha$,  $y_2 \in C_\beta$,
        $ w \in  \Spanofset {\vectorv_0}$,
        and
        $\alpha,\beta\in[1,\Gamma)$,
        $\alpha \neq \beta$
        then
        $\norm {y_1} > 0$,
        and we show that
\begin{align}
        \label{eq:dist:new}
            \norm{y_1 - y_2
        +
             w
             } &\ge (\norm {y_1} + \norm {y_2}) \cdot 2 / 11
        \qquad\text{and}
        \\
        \label{eq:distYw:new}
            \norm{y_1 
        +
             w } 
            &\ge ( \norm {y_1} + \norm w) \cdot 2 / 11
            .
\end{align}
        Indeed, we have $y_1 \in B(r t \vectorv_\alpha, t/10)$
        for some
        $t>0$,
        $r\in \{\pm 1\}$
        and
        $y_2 \in B(s \tau \vectorv_\beta, \tau / 10)$
        for some
        $\tau > 0$,
        $s\in \{\pm 1\}$.
        By~\eqref{eq:faraway2},
        and~\eqref{eq:spanaway2},
        $\norm{r t \vectorv_\alpha  - s \tau \vectorv_\beta 
        +
        	 w } > (t + \tau)\cdot 3/10$
        and
        $\norm{r t \vectorv_\alpha  
        +
        	 w } >  (t + \norm w) \cdot 3/10$,
        hence
        $\norm{ y_1 - y_2 
        +
        	  w } > (t + \tau)\cdot 2/10 > (\norm {y_1} + \norm {y_2} ) \cdot 2/11 $
        and
        $\norm{ y_1 
        +
        	 w } >  (t + \norm w) \cdot 2/10 > (\norm {y_1} + \norm w) \cdot 2/11 $.

        Inequalities~\eqref{eq:dist:new}
        and
        \eqref{eq:distYw:new}
        in particular imply
\begin{align}\label{eq:cones:disj:new}
        ( C_\alpha + \Spanofset  {\vectorv_0} )
        \cap
        ( C_\beta + \Spanofset  {\vectorv_0} )
        &=\emptyset
        \qquad
        &\text{for } 
        \alpha \neq \beta,
		\quad\alpha,\beta\in[1,\Gamma)
\\\noalign{and}
\label{eq:span:cone:disj}
        ( C_\alpha + \Spanofset  {\vectorv_0} )
        \cap
         \Spanofset  {\vectorv_0}
        &=\emptyset
        \quad
&\text{for } 
\alpha\in[1,\Gamma)       .
\end{align}%

\smallbreak

        Let
        $
          G
         =
          \bigcup _ { 1 \le \alpha < \Gamma }
          (
          x_\alpha e_0 + C_{\alpha}
          )
          \subset Y
        $
        and
        $H =
          Y
        \setminus
          G
        $,
        where
        we recall that
        $\{ x_\alpha \}_{1\le\alpha < \Gamma}$ is an enumeration of $\setTexample$,        
        where
        $\Gamma \le \omegaone$.
        $H$ is a closed subset of $Y$ as $G$
        is
        open. Also, by~\eqref{eq:span:cone:disj}, we have $H\supset\{xe_\firstindex\setcolon x\in\R\}$, so the first part of~\eqref{eq:setTexample} is satisfied.

        Let
   \begin{align}\label{eq:new:def:f:simple}
        f(y)
        =
                        \dist (y, H)           \qquad      \text{ for }y \in Y.
   \end{align}

        We now study the differentiability properties of~$f$ in various directions
        at all points $x e_{\firstindex} \in Y$, $x\in \R$.
        For all $x\in\R$, we have that $f(x e_{\firstindex})=0$
        by~\eqref{eq:span:cone:disj},
        hence
        for every $x\in \R$,
        \begin{align}\label{eq:newexa:der}
        f'(x e_{\firstindex}; e_{\firstindex})
        =
        0
        .
        \end{align}

        Let $1 \le \alpha<\Gamma$.
        We claim that
        for every $\eta\in (0, 1/11)$ and every $\delta>0$,        
        \begin{equation}\label{eq:newexa1}
                \hadamdersetboth[\delta,\eta] f(
                x_\alpha
                e_{\firstindex}, 12 e_{\firstindex} + e_\alpha) = \{ 0 \}
.
        \end{equation}
Indeed,
        by~\eqref{eq:dist:new},
        used for $ y_1 = t \vectorv_\alpha\in C_\alpha$ and for every
        $y_2\in C_{\beta}$,
        whenever
                $\beta \in [1, \Gamma) \setminus \{ \alpha
        \}$,
        we have
\begin{equation}\label{eq:otherbeta}
        \dist(
        x_\alpha \vectorv_\firstindex +
        t(12 \vectorv_\firstindex
           + \vectorv_\alpha
           ),
        x_\beta \vectorv_\firstindex +
        C_{\beta})
        \ge 2t/11
\end{equation}
        for any
        $t>0$
        so that
        for every \(
        \hat u \in B(0_Y,\eta)
        \subset B(0_Y, 1/11)
        \),
\begin{equation}\label{eq:otherbeta2}
        \dist(
        x_\alpha \vectorv_\firstindex +
        t( 12 e_\firstindex+e_\alpha+\hat u),
            x_\beta \vectorv_\firstindex +
        C_{\beta})
        >
        0
        .
\end{equation}
        In the case $\beta=\alpha\in[1,\Gamma)$
        we obtain that~\eqref{eq:otherbeta2} holds true
        by considering that
\begin{multline}\label{eq:samebeta}
        \dist(
        x_\alpha \vectorv_\firstindex +
        t(12 \vectorv_\firstindex
           + \vectorv_\alpha
           + \hat u
           ),
        x_\alpha \vectorv_\firstindex +
        C_{\alpha})
        \\
        =
        \dist(
        t(12 \vectorv_\firstindex
           + \vectorv_\alpha
           + \hat u
           ),
        C_{\alpha})
        >
        \dist(
        12 t\vectorv_\firstindex
           ,
        C_{\alpha})
        - 2t
\end{multline}
        where, by \eqref{eq:distYw:new},
        the last expression is greater than $2t -2t \ge0$.
        Note that the validity of \eqref{eq:otherbeta2} for every
        $\beta \in [1,\Gamma)$ implies
        $f(x\vectorv_0+t(12\vectorv_\firstindex+e_\alpha+\hat u))=f(x\vectorv_0)=0$,
        and, recalling Definition~\ref{def:Hadamard-derived-set}, this verifies~\eqref{eq:newexa1}.

\medbreak

        We now show that $e_{\firstindex} \in \UPhi f(x e_{\firstindex})$ if and only if $x\not\in \setTexample$, in other words,
the second part of~\eqref{eq:setTexample} is satisfied.
        Assume first that
        $x \in \setTexample$.
        Let
       $x=x_\alpha\in \setTexample$, where $1\le \alpha < \Gamma$.
        We observe that
        \begin{equation}\label{eq:newexa2}
                 \emptyset
                \neq
                 \hadamdersetboth[\delta,\varepsilon] f(x_\alpha e_{\firstindex}, e_\alpha)
                \subset
                 [1/\ten -\varepsilon, 1+\varepsilon]
        \end{equation}
        for every $\varepsilon>0$ and $\delta>0$.
        Indeed,
        $f(x_\alpha e_{\firstindex}) = 0$ as $H\supset\R e_\firstindex$ and,
        for every $t > 0$,
        \[
        B(x_\alpha e_{\firstindex} + t e_{\alpha}, t/\ten) \subset x_\alpha e_{\firstindex} + {C_{\alpha}} \subset Y \setminus H,\]
        hence
        \(
            f(x_\alpha e_{\firstindex} + t e_\alpha)
            \ge
            t/\ten
        \).
        Of course,
        \(
            f(x_\alpha e_{\firstindex} + t e_\alpha)
            \le
            t
        \),
                      since $\Lip f \le 1$.
        Therefore also
        \(
            t/\ten - t\varepsilon
          \le
            f(z)
          \le
            t + t\varepsilon
        \)
        whenever $z=x_\alpha e_{\firstindex} + t\hat{u}\in B(x_\alpha e_{\firstindex} + t e_\alpha, t\varepsilon)$,
        since $\Lip f \le 1$, 
        proving~\eqref{eq:newexa2}. 

        Combining \eqref{eq:newexa:der},
        \eqref{eq:newexa1}
        and
        \eqref{eq:newexa2},
        we see that if 
		$x=x_\alpha\in \setTexample$        
        then
        the defining property 
        \eqref{eq:UPhiRegularity} 
        of $\UPhi f(x_\alpha e_\firstindex)$
        fails
        for $y=x_\alpha e_{\firstindex} \in Y$, $u=12 e_{\firstindex}$,
        $\sigma=+1$,
        $v=e_\alpha$,
        $\varepsilon = 1/30$ ,
        \(
        \varphi\in\{\varphi_1,\varphi_2,\varphi_3\}
        \)
        and for every $\delta>0$ and every
        $\eta=\omega \in (0, 1)$
        as
      \begin{equation}\label{eq:newexa:main_inc}
  f'(x_\alpha e_\firstindex; 12 e_\firstindex)        
        +
  B_\varphi(
  \hadamdersetboth[\delta,\varepsilon]
  f(x_\alpha e_\firstindex,e_\alpha),\varepsilon)
\subset 
  0+B_\varphi([2/30,31/30],1/30)
\subset
[1/30,+\infty),
       \end{equation}
       and the latter set does not contain the set~\eqref{eq:newexa1},
       \{0\}.

        This shows that $ 12 e_{\firstindex} \notin \UPhi f(x_\alpha e_{\firstindex})$, hence
        $e_{\firstindex} \notin \UPhi f(x_\alpha e_{\firstindex})$,
        cf.\ Lemma~\ref{l:linearity}\itemref{two:lin}.

\medbreak

        Now we assume $x\in \R\setminus \setTexample$. 
        We will show that directional 
		(not Hadamard)
        one-sided derivative
        $f_+'(x  \vectorv_\firstindex; v)$ is zero for every $v\in Y$.
        From this it follows that both sided derivatives are zero
        and, because $f$ is Lipschitz, Hadamard directional derivatives are zero too,
        cf.\ Remark~\ref{r:HadSameGat}.
        By Lemma~\ref{l:easyHA}\itemref{l:easyHA:item:one} with $R=[Y]$, this implies
        $\UPhi f(x \vectorv_\firstindex) = Y$, in particular
        $ \vectorv_\firstindex \in Y = \UPhi f(x \vectorv_\firstindex)$.

        Fix $v\in Y$.
        To prove
     that directional one-sided derivative
        $f_+'(x  \vectorv_\firstindex; v) = 0 $,
        it is enough to show
        that there is $\ddelta>0$ such that
\begin{equation}
        \label{eq:want}
        x \vectorv_\firstindex + t v  \in H
        \qquad
        \text{ for every $t\in(0,\ddelta)$}.
\end{equation}
       As \eqref{eq:want} is trivially true when $v=0$, we will further assume that $v\neq 0$.

        If $x \vectorv_\firstindex + t v \in H$ for every 
        $t>0$,
         we let $\ddelta=1$.
        Alternatively, assume that
        there are $\alpha\in[1,\Gamma)$ and
        $t_0>0$
        such that
        $x \vectorv_\firstindex + t_0 v \in x_\alpha \vectorv_\firstindex + C_\alpha$.
        That means that $t_0 v$ lies in the cone
        $ (C_\alpha + \Spanofset {\vectorv_\firstindex}  )$
        and hence the same is true for $t v$ for every $t>0$.
        By~\eqref{eq:cones:disj:new}, this implies that
\begin{equation*}
        t v \notin  C_{\beta} + \Spanofset {\vectorv_\firstindex}
        \qquad
        \text{ for every $t>0$ and every $\beta \neq \alpha$}
\end{equation*}
        which in turn implies
\begin{equation}
        \label{eq:want:others}
        x \vectorv_\firstindex + t v
        \notin
        x_\beta \vectorv_\firstindex + C_{\beta}
\end{equation}
        for every $t>0$ and every $\beta \neq \alpha$.

        We now show that~\eqref{eq:want:others} is in fact satisfied for $\beta=\alpha$ too, but only for $t\in(0,\ddelta)$, where
        $\ddelta = \abs {x_\alpha - x} / ( 11 \norm v ) $;
      note $\ddelta>0$ since $x\not\in \setTexample$ and $x_\alpha\in \setTexample$.

        By ~\eqref{eq:distYw:new}
        used
        for
        $w=(x_\alpha-x) \vectorv_\firstindex$
        and for all $y_1 \in C_\alpha$,
        we see that for all $t \in (0, \ddelta)$,
        we have
        \(
        \dist (x \vectorv_\firstindex + t v ,
                x_\alpha \vectorv_\firstindex + C_{\alpha}
                )
        \ge
        \norm{ (x_\alpha-x) 
        	\vectorv_\firstindex }
        \cdot 2 / 11
        - t \norm{v}
        >0
        \),
        and therefore
\begin{equation}
        \label{eq:want:mine}
        x \vectorv_\firstindex + t v
        \notin  x_\alpha \vectorv_\firstindex + C_{\alpha}
        \qquad
        \text{ for every $t \in(0, \ddelta)$}
        .
\end{equation}
        Now,~\eqref{eq:want} follows by \eqref{eq:want:others} and~\eqref{eq:want:mine}.
        Hence we have $e_{\firstindex}\in \UPhi f( x e_{\firstindex} )$ for all $x\in \R\setminus \setTexample$.
        This finishes the proof of
the second part of~\eqref{eq:setTexample}.
\end{example}

        \medbreak

Now our aim will be to verify the Borel measurability of $\EAPhi f$ in the case when $Y$ is a separable space. Before we do that, we will need to relate the behaviour of Lipschitz functions and their extensions. 
From now on we assume that $Y,Z$ are Banach spaces. The assumption of separability of $Y$ will be used first in Lemma~\ref{l:BorelFull}.

\begin{lemma}\label{l:HDexten} 
        Let $\Edomf\subset Y$.
        Assume
        $f\fcolon \Edomf \to Z$ 
        and its extension
        $\extendedf \fcolon Y \to Z$ are $K$-Lipschitz functions with $K>0$, $y\in \Edomf$ and $u\in Y$.
        Then
       \begin{equation}\label{eq:HDexten0}
            \hadamderset[\delta,\omega] f(y, u )
           \subset
            \hadamderset[\delta,\omega] \extendedf (y, u )
.
       \end{equation}
 If
        $\Edomf$ is not porous at $y$ in the direction of $u$
        then
        for every $\varepsilon>0$
        there is $\delta_0>0$ such that for every $\delta\in(0,\delta_0)$
       \begin{equation}\label{eq:HDexten}
            \hadamderset[\delta,\varepsilon] \extendedf (y, u )
           \subset
            \hadamderset[\delta,\varepsilon] f(y, u )
            +
            B(0_Z,2 \varepsilon K  )
          .
       \end{equation} 
\end{lemma}
\begin{proof}
        The first 
        inclusion~\eqref{eq:HDexten0} 
        follows from Definition~\ref{def:Hadamard-derived-set}.

        Assume $y\in\Edomf$ and there exists $\varepsilon>0$ such that for
        every $\delta_0>0$ there is a $\delta\in(0,\delta_0)$ such that
        \eqref{eq:HDexten} is not satisfied.          
        This means that there are sequences 
        $\delta_n \searrow 0$,
        $t_n\in(0,\delta_n)$, 
        and $u_n\in B(u, \varepsilon)$
        with
        \begin{equation}\label{eq:eq71}
             \dist
             \left(
                \frac{ \extendedf (y+t_n u_n) - \extendedf (y) }{ t_n }
                ,
                \hadamderset[\delta_n,\varepsilon] f(y, u )
             \right)
             \ge 2\varepsilon K
             .
        \end{equation}
        We show that this implies that         $
        B(y+t_n u, t_n \varepsilon) \cap \Edomf = \emptyset
        $ for every $n\in \N$, 
        which contradicts the non-porosity of $\Edomf$ at $y$ in the direction of $u$.

        Indeed, assuming $z\in B(y+t_n u, t_n \varepsilon) \cap \Edomf$, we get
        $z= y+t_n \hat u \in \Edomf$ for some $\hat u \in B(u,\varepsilon)$,
        so
        $
                (f(z)-f(y))/t_n
            \in
                \hadamderset[\delta_n,\varepsilon] f(y, u )
        $ and
        \begin{equation*}
             \norm{
                \frac{ \extendedf (y+t_n u_n) - \extendedf (y) }{ t_n }
                -
                \frac{ f(z) - f(y) }{ t_n }
             }
             \le K \norm{z - (y + t_n u_n) } / t_n
             = K \norm{\hat u - u_n }
             < 2 \varepsilon K
             .
        \end{equation*}
        This directly contradicts  \eqref{eq:eq71}.
\end{proof}

        The next few lemmata are devoted to the case $\Edomf=Y$.

      \begin{lemma}\label{l:MP5:6-Hadam}
                If $f\fcolon \Edomf= Y \to Z$ is Lipschitz,
                $v \in Y$,
                $\varphi \in \Lip_1(Z)$,
                $\sigma \in \{ \pm1\}$
                and
                \(
                \delta_1,
                                \allowbreak
                \delta_2,
                                \allowbreak
                \delta_3,
                                \allowbreak
                \omega_2,
                                \allowbreak
                \omega_3,
                 \varepsilon
                 \in (0,\infty)
                \) 
                 then the set
                 \[
                 E_{\sigma}=
                 \{
                                 (y, u) \in \EdomfY \times Y
                         \setcolon
                                \derset[\delta_1] f(y,\sigma u + v)
                                \subset
                                \hadamderset[\delta_2,\omega_2] f(y,\sigma u)
                                +
                                B_\varphi(
                                \hadamderset[\delta_3,\omega_3] f(y,v)
                                , \varepsilon
                                )
                 \}
                 \notag
                 \]
                 is Borel.
      \end{lemma}
      \begin{proof}
Recall that $\delta$-derived sets $\derset[\delta]f(\cdot,\cdot)$ are defined in~\eqref{eq:dersetDEF}.

        Note first that it is enough to show that
        $E_{1}$ is Borel as $E_{-1}$ is
        obtained from $E_{1}$ by applying symmetry in the second coordinate
        $(y,u)\mapsto (y,-u)$.

         Let
        \[
          H
          =
                 \left\{
                                 (y, u, t)\in \EdomfY \times Y \times (0,\delta_1)
                         \setcolon
                                \frac{ f(y+t (u+v)) - f(y) }{ t }
                            \in
                                \hadamderset[ \delta_2, \omega_2] f(y,u)
                                +
                                B_\varphi(
                                \hadamderset[ \delta_3, \omega_3] f(y,v)
                                , \varepsilon
                                )
                 \right\}
                 .
         \]
         Then
         $H$ is open.
         Indeed,
         if $(y_0, u_0, t_0)\in H$,
         then there are
         $u_2\in B(u,\omega_2)$,
         $v_3\in B(v,\omega_3)$
         and
         $t_2\in (0,\delta_2)$,
         $t_3\in (0,\delta_3)$
         such that, for $y=y_0$, $u=u_0$, $t=t_0$,
         \begin{equation}\label{eq:obeyphi7}
                      \varphi
                      \left(
                                \frac{ f(y+t (u +v)) - f(y) }{ t }
                              -
                                \frac{ f(y+t_2 u_2) - f(y) }{ t_2 }
                      \right)
                      -
                      \varepsilon
                      <
                      \varphi
                      \left(
                                \frac{ f(y+t_3 v_3) - f(y) }{ t_3 }
                      \right)
         \end{equation}
         By the continuity of $f$ and $\varphi$,
         there is $\eta>0$ such that \eqref{eq:obeyphi7} is
         true for all
         $y \in  B(y_0, \eta)$,
         $u \in  B(u_0, \eta)$,
         $t\in (0,\delta_1) \cap (t_0-\eta, t_0+\eta)$,
         which implies $(y,u,t)\in H$.

         For any
         $n>2/\delta_1$, let $I_n=[1/n, \delta_1 - 1/n]$ and
         \[
           H_n \coloneq \{ (y, u) \in \EdomfY \times Y \setcolon (y, u, t) \in H \text{ for every } t\in I_n \}.
         \]
         Fix
         $n
         >2/\delta_1
         $,
         $y_0 \in \EdomfY$ and $u_0\in Y$.
         If $(y_0, u_0) \in H_n$
         and $t_0\in I_n$
         then there exists
         $\eta=\eta(t_0)>0$
         such that
         \eqref{eq:obeyphi7} is true for every
         $t\in I_n \cap (t_0-\eta, t_0+\eta)$,
         $u \in B(u_0, \eta)$
         and $y \in  B(y_0, \eta)$.
		So if $(y_0,u_0)\in H_n$, 
         then $I_n$
         is covered by finitely many of intervals $(t-\eta(t), t+\eta(t))$, $t\in T$ where $T\subset I_n$ is finite.
         Letting $\bar\eta=\min_{t\in T}\eta(t)$,
         we see that
         \(
              B( y_0, \bar\eta) 
            \times
                B(u_0,\bar\eta)
            \times
                I_n
          \subset
            H
         \).
         Therefore $H_n$
         is an open set and
         $E_1 = \{ (y, u) \in \EdomfY \times Y \setcolon (y, u, t) \in H \text{ for every } t\in (0, \delta_1) \} = \bigcap_{n>2/\delta_1} H_n$
         is a $G_\delta$ set.
      \end{proof}
 The following is a stronger version of \cite[Lemma~5.6]{MP2016} ---
 instead of inscribing 
 a Borel set, we directly show that 
 the set analogous to
 $E_\sigma$ is Borel.
  \begin{corollary}\label{cor:MP5:6-fixedStronger}
            If $f\fcolon \Edomf= Y \to Z$ is Lipschitz,
            $v \in Y$,
            $\varphi \in \Lip_1(Z)$,
            $\sigma \in \{ \pm1 \}$
            and
            $\delta_1, \delta_2, \delta_3,
             \varepsilon
             > 0
             $
             then the set
             \[
             \{
                             (y, u) \in \EdomfY \times Y
                     \setcolon
                            \derset[\delta_1] f(y,\sigma u + v)
                            \subset
                            \derset[\delta_2] f(y,\sigma u)
                            +
                            B_\varphi(
                            \derset[\delta_3] f(y,v)
                            , \varepsilon
                            )
             \}
             \notag
             \]
             is Borel.
  \end{corollary}
  \begin{proof}
        The proof is the same as the proof of Lemma~\ref{l:MP5:6-Hadam},
        with minor adjustments of letting $u_2=u$ and $v_3=v$.
  \end{proof}

      \begin{lemma}\label{l:MP5:6-HDHD}
                If $f\fcolon \Edomf= Y \to Z$ is Lipschitz,
                $v \in Y$,
                $\varphi \in \Lip_1(Z)$,
                $\sigma \in \{ \pm1\}$
                and
                \(
                \kappa,
                                \allowbreak
                \lambda,
                                \allowbreak
                \omega,
                 \tau
                 \in (0,\infty)
                \) 
                 then there is a Borel set $E\subset Y\times Y$ such that
                \( 
                 M _ {\tau}
                 \subset
                 E
                 \subset
                 M _ {\tau+\xomegaone \Lip f}
                \) 
                 where
 \begin{align}\label{eq:setM}
                 M _ \epsilon
                 &=
                 M_\epsilon ( {\sigma,v,\varphi, \omega, \xdeltaone,\xomegaone
                 })\\
                 \notag
                 &=
                 \{
                                 (y, u) \in \EdomfY \times Y
                         \setcolon
                                \hadamderset[\xdeltaone,\xomegaone]
                                                                f(y,\sigma u + v)
                                \subset
                                \hadamderset[\xdeltatwo,\xomegatwo] f(y,\sigma u)
                                +
                                B_\varphi(
                                \hadamderset[\xdeltathree,\omega] f(y,v)
                                , \epsilon
                                )
                 \}
                 .
 \end{align}
      \end{lemma}
      \begin{proof}
      If $f$ is a constant function, the statement is obviously true with $E=Y\times Y$.
      So we can assume $\Lip f>0$.
      Let
      $E$ be the set $E_{\sigma}$ from Lemma~\ref{l:MP5:6-Hadam}
      used with $\delta_1 = \delta_2 = \delta_3 = \kappa$, $\omega_2 = \lambda$ and $\omega_3=\omega$.
      Since \(
      \derset[\xdeltaone]
                        f(y,\sigma u + v)
      \subset
      \hadamderset[\xdeltaone,\xomegaone]
                        f(y,\sigma u + v)
      \)
      by~\eqref{eq:DinHD},
      we have
      \(
      M _ {\tau}
      \subset E
      \).
      Since \(
      \hadamderset[\xdeltaone,\xomegaoneNOINDEX]
                        f(y,\sigma u + v)
      \subset
      \derset[\xdeltaone]
                        f(y,\sigma u + v)
      +
      B(0_Z, \xomegaoneNOINDEX \Lip f)
      \)
      by~\eqref{eq:HDinD},
      we
      obtain with the help of~\eqref{eq:phi:addout} that
      \(
         E \subset
         M _ {\tau+\xomegaoneNOINDEX \Lip f}
      \).
      \end{proof}

\begin{lemma}\label{l:BorelFull}
        Let $f\fcolon Y\to Z$ be a Lipschitz function defined on the whole space $Y$.
        Assume that $Y$ is a separable Banach space and $\Phi \subset \Lip_1(Z)$ is separable.
        Then the graph of
        $\UPhi f$
        is a Borel set.
        In particular, for any $\Edomf_0\subset Y$,  the graph of $\regtanex {\Edomf_0}$ is Borel in $\Edomf_0\times Y$.
        \end{lemma}
\begin{proof}

Assume $f$ is $K$-Lipschitz, $K>0$.
        We want to show that the set
        \[
          A = \{ (y,u) \in Y\times Y \setcolon u\in
          \UPhi f(y)
          \}
        \]
        is Borel.
        Let $\denseY$ and $\densePhi$ be countable dense subsets of $Y$ and $\Phi$, respectively.
        For $\sigma \in \{ \pm 1\}$, $v \in \denseY$,
        $\varphi\in \densePhi$ and $p,q,s\in \N$
        we
        let $\tau(p)=\omega(p)=1/p$, $\kappa(q)=1/q$, $\lambda(s)=1/s$
        and
        use Lemma~\ref{l:MP5:6-HDHD} to find a Borel set
        $E_{\sigma, v, \varphi, p, q, s} \subset Y\times Y$
        such that
        \begin{equation}\label{eq:MsEsM}
                 M_{\tau(p)
                 }
                 (\sigma, v, \varphi, \omega(p)),   \kappa(q), \lambda(s))
                 \subset
                 E_{\sigma, v, \varphi, p, q, s}
                 \subset
                 M_{\tau(p) + \lambda(s) K
                 }(\sigma, v, \varphi,  \omega(p), \kappa(q), \lambda(s))
        \end{equation}
        where
        \(
                 M _ \epsilon(\sigma, v, \varphi,  \omega,  \kappa, \lambda )
        \)
         for $\epsilon, \omega, \kappa, \lambda > 0$
         denotes the set in~\eqref{eq:setM}.
		 Let further $Q=\{1/n\setcolon n\ge1\}$.
         We show 
		first
         that $A=A_1$
         where
 \begin{align}
       \notag
            A_1
       &= A_2 \cap A_3       
       \qquad
        \text{ and}
   \\
       \notag
           A_2
       &=
           \{ (y, u) \in Y\times
           Y
           \setcolon f'(y;u)   \text{ exists} \}
       ,
   \
           A_3
=           \bigcap _ { p = 1 } ^ { \infty }
           \bigcap _ { \varphi \in \densePhi }
           \bigcap _ { v \in \denseY }
           \,
           \,
           \,
           \bigcup _ { r = 1 } ^ { \infty }
           \,
           \,
           \,
           \bigcap _ { q = r } ^ { \infty }
           \bigcap _ { s = r } ^ { \infty }
           \bigcap _ { \sigma \in \{ \pm 1 \} }
                 E_{\sigma, v, \varphi, p, q, s}
                \, .
 \end{align}
By Lemma~\ref{l:UPhibydense}, we have
\[ A=
   A_2 
       \cap 
	\bigcap_{\varepsilon\in Q}
	\bigcap_{\varphi\in\Phi_0}
	\bigcap_{v\in Y_0\setminus \{0\}}\ 
	\bigcup_{\delta_0>0}\ 
	\bigcap_{\delta\in (0,\delta_0)\cap Q}\ 
	\bigcap_{\omega\in (0,\delta_0)\cap Q}\ 
		\bigcap_{\sigma=\pm1}
D _ {\varepsilon, \varphi, v, \delta_0,\delta,\omega,\sigma}\, ,
        \] 
       where 
\begin{multline*}       
       D _ {\varepsilon, \varphi, v, \delta_0,\delta,\omega,\sigma}
       =
   A_2
       \cap 
\{(y,u)\in Y\times Y\setcolon
\hadamderset [\delta,\omega] f(y, \sigma u + v)
           \subset
           f'(y;\sigma u) + B_\varphi(\hadamderset[
                                                      \delta,
                                                      \varepsilon
                                                    ]
                                                  f(y, v), \varepsilon
                                      )
                                                  \}.
\end{multline*}
         Assume $(y,u) \in A$;
         then $f'(y; u)$ exists and,
         by Remark~\ref{r:derinHA},
         \[ 
           f'(y; \sigma u) \in \closure {
                                          \hadamderset
                                                 [1/q, 1/s]
                                          f(y, \sigma u)
                                        }
         \subset
                                                   \hadamderset
                                                 [1/q, 1/s]
                                          f(y, \sigma u)
		 +B(0_Z,\varepsilon)
         \] for every $q, s\in \N$, $\varepsilon>0$ and $\sigma\in\{\pm1\}$.
Therefore, as $(y,u)\in A$, we have that                
         for every 
         $\varepsilon=1/(2p)\in Q$,
         $\varphi \in \densePhi$ and $v\in \denseY$  
         we may find $\delta_0$ such that, for every $\delta, \eta \in (0, \delta_0)\cap Q$ and $\sigma \in \{ \pm 1 \}$, 
\begin{align*}
           \hadamderset[\delta,\eta] f(y, \sigma u + v)
           &\subset
           f'(y;\sigma u) + B_\varphi(\hadamderset[\delta,
                                                      \varepsilon]f(y, v), \varepsilon
                                     )
\\&\subset
           \hadamderset[1/q, 1/s]f(y, \sigma u)
		  + B_\varphi(\hadamderset[\delta,\varepsilon]f(y, v), \varepsilon
                             )                                                  
                                                  +B(0_Z,\varepsilon)
\\&\subset
           \hadamderset
                                                 [1/q, 1/s]
                                          f(y, \sigma u)
		  + B_\varphi(\hadamderset[
                                                      \delta,
                                                      1/p 
                                                    ]
                                                  f(y, v), 1/p
                             )
                                                  ,
\end{align*}
by~\eqref{eq:HDMonotOmega} and~\eqref{eq:phi:addout}.
        Note that for $v=0$ the first line of the above chain of inclusions follows from
        Remark~\ref{r:regtestdir0}.

         Find an integer $r> 1/\delta_0$.
         Then, for every
         $q\ge r$
         and
         $s\ge r$,
         we have
         \[
            (y, u)
            \in 
                 M_{1/p
                 }(\sigma, v, \varphi, 1/p , 1/q, 1/s)
                 \subset
                 E_{\sigma, v, \varphi, p, q, s}
         \, .
         \]
         This shows that $A\subset A_1$.

         Assume now that $(y,u) \in A_1$.
         Then Hadamard derivative $f'(y; u)$ exists.
         Fix $\varepsilon=1/p \in Q$,
         $\varphi \in \densePhi$
         and
         $v\in Y_0\setminus\{0\}$.
         We will show
         that there is $\delta _ 0 > 0$ such that for every $\delta, \omega \in (0, \delta_0)\cap Q$ and $\sigma \in \{ \pm 1 \}$, it holds
         \begin{equation}\label{eq:wanttoprove}
             \hadamderset[\delta,\omega] f(y, \sigma u + v)
             \subset
             f'(y;\sigma u)
             + B_\varphi(\hadamderset[ \delta, \varepsilon] f(y, v), \varepsilon
             )
             .
         \end{equation}
         By our assumption $(y,u) \in A_1$, so there is $r\in \N$ such that
         \(
               (y, u)
               \in 
               E_{\sigma, v,  \varphi, 3p, q, s}
         \)
         for every $q\ge r$, $s \ge r$ and  $\sigma\in\{\pm1\}$. This implies, by~\eqref{eq:MsEsM}, that
         for 
all
         $q,s\ge r$ and  $\sigma\in\{\pm1\}$
         \begin{equation}\label{eq:fromEsM}
         \hadamderset[1/q,1/s] f(y,\sigma u+v)
         \subset
         \hadamderset[1/q,1/s] f(y,\sigma u)
         +
         B_\varphi(\hadamderset[1/q,1/(3p)] f(y,v),1/(3p)
         +K/s
         )         
         ,
         \end{equation}
         as
$f$ is $K$-Lipschitz. 
Note that if $s>3Kp$, then $K/s<1/(3p)$.
         Since Hadamard derivative $f'(y; u)$ exists, we apply Remark~\ref{r:HDsetandder} to conclude that there is $\delta_1>0$ such that 
\begin{equation}\label{eq:derex}
         \hadamderset[\delta,\omega] f(y, \sigma u) 
         \subset  f'(y;\sigma u) + B(0, 1/(3p)), 
\end{equation}
         for every $\delta,\omega \in (0, \delta_1)$ and $\sigma\in\{\pm1\}$.
         Let
        $\delta_0=\min(1/r,1/(3Kp), \delta_1)$.
         Then
for all $\delta,\omega\in(0,\delta_0)\cap Q$, $\sigma\in\{\pm1\}$ and
$\varepsilon=1/p$, \eqref{eq:wanttoprove} follows from~\eqref{eq:fromEsM} and~\eqref{eq:derex}.

  This finishes the proof of $A=A_1$.

   By Lemma~\ref{l:HadamMeas},
   set
   $A_2$
   is Borel
   in $Y\times Y$.
   Set $A_3$ is Borel because
   all
   $E_{\sigma, v, \varphi, p, q, s}$
   are Borel.
   Therefore,
    $A=A_1$ is a Borel set.

        For any $\Edomf_0\subset Y$, the graph of $\regtanex \Edomf_0$ is equal
        to the intersection of $\Edomf_0\times Y$ with the graph of
        $\UPhirplus d_{\Edomf_0}$, see Remark~\ref{r:disttan:eq},
        where $\Phir=\Phir(\R)$ is separable by Remark~\ref{r:ZsepPHIsep}\itemref{r:ZsepPHIsep:item:Ur}.
\end{proof}

Finally, we are ready to show that 
$\EAPhi f$ is Borel measurable, for any Lipschitz partial function~$f$.

\begin{proposition}\label{p:Borel1}
        Let $Y,Z$ be separable Banach spaces, $\Edomf\subset Y$ and
        $f\fcolon \Edomf \to Z$ Lipschitz.
        Let $\Phi\subset\Lip_1(Z)$ be separable. 
        Then
        $\EAPhi f$
        is Borel measurable.
\end{proposition}
\begin{proof}
Following  Definition~\ref{def:assignmentEA}, 
let $\tilde Z$ and $\tilde f\fcolon Y\to\tilde Z$ be as
in~\eqref{eq:extensionBL_codomain} and~\eqref{eq:extensionBL} of the proof of  Lemma~\ref{l:extf}.
Let $Z_0 = \closedSpan f(\Edomf)\subset\tilde Z$,
$\Phi_0 = \Phi|_{Z_0}$ be the collection of restrictions
and let $\tildePhi\subset\Lip_{2}(\tilde Z)$
be as in~\eqref{eq:extensionPhi}
of
the proof of Lemma~\ref{l:extphi}.

Using now that $\frac12\tildePhi$ is separable  and recalling that  $\Phir=\Phir(\R)$ is separable too (see Remark~\ref{r:ZsepPHIsep}\itemref{r:ZsepPHIsep:item:Ur}), we
conclude,
        by Lemma~\ref{l:BorelFull}
        that
        both
        $\Hindexedbyfnspace{\frac12\tildePhi} \extendedf=\Hindexedbyfnspace{\frac12\tildePhi,[Y]} \extendedf$ and
        $\UPhirany d_{\Edomf}$,
        where $d_{\Edomf}(y)=\dist(y,\Edomf)$,
        have Borel measurable graphs.
        Therefore the intersection of the graphs,
\begin{align*}
        D&\coloneq
          \{
(y, u)
\setcolon
y \in Y
,\
u \in
      \Hindexedbyfnspace{\frac12\tildePhi,[Y]}\extendedf (y) 
\cap
\UPhirany d_{\Edomf}(y)
\}
\end{align*}
        is Borel, too. 
        Thus, the intersection of $D$ with $\Edomf \times Y$ is Borel in $\Edomf \times Y$.

        Finally,
        by Remark~\ref{r:disttan:eq},
        $\regtanex_{y} \Edomf =\UPhirany d_{\Edomf} (y)$ for every $y\in\Edomf$
       which finishes the proof that condition~\itemref{def:AM:item:G1} of Definition~\ref{def:assignment-meas} is satisfied for $\EAPhi f (y)$, see Definition~\ref{def:assignmentEA} and Remark~\ref{r:onehalf}.

        Condition~\itemref{def:AM:item:M} of Definition~\ref{def:assignment-meas}
        follows immediately by Remark~\ref{rem:assMeas}\itemref{rem:assMeas:item:PARTiM}.
\end{proof}

\begin{theorem}\label{thm:disttan}
        Let $Y$ be a separable Banach space, $M\subset Y$.
        Then $\regtanex M$  is a regularized tangent.
\end{theorem}
\begin{proof}
Assignment $\regtanex M$ is defined in Definition~\ref{def:assignmentPhi}.
        To prove the statement, we need to verify the four properties listed in Definition~\ref{def:aregTan}.
        For every $y\in M$,
        $ \regtanex_y M $
        is a closed linear subspace of $Y$ by 
        Lemma~\ref{l:linearity} and Lemma~\ref{l:UPhi-closed-cnts}.
        The non-porosity of $M$ at its point $y$ in directions $u\in\regtanex_y M$ follows from 
        Lemma~\ref{l:nporous}\itemref{l:nporous:nporous}.
        The completeness 
        of $\regtanex M$
        follows from Theorem~\ref{t:MP52}.
        Finally, the graph of $\regtanex$ is relatively Borel in $M\times Y$ by Lemma~\ref{l:BorelFull}.
\end{proof}
\begin{corollary}\label{c:restriction}
Let $Y$ be a separable Banach space, $Z$ be a Banach space. Let 
$f\fcolon\Edomf\subset Y\to Z$ be a function, $U$ be a Borel measurable and complete Hadamard domain assignment for $f$. Let $\domainHzeroWasZero\subset\Edomf$ and $f_0$ be the restriction of $f$ to $\domainHzeroWasZero$.

        For $y\in \domainHzeroWasZero$, let $\tilde U(y)=U(y)\cap \regtanex_y
        \domainHzeroWasZero$. Then $\tilde U$ is a Borel measurable and
        complete Hadamard domain assignment for $f_0$.  Additionally, for every
        $y\in \domainHzeroWasZero$ and $u\in \tilde U(y)$  it holds that
        $\domainHzeroWasZero$ is not porous at~$y$ in the direction of~$u$, and
        Hadamard derivatives $f'(y;u)$ and $f_0'(y;u)$ exist and coincide.
\end{corollary}
\begin{proof}
        We refer to Theorem~\ref{thm:disttan} and to Definition~\ref{def:aregTan}\itemref{def:aregTan:nporous}
        for  non-porosity.
        If $y\in \domainHzeroWasZero$ and
        $u\in \tilde U(y)\subset U(y)$
        then $f'(y;u)$
        exists.
        Then also the Hadamard derivative
        $f_0'(y;u)$ of the restriction exists,
        and they are the same,
        cf.\  Remark~\ref{r:thickandunique}.
\end{proof}
\begin{remark}\label{r:disttan:subspace}
	Let $Y$ be a Banach space and $X$ a closed linear subspace of $Y$.
	Then $\regtanex_x X = X$ for every $x\in X$.
	Indeed, the function $d_X = \dist(\cdot, X)$ is
	positively
	$1$-homogeneous
	and satisfies $d_X(x+y)=d_X(y)$ whenever $x\in X$, $y\in Y$.
	In this case,
	one-sided derivative
	$(d_X)'_+(x;y)$
	exists
	and
	$(d_X)'_+(x;y) = \dist(y,X)
	=d_X(y)
	$.
	Therefore 
	$d_X'(x; y)$ exists if and only if $(d_X)'_+(x;y) = - (d_X)'_+(x;-y)$ which is true if and only if 
	$d_X(y)=0$, i.e.\
	$y\in X$.
	If $u\in Y\setminus X$, then  $d_X'(x; u)$ does not exists,
	so $u \notin \UPhirplus d_X(x)$.
	If $u\in X$  then $(d_X)'(x;u)=0$ exists,
	and 
	for any $v\in Y$
	the 
	approximation~\eqref{eq:UPhiRegularity} 
	of the identity $(d_X)'_+(x;u+v)=(d_X)'(x;u)+(d_X)'_+(x;v)$
	by the Hadamard derived sets reveals that $u \in  \UPhirplus d_X(x)$.
	Cf.\ also Remark~\ref{r:disttan:eq}.
\end{remark}

\section{Derivative assignments for pointwise Lipschitz and general functions}
\label{sec:PointLip}

        In this 
        short 
        section we discuss Hadamard differentiability of general functions
        $f\fcolon \Edomf\subset \spaceY\to \spaceZ$, restricting our attention to the set of points
        where $f$ is
        Lipschitz. 
        Since for pointwise Lipschitz functions this set coincides with their domain, we show 
        in Theorem~\ref{t:PointLip} 
		that they possess a Borel measurable complete Hadamard derivative assignment.

        We use the results from the previous sections together with
        ideas and methods from \cite{MZ,BongiornoCMUC, Bongiorno}.
        In particular, an important part is played by the idea of how
        to use the bound obtained
        by differentiability of the distance function from a set $L$
        to deduce differentiability of a function $f$ from
        differentiability of $f|_L$.
        Practically, the idea means to use Lemma~\ref{lem:MZ};
        we however do not look at all differentiability points of the distance
        function but at the point-direction pairs
        determined by
        a derivative assignment for the distance function.

        A description of the technique of decomposition of a pointwise Lipschitz function into a sequence of partial functions
        that are Lipschitz on and around their domains can be found in \cite{MZ, BongiornoCMUC, Federer}.

Recall the notation introduced in Definition~\ref{def:domination-general}: for
sets $D\subset \spaceY\times \spaceY$,
we denote
their sections $D_{\pointy}=\{ u \in \spaceY\setcolon (\pointy, u )\in D\}$.

\begin{theorem}\label{t:PointLip}
        Let $\spaceY$ and $\spaceZ$ be Banach spaces.
        Assume
        that
        $\spaceY$ is separable
        and
        $\spaceZ$ has the Radon-Nikod\'ym property.
        Let $\Edomf \subset \spaceY$
        and let $f\fcolon \Edomf\to \spaceZ$
        be a function.
        Let $\Edomfzero = \domlip(f)\subset\Edomf$
        be the domain of pointwise Lipschitzness of $f$ as defined in~\eqref{eq:dom_Lip}, and let $f_0$ be the restriction of $f$ to  $\Edomfzero$.

        Then there is a set $D$
        which is
        linearly dominating
        in
        $\Edomfzero \times \spaceY$,
        \textup{(}relatively\textup{)} Borel in~$\Edomfzero \times \spaceY$,
        such that Hadamard derivative
        $f'(\pointy;u)$ exists for every $(\pointy,u)\in D$, and depends
        continuously and linearly on $u\in D_\pointy$ for every $\pointy\in
        \Edomfzero$.

        Moreover,  there is a Borel measurable complete Hadamard derivative
        assignment $U$ for $f_0$, such that 
        for every $y\in\Edomfzero$ and
        $u\in U(y)$ the set 
        $\Edomfzero$ is not porous at $y$ in the direction of~$u$, and
        Hadamard derivatives $f'(y;u)$ and $f_0'(y;u)$ exist and
        coincide.

        If $f$ is a pointwise Lipschitz function, i.e.\ $\Edomfzero=\Edomf$,
        then there exists a Borel measurable complete Hadamard derivative assignment for $f$.
\end{theorem}
\begin{proof}
        For each $n\in\N$ consider
        the set
        \(
              L_{n,0} \coloneq
                      \{ \pointy \in \Edomf
                \setcolon
                      \norm{ f(\pointx) - f(\pointy) } \le n\norm{\pointx - \pointy}
                      \text{ for every }\pointx\in \Edomf \cap B(\pointy, 1/n)
                      \}
        \);
        we claim it
        is closed in $\Edomf$.
Let $\haty\in \Edomf\cap (\overline{L_{n,0}}\setminus L_{n,0})$.
   Then there exists
   $x\in \Edomf \cap B(\haty, 1/n)$
   such that
   $\norm{f(x) - f(\haty)} > n \norm{x-\haty}$.
Let $y_k\in L_{n,0}$ ($k\in \N$)
be such that $y_k \to \haty$; then there is $k_0 \in \N$ such that for all $k \ge k_0$,
   we have
   $x\in \Edomf \cap B(y_k, 1/n)$
   and
   $\norm{f(x) - f(y_k)} > n \norm{x-y_k}$,
   which shows $y_k \notin L_{n,0}$, a contradiction.

        Let now
        $L_n=\{y\in L_{n,0}\setcolon \norm{f(y)}\le n\}$. Since $f$ is (locally) 
        $n$-Lipschitz at every $y\in L_{n,0}$, the restriction of $f$ on
        $L_{n,0}$ is continuous, and hence $L_n$ is closed
        in $\Edomf$.

        For each \( \pointy\in \Edomfzero \),
        there exist $r>0$
        and $L>0$
        such that $\norm{ f(\pointx)-f(\pointy) } \le L \norm{ \pointx - \pointy } $
        for all $\pointx\in \Edomf \cap B(\pointy,r)$, so taking $n > \max ( L, 1/r ,\norm{f(\pointy)})$ we get that $\pointy\in L_n$.
        Hence $\Edomfzero = \bigcup_n L_n$
        and
        $\Edomfzero$ is (relatively) Borel in $\Edomf$.

        Let $f_{n} \coloneq f | _ { L_n }$ be the restriction of $f$ to $L_n$.
Note that $f_n$ is $2n^2$-Lipschitz on $L_n$
        and that $\spaceZ_{n} \coloneq \closedSpan f _ {n}(L _ {n})$ is a separable Banach space with the Radon-Nikod\'ym property.
        Let $U_{n} \coloneq \EAPhi f_{n} $ be the Hadamard derivative domain assignment for $f_{n}$
        provided
        by Proposition~\ref{p:UPhiIsAssignment}
        with \(  \Phi=\Phi\ltr=\Phi\ltr(
        \spaceZ_{n}
        )
        \),
        which is separable and admissible by
        Remark~\ref{r:ZsepPHIsep}\itemref{r:ZsepPHIsep:item:Ur}.
        Note that
        $U_{n} (\pointy)$ is a closed linear subspace of $Y$, by Definition~\ref{def:assignment}, cf.\ Proposition~\ref{p:UPhiIsAssignment}. Furthermore,
        $U_{n} (\pointy) \subset \regtanex_\pointy L_n $ by~\eqref{eq:def:EA} of
        Definition~\ref{def:assignmentEA},
        $U_n \fcolon L_n\to2^Y$
        is complete, by Theorem~\ref{t:MP52}, and 		Borel measurable, by Proposition~\ref{p:Borel1}.
        For each $n\ge1$, let 
\begin{equation}
\widetilde U_{n}(\pointy)=
\begin{cases} U_n(\pointy),&\text{if }\pointy\in L_n,\\
\spaceY,&\text{if }\pointy\in \Edomf_0\setminus L_n 
\end{cases}
\end{equation}
and
\[
U(\pointy)=\bigcap_{n\ge1}\widetilde U_{n}(\pointy),
\qquad
\pointy\in \Edomfzero.
\]
        Let also $D= \{ (\pointy,u) \setcolon \pointy\in \Edomfzero,\ u \in U(\pointy) \}$,
        so that $D_\pointy= U(\pointy)$ for every $\pointy\in \Edomfzero$.

        We now show that
         \nopagebreak
\begin{enumerate}[(a)]
\item \label{t:PointLip:item:lin-dom}
                $D$ is linearly dominating
                in
                $\Edomfzero \times \spaceY$,
     \item \label{t:PointLip:item:non-por}
                $\Edomfzero$ is non-porous at~$y$ in the direction of $u$, for every $(y,u)\in D$,
        \item \label{t:PointLip:item:dir-der}
                the
                Hadamard derivative $f'_\HforHadam(\pointy;u)$ exists
                for every $(\pointy,u) \in D$, and $f'_\HforHadam(\pointy;u)$
                depends linearly and continuously on $u\in D_\pointy=U(\pointy)$,
        \item \label{t:PointLip:item:der-f0}
                the
                Hadamard derivative $(f_0)'_\HforHadam(\pointy;u)$ exists and is
                equal to $f'_\HforHadam(\pointy;u)$, for all $(\pointy,u)\in D$,
 \item \label{t:PointLip:item:der-as}
 $U$ is a Borel measurable complete Hadamard derivative assignment for
 $f_0$, see Definitions~\ref{def:assignment},
 \ref{def:completeAssignment} and~\ref{def:assignment-meas}.
                
\end{enumerate}

        Note that for each
        $\pointy\in \Edomfzero$,
        we have that $U(\pointy)$ is a closed linear subspace of~$\spaceY$.
        Fix now
        $p_0\coloneq y \in \Edomfzero=\bigcup_{n\ge1}L_n$ and $u\in U(p_0)$.
        Let
        $n$ be such that $p_0\in L_n$.
        As $p_0\in L_n$ and $u\in U(p_0)\subset\widetilde U_n(p_0)=U_n(p_0)\subset\regtanex_{p_0}L_n$
        we have that $\Edomf_0 \supset L_n$ is non-porous (hence thick) at~$p_0$ in the direction
        of~$u$,
        by Lemma~\ref{l:nporous}\itemref{l:nporous:nporous};
        this 
        verifies~\itemref{t:PointLip:item:non-por} above and also
        Definition~\ref{def:assignment}\itemref{def:assignment:item:notporous}
        for $F_0$ and $U$, which is a part of~\itemref{t:PointLip:item:der-as}.
        Next
        we want to show that $f$ is Hadamard differentiable at $p_0$ in the direction of~$u$, a part of~\itemref{t:PointLip:item:dir-der}.
        From the definition of $L_n\subset L_{n,0}$ we deduce that
        $\norm { f(\pointy) - f(p) } \le n \norm{\pointy - p}$ for every $p\in L_n$ and $\pointy\in \Edomf\cap B(p,1/n)$.
        We also know that $f^*  \coloneq  f_{n}$ is Hadamard differentiable
        at 
        $p_0\in L_n$
        in the direction of~$u\in U(p_0)$.
        We now 
        apply Lemma~\ref{lem:MZ} with
        $P \coloneq L_n$, functions  $f$ defined on $\Edomfzero$ and $f^*$ on
        $P=L_n$, $\varphi \coloneq \dist(\cdot,L_n)$, $K=n$, $\rho=1/n$ and the
        vector $u\in U(p_0)$ fixed earlier.
        Using $u\in U(p_0)\subset \regtanex_{p_0}L_n$ and
        Lemma~\ref{l:nporous}\itemref{l:nporous:der}, cf.\ also
        Definition~\ref{def:aregTan}\itemref{def:aregTan:nporous}, we get that
        $\varphi$ is Hadamard differentiable
        at~$p_0$ in the direction of~$u$.
        Using Lemma~\ref{lem:MZ} we deduce that $f$ is Hadamard differentiable
        at $p_0$ in the direction of $u$ as required
        by Definition~\ref{def:assignment}\itemref{def:assignment:item:hadamDer},
        and that $f'_\HforHadam(p_0;u)=(f_{n})'_\HforHadam(p_0;u)$.
        As $(f_{n})'_\HforHadam(p_0;\cdot)$ is continuous and linear on $U_{n}(p_0)$,
        cf.\ Definition~\ref{def:assignment}\itemref{def:assignment:item:contlin} for $U_n$,
        the same is true for $f'_\HforHadam(p_0;\cdot)$ on $U(p_0) \subset U_{n}(p_0)$, and
        this verifies \itemref{t:PointLip:item:dir-der}.
As $f_0=f|_{\Edomfzero}$, we have~\itemref{t:PointLip:item:der-f0}.
Note that~\itemref{t:PointLip:item:dir-der} and~\itemref{t:PointLip:item:der-f0} verify conditions of Definition~\ref{def:assignment}\itemref{def:assignment:item:hadamDer},\itemref{def:assignment:item:contlin} 
for $f_0$, $F_0$ and $U$.

        It is obvious from Definition~\ref{def:completeAssignment}
        that the extended assignment $\widetilde U_{n} \fcolon \Edomfzero \to 2^\spaceY$ is complete.
        It follows
        that for every separable Banach space $\testspaceX$ and
        every Lipschitz mapping $g\fcolon G\subset \testspaceX\to \spaceY$,
        there is a set $N_{n} \in \mathcal L_1 (\testspaceX)$ such that $g'(\pointx;u)$ belongs to  $\widetilde U_{n}(g(\pointx))$
        whenever $\pointx\in g^{-1}(\Edomfzero)\setminus N_{n}$, $u\in \testspaceX$, and Hadamard derivative $g'(\pointx;u)$ exists.
        In particular, this implies  that
        for every separable Banach space $\testspaceX$ and
        every Lipschitz mapping $g\fcolon G\subset \testspaceX\to \spaceY$
        there is an $\mathcal L_1$-null set $N = \bigcup_{n} N_{n} $ for which we have that
        the derivative  $g'(\pointx;u)$ belongs to  $\bigcap \widetilde U_n(g(\pointx))=U(g(\pointx)) = D_{g(\pointx)}$
        whenever $\pointx\in g^{-1}(\Edomfzero)\setminus N$, $u\in
        \testspaceX$, and $g'(\pointx;u)$ exists.
        Thus $D$ is (linearly) $\mathcal L_1$-dominating in $\Edomfzero \times
        \spaceY$ and the derivative assignment $U$ is complete, see
        Remark~\ref{r:completeDomin}; this proves~\itemref{t:PointLip:item:lin-dom} and the completeness in~\itemref{t:PointLip:item:der-as}.

        Now we show that assignment $U$ is Borel measurable.
        As $U_n$ is Borel measurable,
        the set 
        $S_n\coloneq S$
        from Definition~\ref{def:assignment-meas}
        is Borel 
        in $L_n \times \spaceY$, and therefore in $\Edomfzero\times \spaceY$, as $L_n$ is closed in $\Edomfzero$.
        Hence
        $\{(y,u) \setcolon y\in\Edomfzero, u\in U(y)\} = D$
        is Borel in $\Edomfzero\times \spaceY$ as it is a countable intersection
        of sets $\{(y,u) \setcolon y\in\Edomfzero, u\in\widetilde U_{n}(y)\}
        =S_n\cup 
        (\Edomfzero\setminus L_n)\times \spaceY
        $, 
        each of
        which 
        is Borel
        in $\Edomfzero\times \spaceY$.
        Additionally, we
        note that $(y,u)\mapsto f^\bd(y)(u)=f'_\HforHadam(y;u)$ is Borel measurable on
        $\{(y,u) \setcolon  u\in U(y)\}$
        by Remark~\ref{rem:assMeas}\itemref{rem:assMeas:item:PARTiM}. This finishes the proof of~\itemref{t:PointLip:item:der-as}.
\end{proof}

\begin{remark}\label{r:PointLip}
The statement of Theorem~\ref{t:PointLip} remains true if $\Edomf_0$ is only a subset of $\domlip(f)$. Indeed, letting first $U_1$ be constructed by Theorem~\ref{t:PointLip} for $\Edomf_1=\domlip(f)$ and $f_1=f|_{\Edomf_1}$
and then applying Corollary~\ref{c:restriction} with $(H_0,F,f,U)=(\Edomf_0,\Edomf_1,f_1,U_1)$ gives the required conclusion.
\end{remark}

\section{Chain Rule for Hadamard derivative assignments}
        \label{sec:chain_rule}

        We start with the classical Chain Rule formula for Hadamard derivatives.

        \begin{lemma}[Classical Chain Rule formula for Hadamard derivatives]\label{l:classchain}
                Let $X, Y, Z$ be Banach spaces,
                and $h\fcolon H\subset X \to Y$, $f\colon \Edomf \subset Y \to Z$ functions.
                Assume $x\in h^{-1}(\Edomf)$, $u\in X$.

                \begin{inparaenum}[\textup\bgroup(1)\egroup]
                \item\label{l:classchain:first}
                If
                $v\in Y$ is a Hadamard derivative of $h$ at $x$ in the direction of $u$
                and
                $w\in Z$ is a Hadamard derivative of $f$ at $h(x)$ in the direction of $v$
                then
                $w$ is a Hadamard derivative of $f\circ h$ at $x$ in the direction of $u$.

                \item\label{l:classchain:second}
                If $h^{-1}(\Edomf)$ is thick at $x$ in the direction of $u$
                and $v\coloneq h'_\HforHadam(x;u)$ exists
                then
                $\Edomf$ is thick at $h(x)$ in the direction of $v$.
                If, in addition,
                $f'_\HforHadam(h(x);v)$ exists,
                then $(f\circ h)'_\HforHadam(x; u)$ exists and
                \end{inparaenum}
         \begin{equation}\label{eq:classchain}
                (f\circ h)'_\HforHadam(x; u) =  f'_\HforHadam(h(x); h'_\HforHadam(x;u))
                .
         \end{equation}
        \end{lemma}
        \begin{proof}
                \itemref{l:classchain:first}
                This is simply because whenever $u_n\to u$ and $0\ne t_n\to0$ are
                such that
                \(   x+t_nu_n\in\dom(f\circ h)=
				h^{-1}(\Edomf)\subset H
                \),
                then 
				$\tilde u_n\coloneq \frac{h(x+t_nu_n)-h(x)}{t_n}\to h'(x;u)$
            	and $h(x)+t_n\tilde u_n\in \Edomf $, so that
\[
\frac{(f\circ h)(x+t_nu_n)-(f\circ h)(x)}{t_n}
=\frac{f(h(x)+t_n\tilde u_n)-f(h(x))}{t_n}\to
f'(h(x); h'(x;u)),
\]
see also~\cite{Yamamuro}.

                \itemref{l:classchain:second}
                The set
                $\Edomf \supset h(h^{-1}(\Edomf))$
                is thick 
                at $h(x)$in the direction of $v$
                by Remark~\ref{r:thickandmap}.
                The respective derivatives are uniquely defined by Remark~\ref{r:thickandunique}.
\end{proof}

        We now prove Proposition~\ref{p:chain-orig} and Theorem~\ref{thm:chain:nintro}.
        Recall that the statement of Proposition~\ref{p:chain-orig}, which we restate as Proposition~\ref{p:chain}, involves the notion of the
        composition of Hadamard assignments, see Definition~\ref{def:circ}.
        In this definition the domain of the linear map is defined as the intersection of the natural domain with a regularised  tangent $\regtan_x M$, see Definition~\ref{def:aregTan}. Note that
        existence of a regularized tangent is verified in
        Theorem~\ref{thm:disttan}, where it is shown that $\regtanex M$ is an example of
        a regularized tangent.
         Hence $\regtanex M$ provides a working example
        for the composition of derivative assignments, which we prove now, in Proposition~\ref{p:chain}, to be complete. The latter means that the familiar Chain rule formula for (Hadamard) derivatives is satisfied `almost everywhere', see Theorem~\ref{thm:chain:n} (a restatement of Theorem~\ref{thm:chain:nintro}).

        It is interesting to remark that there could be more than one working example of such composition, if $\regtanex$ is replaced by another valid example of $\regtan$.

\begin{proposition}[Chain Rule for Hadamard derivative assignments, Proposition~\ref{p:chain-orig}]\label{p:chain}
\restateWithoutLPrefix{P:chain}
\end{proposition}

\begin{proof}
Recall that by Definition~\ref{def:circ}, $f^\bd \circ h^\bd$ is defined as
        \begin{align}\label{eq:BDcomposition-copy}
            (f^\bd \circ h^\bd )
            (x)
          =
            f^\bd(h(x)) \circ  h^\bd(x)
\circ \identity_{ \textstyle  \regtan_x \domainHzeroWasOne }\ ,
\qquad x\in \domainHzeroWasOne
            .
        \end{align}

        For $x\in H$ and $y\in \domf $,
        let 
        $U^h(x)$,
        $U^f(y)$
        denote the domains of
        $h^\bd(x)$,
        and
        $ f^\bd(y)$
        respectively,
        that is,
        $U^h$
        and
        $U^f$
        are
        Hadamard
        derivative
        domain
        assignments for $h$ and $f$.
        Consider assignment
        $U^f \circ U^h$ defined for $x\in \domainHzeroWasOne \subset H $ by
        \begin{align}\label{eq:Ucomposition-rep}
            U(x) \coloneq(U^f \circ U^h) (x)
          &=
           \dom
            (
              (f^\bd \circ h^\bd )
              (x)
            )
         \\
         \notag
          &=
                \bigl
                \{
                                u \in U^h(x)
                              \setcolon
                                h^\bd(x) (u)
                                \in U^f(h(x))
                \bigr
                \}
           \cap
                \regtan_x \domainHzeroWasOne
         \\
         \notag
          &=
                \bigl
                \{
                                u \in U^h(x)
                              \setcolon
                                h'(x;u) \in U^f(h(x))
                \bigr
                \}
           \cap
                \regtan_x \domainHzeroWasOne
                ,
        \end{align}
        see~\eqref{eq:Ucomposition} in Remark~\ref{r:circ}.
        Recall that $\domainHzeroWasOne$ is the domain of $f\circ h$, hence $\domainHzeroWasOne= h^{-1}(\domf)\subset H$.

\smallbreak

\itemref{p:chain:deras}
        Fix $x\in \domainHzeroWasOne$ and notice that by the Definition~\ref{def:assignment}, parts~\itemref{def:assignment:item:hadamDer} and~\itemref{def:assignment:item:contlin}, applied to $h$ and $f$,
\begin{equation}\label{eq:chain:inn}
        u\mapsto h'(x; u)
\end{equation}       
        is  a continuous linear map on $U^h(x)$, and 
        the same is true for 
\begin{equation}\label{eq:chain:out}
        v\mapsto f'(h(x);v) 
\end{equation}
        on $U^f(h(x))$.
        Furthermore,
        $U^f(h(x))$,
        $U^h(x)$,
        as well as
        $
                \regtan_x \domainHzeroWasOne
        $,
        are closed linear spaces,
        see Definition~\ref{def:aregTan}\itemref{def:aregTan:cllin}
 and Definition~\ref{def:assignment},
 and $h'(x;\cdot)$ is continuous and  linear on $U^h(x)$ by Definition~\ref{def:assignment}\itemref{def:assignment:item:hadamDer},\itemref{def:assignment:item:contlin}. 
        Hence $U(x)$ in \eqref{eq:Ucomposition-rep} is a closed linear subspace of $X$.

        If $u \in \regtan_x \domainHzeroWasOne$,
        then $\domainHzeroWasOne$ is thick at $x$ in the direction of $u$
        by Definition~\ref{def:aregTan}\itemref{def:aregTan:nporous}
        and Remark~\ref{r:thickandporous};
        this verifies
        Definition~\ref{def:assignment}\itemref{def:assignment:item:notporous} for $U(x)$. 
        By Lemma~\ref{l:classchain},
        we know that for every $u\in U(x)$
        the Hadamard derivative $(f\circ h)'(x; u)$ exists and 
        \begin{align}\label{eq:classchain1}
        (f\circ h)'(x; u) =  f'(h(x); h'(x;u))
        .
        \end{align}
        
        Linearity and continuity of~\eqref{eq:classchain1} follows from linearity and continuity of~\eqref{eq:chain:inn} and~\eqref{eq:chain:out}.
        This
      completes the proof of
        parts~\itemref{def:assignment:item:hadamDer}
        and~\itemref{def:assignment:item:contlin} of
        Definition~\ref{def:assignment} for $U(x)$, so               $f^\bd \circ h^\bd $ is a Hadamard derivative assignment for $f\circ h$.

\smallbreak

\itemref{p:chain:complete}
        Assume that $h^\bd$, $f^\bd$ (and therefore $U^h$ and $U^f$) are complete.
        To check the completeness of~$U$,
        which is sufficient to prove~\itemref{p:chain:complete},
        we
        let $X_0$ be a separable Banach space and $\testfn\fcolon G\subset X_0 \to X$  a Lipschitz function.
        Because $U^h$ is complete for $h$,
        there exists $N_1 \in \mathcal L_1(X_0)$ such that $\testfn'(x; u) \in U^h(\testfn(x))$
        for every $x\in \testfn^{-1}(H)\setminus N_1$
        and every $u\in X_0$ such that Hadamard derivative $\testfn'(x; u)$ exists.
        Obviously,
        $h\circ \testfn
		\fcolon
                g^{-1} (H)
                \subset
                G\subset X_0\to Y
        $
        is a
        pointwise
        Lipschitz function.
        Recall $U^f$ is assumed to be complete (that is, complete with respect to the
        class
        $\tests_L$ of Lipschitz test mappings) for $f$.
        By Remark~\ref{r:testPointwise}, it is also complete with respect to $\tests_{pL}$,
        that means,
        pointwise Lipschitz functions can be used as test mappings for $U^f$
        in Definition~\ref{def:completeAssignment}.
        Hence
        there exists $N_2 \in \mathcal L_1(X_0)$ such that
        $(h\circ \testfn)'(x; u ) \in U^f(h(\testfn(x)))$
        for every $x\in (h\circ \testfn)^{-1}(\Edomf)\setminus N_2=\testfn^{-1}(\domainHzeroWasOne)\setminus N_2$
        and every $u\in X_0$ such that Hadamard derivative $(h\circ \testfn)'(x; u)$ exists.

        By Definition~\ref{def:aregTan}\itemref{def:aregTan:complete},
		since $H_0$ is separable,
        $x\in \domainHzeroWasOne \mapsto \regtan_x \domainHzeroWasOne$
        is a complete
        assignment.
        Therefore
        there exists $N_3 \in \mathcal L_1(X_0)$ such that
        $\testfn'(x; u ) \in \regtan_{\testfn(x)} \domainHzeroWasOne $
        for every $x\in \testfn^{-1}( \domainHzeroWasOne)\setminus N_3$
        and every $u\in X_0$ such that Hadamard derivative $\testfn'(x; u)$ exists.

        Finally, let $N_4$ be the set of points $x\in
        g^{-1}(\domainHzeroWasOne)$ such that there exists $u\in X_0$ with the
        property that $g^{-1}(\domainHzeroWasOne)$ is porous at $x$ in the
        direction of $u$. By Remark~\ref{r:porousAreNull} we have
        $N_4\in\lone(X_0)$.

        Now, define $N= N_1 \cup N_2 \cup N_3\cup N_4\in\mathcal L_1(X_0)$,
        fix
        a point $x_0 \in \testfn^{-1}( \domainHzeroWasOne )\setminus N\subset X_0$
        and a direction vector
        $u\in X_0$ such that
        Hadamard derivative
        $\testfn'(x_0; u)$ exists.
        We claim that 
\begin{equation}\label{eq:chain:complete}
	        \bar u \coloneq \testfn'(x_0; u) \in
            (U^f \circ U^h) (\testfn(x))
        .
\end{equation}        
        Proving this claim is sufficient for completeness of $U$.
        
        If $\bar u=0$, 
		then~\eqref{eq:chain:complete} is satisfied,
        so assume $\bar u\ne0$ which also implies $u\ne0$.
        Note that $\testfn^{-1}(\domainHzeroWasOne)$ is not porous at~$x_0$ in the direction of~$u$ as $x_0\in \testfn^{-1}(\domainHzeroWasOne)\setminus N_4$;
        thus $\testfn^{-1}(\domainHzeroWasOne)$ is thick at~$x_0$ in the direction of~$u$, see Remark~\ref{r:thickandporous}.
        Therefore
        by Remark~\ref{r:thickandunique},
        Hadamard derivatives
        $\testfn'(x_0; u)$ and $(h\circ \testfn)'(x_0; u )$
        are uniquely determined
        if they exist.
        Furthermore,
        as $x_0\in \testfn^{-1}(H_0) \setminus N_1\subset \testfn^{-1}(H) \setminus N_1$,
        we have $\bar u\in U^h(\testfn(x_0))$, and
        thus,
        by Definition~\ref{def:assignment}\itemref{def:assignment:item:hadamDer},
        Hadamard derivative $h'(\testfn(x_0); \bar u)$ exists.
        We may therefore apply
        Lemma~\ref{l:classchain},
        to get that
        $h'(\testfn(x_0); \bar u)=(h\circ \testfn)'(x_0; u )$,
        and the latter,
        since $x_0\notin N_2$,
        belongs to $U^f(h(\testfn(x_0)))$.
        That is, $h'(\testfn(x_0); \bar u) \in U^f(h(\testfn(x_0)))$.
        Recall finally that $x_0\notin N_3$, which implies $\bar u = \testfn'(x_0; u ) \in \regtan_{\testfn(x_0)} \domainHzeroWasOne$. 
        Gathering
        these facts
        together
        we conclude
        $\bar u\in U^h(\testfn(x_0))\cap \regtan_{\testfn(x_0)} \domainHzeroWasOne$ 
        and
        $h'(\testfn(x_0);\bar u)\in U^f(h(\testfn(x_0)))$,
        so that
        from~\eqref{eq:Ucomposition-rep} we obtain
     $\bar u
        =
        \testfn'(x_0; u) \in
            (U^f \circ U^h) (\testfn(x_0))
        =U(\testfn(x_0))
        $, which verifies the 
        claim~\eqref{eq:chain:complete}.

\smallbreak
\smallbreak

       \itemref{p:chain:measurable}
        Let
        \[
        S=\{(x,u)\setcolon x\in \domainHzeroWasOne,\ u\in U(x)\}\subset \domainHzeroWasOne\times X,
        \]
        as in Definition~\ref{def:assignment-meas}.
    Let also
    \begin{align*}
        S^h &=\left\{(x,u)\setcolon x\in H,\ u\in U^h(x)\right\}\subset H\times X,
        \\
        S^f &=
		\left\{
		(y,v)
		\in \Edomf\times Y
		\setcolon
		v        \in U^f(y)
		\right\},
        \\
        S^0 &= \{(x,u)\setcolon x\in \domainHzeroWasOne,\
        u\in \regtan_x \domainHzeroWasOne\}.
    \end{align*}
Note, for future reference, that $S^h$ is Borel in $H\times X$, $S^f$ is Borel in $\Edomf\times Y$ and $S^0$ is Borel in $H_0\times X$. For the first two we use the assumption that $h^\bd$ and $f^\bd$ are Borel, see Definition~\ref{def:assignment-meas}\itemref{def:AM:item:G1}, and for the latter Definition~\ref{def:aregTan}\itemref{def:aregTan:Borel}. Our aim is to check that two conditions of Definition~\ref{def:assignment-meas} are satisfied for
$f\circ h\fcolon\dom(f\circ h)= H_0\to Z$,
$(U,f^\bd\circ h^\bd)$
and
$S$.

        Consider any Borel $B\subset Z$.
        Then
        \begin{equation}
         M_1=
         \{
              (y,v) \in \Edomf \times Y
            \setcolon
            v\in U^f(y),
              f'(y; v)
              \in B
         \}
        \end{equation}
        is  Borel in $\Edomf \times Y$, using both conditions of Definition~\ref{def:assignment-meas}
        for $f^\bd$.
        Recall that 
        $S^0$ is Borel in $H_0\times X$. As
        $S^h$ is Borel in $H\times X$ 
 and $H_0\subset H$, we have that
 $S^h\cap (H_0\times X)$ is also Borel in $H_0\times X$; thus
  $S^h\cap S^0= (S^h\cap (H_0\times X))\cap S^0$ is Borel in $H_0\times X$.   
    
        Let
        \(
                \alpha \fcolon
        S^h \cap S^0
                \to Y\times Y
        \)
        be
        defined by
        \[
        \alpha(x,u)=(h(x), h'(x;u))
        .
        \]
We are going to show now that $\alpha\fcolon S^h \cap S^0
\to Y\times Y$ is a Borel measurable mapping. 
        Note first that $Y_0 \coloneq \closedSpan h(\domainHzeroWasOne)$
        is separable 
        as we assumed $h(\domainHzeroWasOne)$ is.
        Let $(x,u)\in S^h\cap S^0$.
        Then it is clear that $h'(x;u)$ exists and, moreover, 
$u\in \regtan_{x}\domainHzeroWasOne$,
        hence $\domainHzeroWasOne$ is thick at $x$ in the direction of $u$, by Definition~\ref{def:aregTan}\itemref{def:aregTan:nporous} and Remark~\ref{r:thickandporous}. Thus $h'(x;u)=h_0'(x;u)\in Y_0$, where $h_0=h|_{H_0}$,
        see Remark~\ref{r:thickandunique}.
        Therefore,
        $\alpha(
        S^h \cap S^0)\subset 
        Y_0\times Y_0$ and
        $\alpha\fcolon 
    S^h \cap S^0
        \to Y_0\times Y_0$ is Borel measurable by Lemma~\ref{l:pair0}, using, in particular, that $h_0$ and $h^\bd$ are  Borel measurable. This implies that $\alpha\fcolon S^h \cap S^0
        \to Y\times Y$ is a Borel measurable mapping, as $Y_0\times Y_0\subset Y\times Y$. For future reference we note that
\begin{equation}\label{eq:alphaMeas}
\alpha\fcolon S^h \cap S^0
\to \Edomf\times Y
\qquad	\text{is a Borel measurable mapping}, 
\end{equation}
since whenever $(x,u)\in S^0$ we have $x\in H_0$ and thus $h(x)\in \Edomf$.
As $S\subset S^h\cap S^0$, we deduce from~\eqref{eq:alphaMeas} that
$\beta\coloneq \alpha|_{S}\fcolon S\to\Edomf\times Y$ is Borel measurable.
        
         Thus
        \begin{equation}
         M_2\coloneq
\beta^{-1}(M_1)=
         \{
              (x,u) \in S 
            \setcolon
               \beta(x,u)
              \in M_1
         \}
=
\{
(x,u) \in S 
\setcolon
\alpha(x,u)
\in M_1
\}        \end{equation}
        is Borel in~\( S \).
        Using~
        \eqref{eq:classchain1}
        for the second equality, and definition of $S$ and $U(x)$ in~\eqref{eq:Ucomposition-rep} in the last equality,
        we see that
\begin{align*}
	M_2
	&=	
	\{
	(x,u)\in S
	\setcolon 
	h'(x;u)\in U^f(h(x)), f'(h(x);h'(x;u))\in B  
	\}\\    	
	&=
	\{
	(x,u) \in S
	\setcolon
	h'(x;u)\in U^f(h(x)), 
	(f\circ h)'(x; u) \in B
	\}
	\\
	&=
    \{
    (x,u) \in S
    \setcolon
    (f\circ h)'(x; u) \in B
    \}
\end{align*}
is Borel in $S$.
As $B$ was an arbitrary Borel subset of $Z$, it follows that $(U^f \circ U^h, f^\bd \circ h^\bd )$ satisfies
        condition \itemref{def:AM:item:M}
        of Definition~\ref{def:assignment-meas} for the mapping $f\circ h\fcolon H_0\to Z$.

      What remains to prove is condition~\itemref{def:AM:item:G1} of Definition~\ref{def:assignment-meas}, i.e.\ the Borel measurability of $S$ in $\dom(f\circ h)\times X=\domainHzeroWasOne \times X$.
This follows from~\eqref{eq:alphaMeas} and
        \begin{equation*}
          S=
          S^h\cap S^0\cap\alpha^{-1}(S^f)
          =
          \dom(\alpha)\cap\alpha^{-1}(S^f)=
          \alpha^{-1}(S^f), 
          \end{equation*} 
          since $S^f$ is Borel in $\Edomf\times Y$.
\end{proof}

        We are ready to prove Theorem~\ref{thm:chain:nintro}
        and Corollary~\ref{c:domindom-orig}
        from the Introduction.

\begin{theorem}[Theorem~\ref{thm:chain:nintro}]
\label{thm:chain:n}
\restateWithLPrefix{T:chain}{Ncopy}
\end{theorem}
\begin{proof}
Let $h_0$ be the restriction of $h$ to $\domainHzeroWasZero$, and $f_0$ be the restriction of $f$ to $\domlip(f)$.
	By Theorem~\ref{t:PointLip}, as $Y$ is separable\footnote{\label{f:sepY}Note that this is the only place in the proof of this theorem, where we use the assumption that the space $Y$ is separable.},
	there exists
	a
	Borel measurable complete derivative assignment
	$(f_0^\bd, U^{f_0})$ 	for $f_0$ such that
in addition, for every $y\in\dom(f_0)=\domlip(f)$ and $v\in U^{f_0}(y)$  it holds that 
Hadamard derivatives $f'(y;v)$ and $f_0'(y;v)$ exist and coincide.
Likewise, by Theorem~\ref{t:PointLip} and Corollary~\ref{c:restriction}, see also Remark~\ref{r:PointLip}, as $X$ is separable,
there exists
a
Borel measurable complete derivative assignment
	$(h_0^\bd, U^{h_0})$ 	for $h_0$ such that
for every 
$x\in \domainHzeroWasZero \subset\domlip(h) $ 
and $u\in U^{h_0}(x)$  it holds that
Hadamard derivatives $h'(x;u)$ and $h_0'(x;u)$ exist and coincide.
       Since $h_0\fcolon H_0\to Y$ is clearly continuous, it is  Borel measurable.
	By Proposition~\ref{p:chain},
	$(f_0^\bd \circ h_0^\bd, U^{f_0} \circ U^{h_0})$ is a Borel measurable complete derivative assignment for $f_0\circ h_0$, using that $X$ is separable and that $h_0$ is pointwise Lipschitz on $H_0$.

Note that $\dom(f_0\circ h_0)=h_0^{-1}(\dom f_0)=
h_0^{-1}(\domlip f)=\domainHzeroWasZero$.
	Then
	\( D \coloneq
           \{ (x,u) \setcolon x\in \dom(f_0\circ h_0),\ u \in (U^{f_0} \circ U^{h_0})(x) \}
	\)
	is a relatively Borel and linearly $\mathcal L_1$-dominating in $\dom (f_0\circ h_0) \times X=\domainHzeroWasZero\times X$,
        see
        Definition~\ref{def:assignment-meas}\itemref{def:AM:item:G1} and
         Remark~\ref{r:completeDomin}.

If $(x,u)\in D$, then $u\in (U^{f_0}\circ U^{h_0})(x)$, which implies $u\in U^{h_0}(x)$ and
$v\coloneq h_0'(x;u)\in U^{f_0}(h_0(x))$. 
Let now $x\in\domainHzeroWasZero$.
Using in addition that 
$x\in\dom(f_0\circ h_0) =\dom(h_0)$, i.e.\ $h(x)=h_0(x)$,
we get that
Hadamard derivatives $h'(x;u)$ and $f'(h(x);v)$ exist and coincide with the corresponding derivatives of $h_0$ and $f_0$
for $u\in U^{h_0}(x)$, $v\in U^{f_0}(h(x))$.
The latter derivatives depend continuously and linearly on their directional variables $u$ and $v$,
by Definition~\ref{def:assignment}\itemref{def:assignment:item:contlin}, see also Theorem~\ref{t:PointLip}.
Applying Lemma~\ref{l:classchain}\itemref{l:classchain:first} completes the proof.
\end{proof}

\bigbreak

\begin{corollary}[Corollary~\ref{c:domindom-orig}]\label{c:domindom-copy}
        \restateWithLPrefix{rst:domindom}{copy}
\end{corollary}
\begin{proof}
       Let $A = \setS  \cap T$ where $T
        \subset F\times Y
       $ is
       obtained by Corollary~\ref{c:setDomNpor} for $M=F$.
       Then $A$ is a relatively Borel linearly $\lone$-dominating subset of $F\times Y$.
       Let
       $E_Y
       \coloneq 
       F
		\subset Y
       $, 
       $Z\coloneq\R$
       and, for $y\in F$, let $f(y)=f_0(y)=c \in \R$
       be a constant function. Then  $0\in\R$ is a Hadamard derivative of $f_0$ at any $y\in\Edomf$ in all directions $u\in Y$.
       In particular, for every $(y,u)\in A\subset T$, the set $\Edomf$ is thick at $y$ in the direction $u$, by Remark~\ref{r:thickandporous}, and so the Hadamard derivative $(f_0)'_\HforHadam(y;u)=0$ exists.
       Let $U^{f_0}$ be the assignment defined by 
       $U^{f_0}(y) = A_y$,
       $y\in F$.
       Then $U^{f_0}$ is a complete assignment, see
       Remark~\ref{r:completeDomin}.
       Also, $U^{f_0}$ is a
       complete Borel measurable derivative domain assignment for $f_0$,
       cf.\ Definition~\ref{def:assignment} and Definition~\ref{def:assignment-meas} with $\{(y,u)\setcolon y\in\Edomf,u\in U(y)\}=A$.
       Then we can proceed as in the proof of Theorem~\ref{thm:chain:n},
       see also Footnote~\ref{f:sepY},
       to obtain a set $D$ that is relatively Borel and linearly $\lone$-dominating in $H_0\times X$
       such that,
       in particular,
       $h'_\HforHadam(x;u) \in U^{f_0}(h(x))
        =A_{h(x)}
       $ and 
        thus
       $(h(x), h'_\HforHadam(x;u)) \in A \subset \setS $
       whenever $(x,u) \in D$.
\end{proof}

\bigbreak

        We now define composition of $n$ derivative assignments, for each $n\ge2$.
        This extends Definition~\ref{def:circ} of composition of $n=2$ derivative assignments.

\begin{definition}[Multiple composition]\label{def:comp-n}
        For $n\ge2$ let
        $Y_i$, $i=1,\dots,n+1$, be Banach spaces, $\domf_i\subset Y_i$,
        and let $\fun_i \fcolon \domf_i \to Y_{i+1}$, $i=1,\dots,n$,
        be arbitrary functions.
        Let $\domainHzeroWasOne = \dom (\fun_n \circ \dots \circ \fun_1) = \fun_1^{-1}(\fun_2^{-1}(\dots(\fun_{n-1}^{-1}(\domf_n))\dots))$.

        If
        $\fun_i^\bd$
        is a
        Hadamard derivative assignment
        for $\fun_i$, for $i=1\dots,n$,
        then we define
        $\fun_n^\bd \circ \dots \circ \fun_1^\bd $
        by
        \begin{equation}\label{eq:BDcompositionN}
            (\fun_n^\bd \circ \dots \circ \fun_1^\bd)
            (y)
          =
            \fun_n^\bd(y_n) \circ \dots \circ \fun_1^\bd(y_1)
          \vphantom{\sum}
\circ \identity_{R_n(y)\cap\dots\cap R_2(y)}\ ,
            \qquad y\in \domainHzeroWasOne
            ,
        \end{equation}
        where
        \( 
           y_1=y
           \) and \(
           y_i=\fun_{i-1}(y_{i-1})
        \)
and for all $2\le i\le n$
\begin{equation}\label{eq:comp-n:Ri}
	R_i(y)=\regtan_y \dom(f_{i}\circ\dots\circ f_1) 
	.
\eodhere
\end{equation}
\end{definition}
\begin{remark}\label{r:circIter}
        Similarly to Remark~\ref{r:circ}, 
        the definition \eqref{eq:BDcompositionN}
        of
        \(\fun_n^\bd \circ \dots \circ \fun_1^\bd\)
        depends also on $\fun_{n-1}$, \dots, $\fun_1$, which we suppress in the notation.
\end{remark}
\begin{remark}\label{r:chainIter}
        For $n=2$, Definition~\ref{def:comp-n} coincides with that given by Definition~\ref{def:circ}.
        The only difference
is that we have now already justified the existence of a regularized tangent, see Theorem~\ref{thm:disttan}.
\end{remark}

\begin{remark}
The domain assignment corresponding to the composition of~\eqref{eq:BDcompositionN} is defined as
\begin{align*}
U(y)
&=
\dom(\fun_n^\bd(y_n) \circ \dots \circ \fun_1^\bd(y_1))\cap
\bigcap_{i=2}^n R_i(y)\\
&=
(\fun_1^\bd(y_1))^{-1}\biggl(
(\fun_2^\bd(y_2))^{-1}\Bigl(
\dots
\bigl(
(\fun_{n-1}^\bd(y_{n-1}))^{-1} 
(\dom(\fun_n^\bd(y_n))
)
\bigr)\dots\Bigr)\biggr)
\cap
\bigcap_{i=2}^n R_i(y)
,
\end{align*}
where         \( 
y_1=y
\), \(
y_i=\fun_{i-1}(y_{i-1})
\)
and $R_i(y)$ are defined by~\eqref{eq:comp-n:Ri}.
See also Remark~\ref{r:circ}.
\end{remark}

The following statement generalises Proposition~\ref{p:chain} to the case of composition of $n$ functions.

\begin{proposition}[Iterated Chain Rule]\label{p:chainIter}
        Let $n\ge 2$.
        Assume that
        $Y_i$, $i=1,\dots,n+1$, are Banach spaces, $\domf_i\subset Y_i$,
        and that $\fun_i \fcolon \domf_i \to Y_{i+1}$, $i=1,\dots,n$,
        are functions.

        Let
        $\fun_i^\bd$
        be a
         Hadamard derivative assignment
        for $\fun_i$, for $i=1\dots,n$,
        and let
        $\fun_n^\bd \circ \dots \circ \fun_1^\bd$ be defined by~\eqref{eq:BDcompositionN}.

Then

\begin{inparaenum}[\textup\bgroup(1)\egroup]
\item \label{p:chainIter:deras}
$\fun_n^\bd \circ \dots \circ \fun_1^\bd$ is a Hadamard derivative assignment
        for
        $\fun_n \circ \dots \circ \fun_1$
and
          \begin{align}\label{eq:iteriter}
                  \fun_n ^\bd \circ (\fun_{n-1} ^\bd \circ \dots \circ \fun_1^\bd)
                  =
                  \fun_n^\bd \circ \dots \circ \fun_1^\bd.
          \end{align}

\item\label{p:chainIter:complete}
If 
$Y_1$ is separable, $f_1,\dots,f_{n-1}$ are pointwise Lipschitz, and
all $f_1^\bd,\dots,f_{n}^\bd$ are complete, then 
$\fun_n^\bd \circ \dots \circ \fun_1^\bd $ is         a complete Hadamard derivative assignment
        for
        $\fun_n \circ \dots \circ \fun_1$.

\item\label{p:chain:meas}
If 
$Y_1,\dots,Y_n$ are separable, the functions
$f_1,\dots,f_{n-1}$ are Borel measurable, and
all assignments $f_1^\bd,\dots,f_{n}^\bd$ are Borel measurable too, then 
$\fun_n^\bd \circ \dots \circ \fun_1^\bd $ is 
Borel measurable.
\end{inparaenum}
\end{proposition}
\begin{proof}
We first check~\itemref{p:chainIter:deras} for all $n\ge2$.
If $n=2$, then the statement follows from Proposition~\ref{p:chain}\itemref{p:chain:deras} and Remark~\ref{r:chainIter}.
          Assume $n\ge3$.

          As explained in Remark~\ref{r:circIter}, the inner functions are used in the definition of composition
          but suppressed by the notation.
          In the case of
          the leftmost composition 
          $\circ$ in \eqref{eq:iteriter},
          the inner function is function
          \(
              h \coloneq
              \fun_{n-1}  \circ \dots \circ \fun_1
          \).
          Let
          $h^\bd = \fun_{n-1}^\bd  \circ \dots \circ \fun_1^\bd$.
          Let
          $
                  \domainHzeroWasOne
          \coloneq 
                  \dom( \fun_n\circ \dots \circ\fun_1 ) 
          =
                  \dom( \fun_n\circ h ) 
          $.
          Let          $
             y_1
             =
             y
             \in
             \dom( \fun_n\circ h )
          $,
          $y_i=f_{i-1}(y_{i-1})$, $2\le i\le n$,
          so that 
          $y_n=
			f_{n-1}(y_{n-1})=
			(f_{n-1}\circ\dots\circ f_1)(y_1)=
          h(y)$.
         Applying Definition~\ref{def:circ} to Hadamard derivative assignments  $f_n^\bd$ and $h^\bd$, followed by Definition~\ref{def:comp-n} for $h^\bd$,
          we obtain
          \begin{align*}
                    ( \fun_n^\bd \circ h ^\bd ) (y)
                  &=
                        \fun_n^\bd (y_n)
                    \circ
                        \left (
                                     \fun_{n-1}^\bd (y_{n-1})
                                \circ
                                     \dots
                                \circ
                                     \fun_1^\bd (y_1)
                                \circ
                                     \identity _ { \bigcap_{i=2}^{n-1}R_i(y)}
                        \right ) 
                    \circ
                                  \identity _ { \regtan_y \dom( \fun_n\circ h ) }
               \\
                  &=
                    \left(
                                     \fun_n^\bd (y_n)
                                \circ
                                     \fun_{n-1}^\bd (y_{n-1})
                                \circ
                                     \dots
                                \circ
                                     \fun_1^\bd (y_1)
                    \right)
                                \circ
                        \left (
                               \identity _ { \bigcap_{i=2}^{n-1}R_i(y)}
                            \circ
                                  \identity _ { R_n(y) }
                        \right ) 
               \\
                  &=
                    \left(
                                     \fun_n^\bd (y_n)
                                \circ
                                     \fun_{n-1}^\bd (y_{n-1})
                                \circ
                                     \dots
                                \circ
                                     \fun_1^\bd (y_1)
                    \right)
                    \circ
                                  \identity _ { \bigcap_{i=2}^{n}R_i(y)}
               \\
                  &=
                    \left(
                                     \fun_n^\bd
                                \circ
                                     \fun_{n-1}^\bd
                                \circ
                                     \dots
                                \circ
                                     \fun_1^\bd
                    \right)
                    (y)
					,
          \end{align*}
where $R_i(y)$ are defined by~\eqref{eq:comp-n:Ri}.         
		This proves~\eqref{eq:iteriter}.
          Finally,
        $\fun_n^\bd\circ\dots\circ\fun_1^\bd$ is a
        Hadamard derivative assignment
        for
        $\fun_n \circ \dots \circ \fun_1$
        as
        can
        be shown
        by induction
        via~\eqref{eq:iteriter}
        using
        Proposition~\ref{p:chain}\itemref{p:chain:deras}.

We now turn our attention to
\itemref{p:chainIter:complete} and \itemref{p:chain:meas}.  
If $n=2$, then the statements  
follow from Proposition~\ref{p:chain}~\itemref{p:chain:complete} and~\itemref{p:chain:measurable}, and Remark~\ref{r:chainIter}.              
The case $n\ge3$ follows by induction 
from~\eqref{eq:iteriter}.
We note that the assumption of separability of $Y_1$ in~\itemref{p:chainIter:complete} guarantees that all $\dom(f_i\circ\dots\circ f_1)$, $2\le i\le n$, are separable.  These sets will play the role of $H_0$ in Proposition~\ref{p:chain}\itemref{p:chainIter:complete}. Similarly, the assumption that all $Y_1,\dots,Y_n$ are separable in~\itemref{p:chain:meas} guarantees that the assumptions of Proposition~\ref{p:chain}\itemref{p:chain:measurable} are satisfied when used with $h=f_{i}\circ\dots\circ f_1$, for all $2\le i\le n-1$.
\end{proof}
\begin{remark}
We note that in Proposition~\ref{p:chainIter} the separability assumptions in~\itemref{p:chainIter:complete} and~\itemref{p:chain:meas} are somewhat excessive. Below we list sufficient assumptions.

For~\itemref{p:chainIter:complete}:  
The set $\dom(f_2\circ f_1)\subset Y_1$ is separable.

For~\itemref{p:chain:meas}:
The set $\dom(f_2\circ f_1)\subset Y_1$ and all sets $(f_{i-1}\circ\dots\circ f_1)(\dom(f_i\circ\dots\circ f_1))\subset Y_i$ are separable, for all $2\le i\le n$.	
\end{remark}	
        Our final result generalises Theorem~\ref{thm:chain:nintro} to the case
        of composition of an arbitrary number of
        functions.
        For the sake of simplicity of notation we restrict the statement of
        Theorem~\ref{thm:nchain:iter} to the case of pointwise Lipschitz
        functions $\fung_i$. A formal generalisation of
        Theorem~\ref{thm:chain:nintro} would see us considering general partial
        functions $g_i$ and replacing the set $\dom (\fung_n \circ \dots \circ
        \fung_1 ) \times \spaceXX_1$ with a set which takes into account
        the domains of pointwise Lipschitzness of the functions $\fung_i$, in the spirit
        of the definition of $\domainHzeroWasZero$ from the statement
        of
        Theorem~\ref{thm:chain:nintro}.

\begin{theorem}\label{thm:nchain:iter}
        Let $n\ge 2$.
        Assume that
        $\spaceXX_1,\dots \spaceXX_{n+1}$
        are Banach spaces,
        $\spaceXX_1$ is separable
        and
        $\spaceXX_2,\dots,\spaceXX_{n+1}$ have the Radon-Nikod\'ym property.
        For 
        each
        $i=1,\dots,n$, let $\domf_i$ be a subset of $\spaceXX_i$
        and
        $\fung_i \fcolon \domf_i \to \spaceXX_{i+1}$
        a pointwise Lipschitz function.

        Then there exists a relatively Borel subset $D$ of $\dom (\fung_n \circ \dots \circ \fung_1 ) \times \spaceXX_1$
        which is linearly $\mathcal L_1$-dominating in  $\dom (\fung_n \circ \dots \circ \fung_1 ) \times \spaceXX_1$
        and, for all $(x_1, u_1) \in D$,
        $\fung_n \circ \dots \circ \fung_1$ is Hadamard differentiable at~$x_1$ in the direction of~$u_1$
        and
        \[
                ( \fung_n \circ \dots \circ \fung_1 )'_\HforHadam (x_1; u_1) = u_{n+1}
                ,
        \]
        where $u_{n+1}$ can be calculated using the relations
        \begin{align*}
                  u_{i+1} &= (\fung_i)'_\HforHadam (x_i; u_i),
                  \\
                  x_{i+1} &= \fung_i(x_i)
        \end{align*}
        for $i=1,\dots,n$.
        For every $x\in \dom (\fung_n \circ \dots \circ \fung_1 ) $, the dependence
        on $u_1\in D_{x_1}$ is linear and continuous.
\end{theorem}
\begin{proof}
        Let $\spaceYY_1=\spaceXX_1$, $\funf_1=\fung_1$ and,
        for $i=1,\dots,n$, let $\spaceYY_{i+1}=\closedSpan \fung_i(\Edomf_i\cap \spaceYY_i)$;
        for $i=2,\dots,n$, let $\funf_{i}$ be the restriction of $\fung_i$ to $\Edomf_i\cap \spaceYY_i$.
        Then $\fung_n \circ \dots \circ \fung_1 = \funf_n \circ \dots \circ \funf_1$
        and all $\spaceYY_i$ are separable.

        For every $k=1,\dots,n-1$,
        $\funf_{k+1}$ and $\funf_{k}\circ\dots\circ \funf_1$ are pointwise Lipschitz functions.
        By Theorem~\ref{thm:chain:nintro}, there is a set $D_k$, relatively Borel
        and linearly dominating in
        \( \dom(\funf_{k+1}\circ\dots\circ \funf_1)  \times \spaceYY_1
        \supset
        \dom(\funf_{n}\circ\dots\circ \funf_1)  \times \spaceYY_1
        \) such that, for every $(x_1, u_1)$ in $D_k$,
        \[
            (\funf_{k+1}\circ\dots\circ \funf_1)'_\HforHadam (x_1; u_1)
            =
            (\funf_k)'_\HforHadam (y; v)
        \]
        where $v=(\funf_k\circ\dots\circ \funf_1)'_\HforHadam (x_1; u_1)$
        and $y=(\funf_k\circ\dots\circ \funf_1)(x_1)$.
        Then $D=\bigcap_{k=1,\dots,n} D_k$ is a relatively Borel set
        linearly dominating in
        $ \dom(\funf_{n}\circ\dots\circ \funf_1)  \times \spaceYY_1 $
        with the required property.
\end{proof}


\begin{thebibliography}{77}
\bibitem{AM}
        G.~Alberti, A.~Marchese,
        \bibpaper{On the differentiability of Lipschitz functions with respect to measures in the Euclidean space},
        Geom. Funct. Anal. 26 (2016), no. 1, 1–66.

\bibitem{AD}
        L.~Ambrosio and G.~Dal Maso,
        \bibpaper{A general chain rule for distributional derivatives},
        Proc. Amer. Math. Soc.,
        108:3 (1990) 691--702.

\bibitem{BL}
        Y.~Benyamini, J.~Lindenstrauss,
        \bibbook{Geometric nonlinear functional analysis.} Vol.1,
        Colloquium Publications (AMS), volume 48,
        American Mathematical Society 2000,
        ISBN 0-8218-0835-4.


\bibitem{BongiornoCMUC}
        D. Bongiorno,
        \bibpaper{Stepanoff's theorem in separable Banach spaces},
        Comment. Math. Univ. Carolin. 39 (1998) 323--335.

\bibitem{Bongiorno}
        D.~Bongiorno,
        \bibpaper{Radon-Nikod\'ym property of the range of Lipschitz extensions},
        Atti Sem. Mat. Fis. Univ. Modena 48 (2000) 517--525.

\bibitem{B}
		J.~Borwein,  A.~Lewis,
		\bibbook{Convex analysis and nonlinear optimization.
		Theory and examples}, second edition. CMS Books in Mathematics/Ouvrages de Mathématiques de la SMC, 3. Springer, New York, 2006

\bibitem{Duda}
        J.~Duda,
        \bibpaper{Metric and $w^*$-differentiability of pointwise Lipschitz mappings},
        Z. Anal. Anwend.~
        26 (2007) 341--362.



\bibitem{Federer}
        H. Federer,
        \bibbook{Geometric measure theory}, Springer (1969).

\bibitem{Kechris}
        A.~S.~Kechris,
        \bibbook{Classical descriptive set theory}, 1995.


\bibitem{KK}
        M. Koc, J. Kol\'a\v{r},
        \bibpaper{Extensions of vector-valued functions with preservation of derivatives},
        J. Math. Anal. Appl. 449 (2017) 343--367.
        \href {http://doi.org/10.1016/j.jmaa.2016.11.080}{DOI 10.1016/j.jmaa.2016.11.080}.

\bibitem{Kuratowski}
        K. Kuratowski,
        \bibbook{Topology}, Vol. I, PWN and Academic Press, Warzsawa 1966

\bibitem{LPT}
J. Lindenstrauss, D. Preiss, J. Ti\v ser,
\bibbook{Fr\'echet differentiability of Lipschitz functions and porous sets in Banach spaces},
Annals of Mathematics Studies 179,
Princeton University Press, 2012.


\bibitem{MP2016}
        O.~Maleva and D.~Preiss,
        \bibpaper{Directional upper derivatives and the chain rule formula for locally Lipschitz functions on Banach spaces},
        Trans. Amer. Math. Soc. 368 (2016) 4685--4730.


\bibitem{MZ}
        J.~Mal\'y and L.~Zaj\'i\v{c}ek,
        \bibpaper{On Stepanov type differentiability theorems},
        Acta Math. Hungar. 145 (2015) 174--190.

\bibitem{Penot}
        J. Penot,
        \bibbook{Calculus Without Derivatives}, Springer 2013.
        DOI 10.1007/978-1-4614-4538-8


\bibitem{Preiss2014}
        D.~Preiss,
        \bibpaper{G\^ateaux differentiability of cone-monotone and pointwise Lipschitz functions},
        Israel J. Math. 203 (2014) 501--534.
        DOI: 10.1007/s11856-014-1119-7.


\bibitem{PZ2001}
        D. Preiss, L. Zaj\'\i\v{c}ek,
        \bibpaper{Directional derivatives of Lipschitz functions},
        Israel J. Math. 125 (2001) 1--27.



\bibitem{Yamamuro}
        S. Yamamuro,
        \bibbook{Differential Calculus in Topological Linear Spaces},
        Lecture Notes in Mathematics 374,
        Springer 1974,
        ISBN 3-540-06709-4,
        ISBN 0-387-06709-4.


\bibitem{Z-PAMS}
        L. Zaj\'i\v{c}ek,
        \bibpaper{Hadamard differentiability via G\^ateaux differentiability},
        Proc. Amer. Math. Soc. 143:1 (2015) 279--288.

\bibitem{Zip}
		M.~Zippin, Extension of bounded linear operators, 1703--1741, 
		\emph{in:}
        W.B. Johnson, J. Lindenstrauss,
        \bibinbook{Handbook of the geometry of Banach spaces}, Volume 2,
        Elsevier 2003,
        ISBN 0444513051.



\end{thebibliography}
\end{document}